\documentclass[10pt,twoside,a4paper,reqno]{amsart}

\usepackage{amsfonts,amsthm,latexsym,mathtools,amsmath,amssymb,enumitem,appendix,color,hyperref,leftindex}
\usepackage[capitalize]{cleveref}
\usepackage{float}
\usepackage{comment}

\usepackage{todonotes}
\presetkeys%
    {todonotes}%
    {inline,backgroundcolor=yellow}{}

\usepackage{graphicx}
\DeclareGraphicsExtensions{.pdf,.png}
\makeatletter
\@namedef{subjclassname@2020}{\textup{2020} Mathematics Subject Classification}
\makeatother

\theoremstyle{plain}
\newtheorem{theo+}{Theorem}
\numberwithin{theo+}{section}
\newtheorem{prop+}[theo+]{Proposition}
\newtheorem{coro+}[theo+]{Corollary}
\newtheorem{lemm+}[theo+]{Lemma}

\theoremstyle{definition}
\newtheorem{defi+}[theo+]{Definition}
\newtheorem{problem}[theo+]{Problem}

\newtheorem*{pb-prime}{Problem 1$\mathbb{'}$}

\newtheorem{not+}[theo+]{Notation}

\theoremstyle{remark}
\newtheorem{rema+}[theo+]{Remark}

\newenvironment{theorem}{\begin{theo+}}{\end{theo+}}
\newenvironment{proposition}{\begin{prop+}}{\end{prop+}}
\newenvironment{corollary}{\begin{coro+}}{\end{coro+}}
\newenvironment{lemma}{\begin{lemm+}}{\end{lemm+}}
\newenvironment{remark}{\begin{rema+}}{\end{rema+}}
\newenvironment{definition}{\begin{defi+}}{\end{defi+}}
\newenvironment{notation}{\begin{not+}}{\end{not+}}

\newcommand{\defin}[1]{%
\relax\ifmmode%
\textcolor{blue}{#1}%
\else \textcolor{blue}{\emph{#1}}%
\fi%
}

\newcommand{\setC}{\mathbb{C}}
\newcommand{\setR}{\mathbb{R}}

\newcommand{\setRS}{\bar{\mathbb{C}}} %Riemann Sphere
\newcommand{\bC}{\setC} 
\newcommand{\bR}{\setR}

\renewcommand{\Re}{\operatorname{Re}}
\renewcommand{\Im}{\operatorname{Im}}

\newcommand{\minvset}[1]{{{\mathrm{M}}_{#1}^T}}

\newcommand{\ttt}{\mathfrak{tr}}
\newcommand{\zeros}{\mathcal{Z}}

\newcommand{\horn}[3]{{\leftindex_{{#1}}{\triangle}_{#3}^{#2}}}
\newcommand{\horninf}[1]{{\leftindex_{{#1}}{\angle}}}

\numberwithin{equation}{section}

\newcommand{\bZ}{\mathbb Z}

\newcommand{\Z}{\mathcal Z}

\def \R{{\mathbb R}}

\newcounter{margnotes}

\begin{document}

\title[On boundary points of minimal sets]
{On boundary points of minimal continuously Hutchinson invariant sets}
          
\author[P.~Alexandersson]{Per Alexandersson}
\address[Per Alexandersson]{Department of Mathematics, Stockholm University, SE-106 91 Stockholm, Sweden}
\email{per.w.alexandersson@gmail.com}

\author[N.~Hemmingsson]{Nils Hemmingsson}
\address[Nils Hemmingsson]{Institute for Mathematical Sciences, Stony Brook University, Stony Brook, NY, USA}
\email{nils.hemmingsson@stonybrook.edu}
          
\author[D.~Novikov]{Dmitry Novikov}
\address[Dmitry Novikov]{Faculty of Mathematics and Computer Science, Weizmann Institute of Science, Rehovot, 7610001 Israel}
\email{dmitry.novikov@weizmann.ac.il}

\author[B.~Shapiro]{Boris Shapiro}
\address[Boris Shapiro]{Department of Mathematics, Stockholm University, SE-106 91
Stockholm, Sweden}
\email{shapiro@math.su.se}

\author[G.~Tahar]{Guillaume Tahar}
\address[Guillaume Tahar]{Beijing Institute of Mathematical Sciences and Applications, Huairou District, Beijing, China}
\email{guillaume.tahar@bimsa.cn}

\date{\today}

\begin{abstract} 
A linear differential operator $T=Q(z)\frac{d}{dz}+P(z)$ with polynomial coefficients defines a continuous family of Hutchinson operators when acting on the space of positive powers of linear forms. In this context, $T$ has a unique minimal Hutchinson-invariant set $M_{CH}^{T}$ in the complex plane. Using a geometric interpretation of its boundary  in terms of envelopes of certain families of rays, we subdivide this boundary into local and global arcs (the former being portions of integral curves of the rational vector field $\frac{Q(z)}{P(z)}\partial_{z}$), and singular points of different types which we classify below.
\par
The latter decomposition of the boundary of $M_{CH}^{T}$ is largely determined by its intersection with the plane algebraic curve formed by the inflection points of trajectories of the  field $\frac{Q(z)}{P(z)}\partial_{z}$. We provide an upper bound for the number of local arcs in terms of degrees of $P$ and $Q$. As an application of our classification, we obtain a number of global geometric properties of minimal Hutchinson-invariant sets.
\end{abstract}

\maketitle
%\newpage

\tableofcontents

\section{Introduction}
\smallskip 
Given a linear differential operator 
\begin{equation}\label{eq:1stN}
T=Q(z)\frac{d}{dz}+P(z)
\end{equation}
where $P,Q$ are polynomials that are not identically vanishing, we say that a closed subset $S\subset \bC$ is \textcolor{blue}{\textit{continuously Hutchinson invariant}} for $T$ ($T_{CH}$-invariant set for short) if for any $u \in S$ and any arbitrary non-negative
number $t$, the image $T(f)$ of the function $f(z) = (z-u)^t$ either has all roots in $S$ or vanishes identically. In \cite{AHN+24}, we have initiated the study of general topological properties of $T_{CH}$-invariant sets. 

The main motivation for the present study is that it covers an interesting and manageable special case of a more general inverse P\'olya-Schur problem introduced in \cite{ABS}. For the convenience of our readers, let us briefly recall what the  P\'olya--Schur problem/theory and its inverse are, see \cite{CsCr, ABS}.  

The main question of the classical P\'olya--Schur theory 
can be formulated as follows.

\begin{problem}\label{prob1} 
Given a subset $S\subset \bC$ of the complex plane, describe the
semigroup of all linear operators $T:\bC[z]\to\bC[z]$
sending any polynomial with roots in $S$ to a
polynomial with roots in $S$ (or to $0$).
\end{problem}

\begin{definition}\label{def0}
If an operator $T$ has the latter property, then we say that
\defin{$S$ is a $T$-invariant set}, or that \defin{$T$ preserves $S$}.
\end{definition}

So far, Problem~\ref{prob1} has only been solved for the circular domains (i.e., images of
the unit disk under M\"obius transformations), their boundaries \cite{BB}, 
and  more recently for strips \cite{BCh}. Even a very similar case of the unit interval is still open at present. It seems that for a somewhat general class of subsets $S\subset \bC$, Problem~\ref{prob1} is currently out of reach of all existing methods.  

\medskip
In \cite{ABS}, the following inverse problem in the P\'olya--Schur theory which seems both natural and more accessible than Problem~\ref{prob1} has been proposed.

\begin{problem}\label{prob:main}
Given a linear operator $T:\bC[x]\to\bC[x]$, characterize all closed $T$-invariant subsets of the complex plane. Alternatively, find a sufficiently large class of 
$T$-invariant sets.
\end{problem} 

Paper~\cite{ABS} concentrates on the fundamental case when $T$ is a linear finite order differential operator with polynomial coefficients and shows that under some weak assumptions on these coefficients, there exists a unique minimal $T$-invariant set (and its analogs when $T$ acts on polynomials of degree greater than or equal to a given positive integer $n$). Many basic properties of $T$-invariant sets such as their convexity, compactness etc are discussed in \cite{ABS} as well as the delicate connection of Problem~\ref{prob:main} to the classical complex dynamics. 

However effective criteria characterizing $T$-invariant sets and explicit description of the minimal $T$-invariant set in somewhat interesting cases are currently missing which motivated the consideration in \cite{AHN+24} of the action of $T$ on integer and positive powers of linear forms. This situation is still quite interesting and appears to be more tractable. 

\medskip
In particular, the following results have been obtained in \cite{AHN+24}:
\begin{itemize}
    \item provided that either $P$ or $Q$ is not a constant polynomial, there is a unique \textit{minimal} continuously Hutchinson invariant set $\minvset{CH}$ for a given operator $T$ (in what follows we will always assume that this condition is satisfied);
    \item the only $T_{CH}$-invariant set is the whole $\mathbb{C}$ unless $|\deg Q -\deg P| \leq 1$;
    \item a complete characterization of operators $T$ for which $\minvset{CH}$ has an empty interior has been obtained (see Section~\ref{sub:regularity} for details).
\end{itemize}

In this paper, we will focus on operators whose minimal set $\minvset{CH}$ has a nonempty interior.

\begin{definition}
For an operator $T$ given by \eqref{eq:1stN} with $P$ and $Q$ not vanishing identically, at each point $z$ such that $PQ(z) \neq 0$, we define the {\em associated ray $r(z)$} as the half-line $\lbrace{z+t\frac{Q(z)}{P(z)}~\vert~t \in \mathbb{R}^{+} \rbrace}$. 
\end{definition}

Remarkably, $T_{CH}$-invariant sets (and, in particular, the minimal one) can be characterized in terms of associated rays.

\begin{theorem}[Theorem 3.18 in \cite{AHN+24}]\label{thm:AR}
A closed subset $S \subseteq \mathbb{C}$ is $T_{CH}$-invariant if and only if it satisfies the following two conditions:
\begin{enumerate}
    \item $S$ contains the roots of the polynomials $P$ and $Q$;
    \item for any point $z \notin S$, the associated ray $r(z)$ is disjoint from $S$.
\end{enumerate}
\end{theorem}

\subsection{Main results}

In the present paper, using Theorem~\ref{thm:AR}, we provide a qualitative description of the boundary of minimal continuously Hutchinson invariant sets, including an exhaustive typology of its singular points. Our classification mainly depends on the intersection of the boundary $\partial \minvset{CH}$ with the curve of inflections $\mathfrak{I}_{R}$ of the field $R(z)\partial_{z}=\frac{Q(z)}{P(z)}\partial_{z}$.

\begin{definition}
The \textit{curve of inflections} $\mathfrak{I}_{R}$ of the vector field $R(z)\partial_z$ is defined as the closure of the set of points satisfying $\Im(R')=0$, see \cite{AHN+24}. It is a real plane algebraic curve of degree at most $d= 3\deg P + \deg Q -1$ (in this paper, we will always have $d\geq 3$).
\par
The curve of inflections splits the complex plane into \textit{inflection domains} where the sign of $\Im(R')$ remains the same.
\end{definition}
Points of $\partial \minvset{CH}$ outside its intersection with $\mathfrak{I}_{R}$ are classified with the help of two correspondences $\Gamma$ and $\Delta$ sending the boundary $\partial \minvset{CH}$ to itself and defined as follows:
\par
For a given point $z$ of the boundary $\partial \minvset{CH}$, $\Gamma(z)$ is essentially the intersection of $\minvset{CH}$ with the integral curve of the rational field $R(z)\partial_{z}$ starting at $z$, where $R(z)=Q(z)/P(z)$. In contrast, $\Delta$ is the intersection of the associated ray $r(z)$ with the closure of $\partial \minvset{CH}$ in the compactification $\mathbb{C} \cup \mathbb{S}^{1}$ of the complex plane (see Section~\ref{sub:extended}). Formal definitions of $\Gamma$ and $\Delta$ are given in Definition~\ref{def:DeltaGamma}. Qualitatively, the boundary $\partial \minvset{CH}$ is made of two kinds of arcs:
\begin{itemize}
    \item \textit{local arcs} which are integral curves of the field $R(z)\partial_{z}$ (i.e. $\Delta(z) = \emptyset$ and $\Gamma(z) \neq \emptyset$);
    \item \textit{global arcs} at each point $z$ of which the associated ray $r(z)$ is tangent to $\partial \minvset{CH}$ elsewhere (i.e. $\Gamma(z)=\emptyset$ and $\Delta(z) \neq \emptyset$).
\end{itemize}

Local arcs are locally strictly convex real-analytic arcs (see Proposition~\ref{prop:localconvex}). In contrast, global arcs (formed by points of global type) can fail to be $C^{1}$.
\par
Local arcs inherit an obvious orientation from the vector field $R(z)\partial_{z}$. Global arcs also have canonical orientation, but its definition requires some work (see Section~\ref{sub:orientationglobal}).
\par
Local and global arcs connect special singular points of $\partial \minvset{CH}$ which in most of the cases belong to the curve of inflections. The latter decomposes into three loci (singular, tangent and transverse), each determining its own variety of singular points.

\begin{definition}\label{defn:transverse}
The \textit{curve of inflections} $\mathfrak{I}_{R}$ of the field $R(z)\partial_z$ decomposes into:
\begin{itemize}
    \item the \textit{singular locus} $\mathfrak{S}_{R}$ formed by the points where several branches of $\mathfrak{I}_{R}$ intersect;
    \item the \textit{tangency locus} $\mathfrak{T}_{R}$ formed by the non-singular points where the field $R(z)\partial_{z}$ is tangent to $\mathfrak{I}_{R}$;
    \item \textit{transverse locus} $\mathfrak{I}_{R}^{\ast}$ formed by the non-singular points of $\mathfrak{I}_{R}$ where the field $R(z)\partial_{z}$ is transverse to $\mathfrak{I}_{R}$.
\end{itemize}
\end{definition}

The singular and the tangency loci are given by algebraic conditions.  Therefore their intersection with  $\partial \minvset{CH}$ is controlled in terms of $\deg P$ and $\deg Q$. On the contrary, many points of the boundary can belong to the transverse locus $\mathfrak{I}_{R}^{\ast}$. We refine the definition of the correspondence $\Delta$ according to the value of $\frac{R(z)}{u-z}$ (which, by definition, is a positive number).

\begin{definition}\label{def:-+0}
We define $\Delta(z)=\big(\overline{r(z)}\setminus{\{z\}}\big)\cap\overline{\minvset{CH}}$, where $\overline{r(z)}, \overline{\minvset{CH}}$ are closures of ${r(z)},{\minvset{CH}}$ in the compactification $\bC\cup\mathbb{S}^1$ of $\bC$, respectively.

For any $z \in \mathfrak{I}_{R} \setminus \mathcal{Z}(PQ)$, we have $\Delta(z)=\Delta^{-}(z)  \cup \Delta^0(z)\cup\Delta^{+}(z)$ where $u \in \Delta(z)\cap\bC$ belongs to:
\begin{itemize}
    \item $\Delta^{-}(z)$ if $R'(z) \leq -\frac{R(z)}{u-z}$;
    \item $\Delta^{0}(z)$ if $R'(z)= -\frac{R(z)}{u-z}$;
    \item $\Delta^{+}(z)$ if $R'(z) \geq -\frac{R(z)}{u-z}$,
\end{itemize}
and $u \in \Delta(z)\cap\mathbb{S}^1$ belongs to 
\begin{itemize}
    \item $\Delta^{-}(z)$ if $R'(z) \leq 0$;
    \item $\Delta^{0}(z)$ if $R'(z)= 0$;
    \item $\Delta^{+}(z)$ if $R'(z) \geq 0$.
\end{itemize}
In particular, if $R'(z)>0$, then $\Delta^{-}(z) = \emptyset$.

\end{definition}

The main result of the present paper is a classification of boundary points of minimal continuously Hutchinson sets.

\begin{theorem}\label{thm:MAINClassification}
For any linear differential operator $T$ given by \eqref{eq:1stN}, any point $z$ of the boundary $\partial \minvset{CH}$ of its minimal $T_{CH}$-invariant set belongs to one of the following types:
\begin{itemize}
    \item \emph{roots of polynomials} $P$ and $Q$ (at most $\deg P + \deg Q$ of them);
    
    \item \emph{singular points of the curve of inflections} (at most $2d$ of them);
    
    \item \emph{tangency points between the curve of inflections and the field $R(z)\partial_{z}$}:
    \begin{itemize}
        \item straight segments, half-lines and lines (contained in at most $\deg P + \deg Q + 1$ lines);
        \item at most $2d^{2}$ isolated points;
    \end{itemize}
    
    \item \emph{points of the transverse locus $\mathfrak{I}_{R}^{\ast}$} belonging to one of the four subclasses:
    \begin{itemize}
        \item \emph{bouncing type}: $\Delta^{+}(z) \neq \emptyset$ and $\Gamma \cup \Delta^{-}(z) \neq \emptyset$;

        \item \emph{switch type}: $\Delta^{+}(z) \neq \emptyset$ and $\Gamma(z) \cup \Delta^{-}(z) = \emptyset$;   

        \item \emph{$C^{1}$-inflection type}: $\Delta^{+}(z) = \emptyset$, $\Delta^{-}(z) \neq \emptyset$ and $\Gamma(z) = \emptyset$;
        
        \item \emph{$C^{2}$-inflection type}: $\Delta^{+}(z) = \emptyset$ and either $\Delta^{-}(z) = \emptyset$ or $\Gamma(z) \neq \emptyset$.
    \end{itemize}

    \item \emph{points not on the curve of inflections} belonging to one of the three subclasses:    
    \begin{itemize}
        \item \emph{local type}: $\Gamma(z) \neq \emptyset$ and $\Delta(z)=\emptyset$;
        \item \emph{global type}: $\Gamma(z) = \emptyset$ and $\Delta(z) \neq \emptyset$;
        \item \emph{extruding type}: $\Gamma(z) \neq \emptyset$ and $\Delta(z) \neq \emptyset$.
    \end{itemize}
\end{itemize}
Here, $d = 3\deg P + \deg Q -1$.
\end{theorem}

There can be many singular points of bouncing, extruding, $C^{1}$-inflection, $C^{2}$-inflection and switch types (we do not have a polynomial bound of their number in terms of $\deg P$ and $\deg Q$). An extensive description of their geometric features is given below:
\begin{itemize}
    \item at \textit{points of extruding type}, the boundary of $\partial \minvset{CH}$ is not convex and it switches from a global  to a local arc (see Section~\ref{sub:extruding} and Figure~\ref{fig:pointtypes});
    
    \item at \textit{points of bouncing type}, $\partial \minvset{CH}$ hits the curve of inflections, but does not cross it. In a neighborhood of such a point, the boundary $\partial \minvset{CH}$ remains in the closure of the same inflection domain (see Section~\ref{sub:bouncing} and Figure~\ref{fig:pointtypes});
    
    \item at \textit{points of switch type}, $\partial \minvset{CH}$ is strictly convex, crosses the curve of inflections and the boundary switches from a local  to a global arc (see Section~\ref{sub:Switch} and Figure~\ref{fig:pointtypes});    
    \item at points of $C^{1}$\textit{-inflection type}, $\partial \minvset{CH}$ crosses the curve of inflections and it switches from a global  to another global arc having the opposite orientation. At such a point the curvature of $\partial \minvset{CH}$ is discontinuous  (see Section~\ref{sub:C1} and Figure~\ref{fig:pointtypes});
    
    \item at points of $C^{2}$\textit{-inflection type}, $\partial \minvset{CH}$ crosses the curve of inflections and the boundary switches from a global arc to a local arc. Besides, the curvature of $\partial \minvset{CH}$ is continuous at such a point (see Section~\ref{sub:InflectionTYPE} and Figure~\ref{fig:pointtypes}).\newline
\end{itemize}

\begin{figure}[!ht]
\begin{center}
\includegraphics[width=0.45\textwidth,page=1]{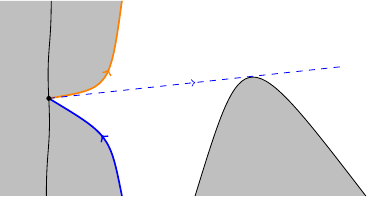}
\hfill
\includegraphics[width=0.45\textwidth,page=2]{per-tikz-figures.pdf} \\
\vspace{1.0cm}
\includegraphics[width=0.3\textwidth,page=3]{per-tikz-figures.pdf}%
\hfill
\includegraphics[width=0.3\textwidth,page=4]{per-tikz-figures.pdf}%
\hfill
\includegraphics[width=0.3\textwidth,page=5]{per-tikz-figures.pdf}
\end{center}
\caption{Top row from the left: Extruding type, Switch type. Bottom row from the left: $C^1$-inflection type, $C^2$-inflection type, Bouncing type. In the pictures, blue arcs are global arcs, red arcs are local arcs and the black arc is a germ of the curve of inflections. The pointed arrow is the associated ray indicating the support points, when applicable.}\label{fig:pointtypes}
\end{figure}

Our second main result is an upper bound on the number of points of $C^{1}$-inflection, $C^{2}$-inflection and switch type in terms of $d=3 \deg P + \deg Q -1$.

\begin{theorem}\label{thm:MainBound}
For any operator $T$ given by \eqref{eq:1stN}, the numbers of points of switch, $C^{1}$-inflection and $C^{2}$-inflection type (respectively $|\mathcal{S}|$, $|\mathcal{I}_{1}|$ and $|\mathcal{I}_{2}|$) in $\partial \minvset{CH}$
satisfy the following bounds:
$$ |\mathcal{S}| \leq e^{16d\ln(d)}+46d^{3};$$
$$ 2|\mathcal{I}_{1}| + |\mathcal{I}_{2}| \leq 2e^{16d\ln(d)}+92d^{3}.$$
\end{theorem}

\begin{corollary}\label{cor:MainLocal}
For any operator $T$ given by \eqref{eq:1stN}, the boundary $\partial \minvset{CH}$ of the minimal set contains at most $d^{16d}+46d^{3}+d(2d+1)$ local arcs.    
\end{corollary}

In the last section of the paper, we deduce many results about the global geometry of minimal sets from the classification of boundary points. In several cases, an exact description can be given in terms of local and global arcs. In particular, we can prove that in generic case, the minimal $T_{CH}$-invariant set is connected in $\mathbb{C}$.

\begin{theorem}\label{thm:CC}
For any linear differential operator $T$ given by \eqref{eq:1stN}, the minimal continuously Hutchinson invariant set $\minvset{CH}$ is a connected subset of $\mathbb{C}$ with the  possible exception of the case when $R(z)$ is of the form $\lambda+\frac{\mu}{z}+o(z^{-1})$ with $\lambda \in \mathbb{C}^{\ast}$ and $\mu/\lambda \in \mathbb{R}$.
\par
In this latter case,  (unless both $P$ and $Q$ are constants and then there is no reasonable notion of a minimal set), $\minvset{CH}$ is formed by at most $\frac{1}{2}\deg P + \frac{1}{2}\deg Q$ connected components.
\end{theorem}

\subsection{Organization of the paper}

\begin{itemize}
    \item In Section~\ref{sec:preliminary}, we provide the basic background information on Hutchinson invariant sets developed in \cite{AHN+24}, including the results about their asymptotic geometry.
    \item In Section~\ref{sec:localanalysis}, we describe the local geometry around singular points of the vector field $R(z)\partial_{z}$ in terms of their local degree and principal value. We also describe the main properties of the curve of inflections defined by the equation $Im(R')=0$ and we also introduce the notion of horns.
    \item In Section~\ref{sec:Boundaryarcs}, we describe boundary points in the complement to the curve of inflections, introducing $\Gamma$ and $\Delta$ correspondences.
    \item In Section~\ref{sec:SingInflection}, we classify boundary points in the generic locus of the curve of inflections, proving Theorems~\ref{thm:MAINClassification} and ~\ref{thm:MainBound} (in Sections~\ref{sub:classification} and~\ref{sub:bound} respectively). Corollary~\ref{cor:MainLocal} is also proved in Section~\ref{sub:bound}.
    \item In Section~\ref{sec:GlobalGeometry}, we apply the latter results to get precise descriptions of minimal sets in several cases. Theorem~\ref{thm:CC} is proven in Section~\ref{sub:CC}.
\end{itemize}
\medskip 
\noindent 
{\it Acknowledgments.} The third author is supported by the Israel Science Foundation (grants No. 1167/17 and 1347/23), by funding received from the MINERVA Stiftung with the funds from the BMBF of the Federal Republic of Germany, and from the European Research Council (ERC) under the European Union Horizon 2020 research and innovation program (grant agreement No. 802107). The fourth author wants to acknowledge the financial support of his research provided by the Swedish Research Council grant 2021-04900 and the hospitality of the Weizmann institute of Science in April 2021 and January 2024 when substantial part of the project was carried out. He is sincerely grateful to Beijing Institute of Mathematical Sciences and Applications for the support of his sabbatical leave in the Fall of 2023. The fifth author would like to thank France Lerner for valuable remarks about the terminology related to singular points.\newline

\section{Preliminary results and basic properties of \texorpdfstring{$\minvset{CH}$}{MCH}}\label{sec:preliminary}

The following notation will be important throughout this text. 

\begin{notation} Given an operator $T$ as in \eqref{eq:1stN}, we define $p_{\infty},q_{\infty} \in \mathbb{C}^{\ast}$, and $p,q \in \mathbb{N}$ so that
$$P(z)=p_{\infty}z^{p}+o(z^{p});$$
$$Q(z)=q_{\infty}z^{q}+o(z^{q}).
$$

\smallskip
Furthermore, we set  $\lambda = \frac{q_{\infty}}{p_{\infty}} \in \mathbb{C}^{\ast}$ and $\phi_{\infty}=\arg(\lambda)$.

\smallskip
Similarly, for any point $\alpha \in \mathbb{C}$, we have $R(z)=r_{\alpha}(z-\alpha)^{m_{\alpha}}+o(|z-\alpha|^{m_{\alpha}})$ with $r_{\alpha} \neq 0$ and $m_{\alpha} \in \mathbb{Z}$. We denote by $\phi_{\alpha}$ the argument of $r_{\alpha}$.
\end{notation}

\begin{remark}
Observe that  frequently used  affine changes of the variable $z$ are applied to the vector field $R(z)\partial_{z}$ and not to the rational function $R(z)$ itself.
\end{remark}

\subsection{Regularity of the minimal set}\label{sub:regularity}

For an operator $T$ as in \eqref{eq:1stN}, its minimal set $\minvset{CH}$ can be of three possible types:
\begin{itemize}
    \item \textit{regular} if $\minvset{CH}$ coincides with the closure of its interior;
    \item \textit{fully irregular} if $\minvset{CH}$ has empty interior;
    \item \textit{partially irregular} if $\minvset{CH}$ has nonempty interior but is not regular.
\end{itemize}

Actually, irregularity is related to specific reality conditions. The characterization of operators for which $\minvset{CH}$ is fully irregular is contained in Theorem~1.15 of \cite{AHN+24}.

\begin{theorem}\label{thm:irregular}
For an operator $T$ as in \eqref{eq:1stN}, the minimal set $\minvset{CH}$ is fully irregular in the following cases:
\begin{itemize}
    \item $R(z)=\lambda$ for some $\lambda \in \mathbb{C}^{\ast}$;
    \item $R(z)=\lambda(z-\alpha)$ for some $\lambda \notin \mathbb{R}_{<0}$, $\alpha \in \mathbb{C}$ and $\deg Q =1$;
    \item $R(z)=\lambda(z-\alpha)$ for some $\lambda \in \mathbb{R}_{>0}$, $\alpha \in \mathbb{C}$ and $\deg Q \geq 2$;
    \item operators satisfying the following conditions (up to an affine change of variable):
    \begin{itemize}
        \item $R(z)$ is real on $\mathbb{R}$;
        \item roots of $P$ and $Q$ are real, simple and interlacing (i.e. the roots of $P$ and $Q$ alternate along the real axis);
        \item $|\deg Q - \deg P| \leq 1$;
        \item if $\deg Q - \deg P = 1$, then $\lambda \in \mathbb{R}_{>0}$.
    \end{itemize}
\end{itemize}
In any other case, $\minvset{CH}$ has a nonempty interior.
\end{theorem}

In this paper, we will always assume that $\minvset{CH}$ has a nonempty interior.

\begin{remark}\label{rem:degree}
If $\deg P + \deg Q \leq 1$, then $\minvset{CH}$ is either fully irregular or coincides with $\mathbb{C}$ (see Theorem 1.15 of \cite{AHN+24}). Therefore, our operators will always satisfy $\deg P + \deg Q \geq 2$. This implies in particular that $d = 3\deg P + \deg Q -1$ satisfies $d \geq 3$.
\end{remark}

Referring to the closure of the interior of $\minvset{CH}$ as the \textit{regular locus} and its complement in $\minvset{CH}$ as the \textit{irregular locus}, we observe that the latter is contained in very specific lines of the plane.

\begin{definition}\label{def:Rinvariant}
For a given rational function $R(z)$, a line $\Lambda$ is called $R$-invariant if for any $z \in \Lambda$ such that $R(z)$ is defined, we have $z+R(z) \in \Lambda$.

\medskip
In particular, up to an affine change of variable, we can assume $\Lambda = \mathbb{R}$ and thus $R(z)$ is a real rational function. Besides, a $R$-invariant line is automatically an irreducible component of the curve of inflections $\mathfrak{I}_{R}$.
\end{definition}

\begin{definition}\label{defn:tail}
For an operator $T$ whose minimal set $\minvset{CH}$ is not fully irregular, a \textit{tail} is a semi-open straight segment $[\alpha,\beta[$ in $\minvset{CH}$ satisfying the following conditions:
\begin{itemize}
    \item the segment $]\alpha,\beta]$ belongs to an $R$-invariant line;
    \item for any $z \in ]\alpha,\beta]$, $\frac{\beta-\alpha}{R(z)} \in \mathbb{R}_{>0}$;
    \item for any $z \in ]\alpha,\beta]$, $z$ is disjoint from the regular locus of $\minvset{CH}$;
    \item $\alpha$ belongs to the regular locus of $\minvset{CH}$;
    \item $\beta \in \mathcal{Z}(PQ)$;
    \item $\beta$ is a root of the same multiplicity for both $P$ and $Q$.
\end{itemize}
In particular, every tail belongs to a $R$-invariant line.
\end{definition}

The following fact has been proven in Corollary~7.8 of \cite{AHN+24}.

\begin{theorem}\label{thm:TAIL}
For an operator $T$ whose minimal set $\minvset{CH}$ is not fully irregular, the irregular locus of $\minvset{CH}$ is a (possibly empty) finite union of tails.
\end{theorem}

In particular, if $P$ and $Q$ have no common roots, then the minimal set of the corresponding operator is either regular or fully irregular.

\subsection{Extended complex plane}\label{sub:extended}

Following Theorem~\ref{thm:AR}, $T_{CH}$-invariant sets are characterized by the position of the associated rays starting in their complements. Let us introduce a certain compactification\footnote{Notice that the most frequently used compactification of $\mathbb{C}$ is $\setRS=\bC P^1$.} of $\bC$ which comes very handy in our considerations. We baptise it the \textit{extended complex plane} $\mathbb{C} \cup \mathbb{S}^{1}\supset \mathbb{C}$.\newline

\emph{The extended complex plane} $\mathbb{C} \cup \mathbb{S}^{1}$ is set-theoretically the disjoint union of $\mathbb{C}$ and $\mathbb{S}^{1}$ endowed with the topology defined by the following basis of neighborhoods:
\begin{itemize}
    \item for a point  $x\in \mathbb{C}$, we choose the usual open neighborhoods of $x$ in $\mathbb{C}$;
    \item for a direction $\theta \in \mathbb{S}^{1}$, we choose open neighborhoods of the form $I \cup C(z,I)$ where $I$ is an open interval of $\mathbb{S}^{1}$ containing $\theta$ and $C(z,I)$ is an open cone with apex $z \in \mathbb{C}$ whose opening (i.e. the interval of directions) is $I$.
\end{itemize}

\begin{definition} Given $R(z)$ as above, let $p\in \bC$ be a non-singular point of $R(z)$. We define $\sigma(p)$ as the argument of $R(p)$. We think of $\sigma(p)$ as a point of the circle at infinity. 
\end{definition}

One can easily see that $\mathbb{S}^{1}$ of the extended plane $\mathbb{C} \cup \mathbb{S}^{1}$ can be identified with the above circle at infinity. The extended plane is compact and homeomorphic to a closed disk. In particular, usual straight lines in $\mathbb{C}$ have compact closures in $\mathbb{C} \cup \mathbb{S}^{1}$. (Below we will make no distinction between a real line in $\mathbb{C}$ and its closure in $\mathbb{C} \cup \mathbb{S}^{1}$). \textit{Open half-planes} in $\mathbb{C} \cup \mathbb{S}^{1}$ are, by definition, connected components of the complement to a line.\newline

Given a $T_{CH}$-invariant set $S \subset \mathbb{C}$, we denote by $\overline{S}$ its closure in the extended plane $\mathbb{C}\cup \mathbb{S}^{1}$.\newline

The following result, but with a slightly different formulation, has been proved in Lemma 4.4 of \cite{AHN+24}.

\begin{lemma}\label{lem:ArcCone}
Given an $T_{CH}$-invariant set $S\subset \mathbb{C}$,  let $\alpha:[0,1]\rightarrow \mathbb{C}$ be such that:
\begin{itemize}
    \item $\forall t \in (0,1)$, $\alpha_{t} \in S^{c}$;
    \item $\sigma(\alpha_{0}) \neq \sigma(\alpha_{1})$;
    \item $\sigma(\alpha)$ is homotopic to the positive arc from $\sigma(\alpha_{0})$ to $\sigma(\alpha_{1})$ in the circle at infinity via a homotopy $H(t,x):[0,1]\times [0,1]$ such that $H(0,x_0)=\sigma(\alpha(0)), $ $H(1,x_0)=\sigma(\alpha(1))$ for all $x_0\in [0,1]$. 
\end{itemize}
If $X$ denotes the connected component of in $\mathbb C \cup \mathbb S^1$ containing the interval $]\sigma(\alpha_{0}),\sigma(\alpha_{1})[$ in the complement of $r(\alpha_{0}) \cup \alpha \cup r(\alpha_{1})$, then $X \subset S^{c}$.
\end{lemma}

\subsection{Integral curves}

Another result has been proved in Proposition A.2 of \cite{AHN+24}.

\begin{proposition}\label{prop:IntegralCurve}
Given a $T_{CH}$-invariant set $S\subset \mathbb{C}$ and some point $z_{0} \in S$, if there is a positively oriented integral curve $\gamma:[0,\epsilon[ \rightarrow \mathbb{C}$ of the vector field $R(z)\partial_{z}$ such that $\lim\limits_{t \rightarrow \epsilon} \gamma(t)= z_{0}$, then for any $t \in [0,\epsilon]$, $\gamma(t) \in S$.
\end{proposition}

When referring to the proposition above, we say that the bounded \emph{backward trajectories} of $R(z)\partial_{z}$ of points in any invariant set $S$ belongs to $S$.

\subsection{Root trails}\label{sub:roottrail}

For any point $u \in \mathbb{C}$, the \textit{root trail} $\mathfrak{tr}_{u}$ of $u$ is the closure of the set of points $z$ such that the associated ray $r(z)$ contains $u$. Except for the trivial cases described in Section~3 of \cite{AHN+24}, root trails are plane real-analytic curves. By definition, the root trail of any point of $\minvset{CH}$ is also contained in $\minvset{CH}$. Furthermore, for any fixed $u\in \bC$, we defined a \emph{$t$-trace} (corresponding to $u$) as any continuous function $\gamma(t)$ such that 
$$tQ(\gamma(t))+(\gamma(t)-u)P(\gamma(t))=0$$
for all $t\geq 0$. That is, any $t$-trace $\gamma(t)$ is a concatenation of parts of $\mathfrak{tr}_u$ such that the resulting curve is continuous for any $t\geq 0$.

\begin{lemma}\label{lem:RootTrailSlope}
Consider a linear differential operator $T$ given by \eqref{eq:1stN} and some point $u \in \mathbb{C}$. Assuming that $R(z)$ is not of the form $\lambda(z-u)$,
then 

(i) for any point $u \in \mathbb{C}$ and any point $z_{0} \notin \mathcal{Z}(PQ)$ such that $z_{0} \in \mathfrak{tr}_{u}$ and $R(z_{0})+(u-z_{0})R'(z_{0}) \neq 0$,  the root trail $\mathfrak{tr}_{u}$ has a unique branch passing through $z_{0}$ and its tangent slope is the argument of $\frac{R^{2}(z_{0})}{R(z_{0})+(u-z_{0})R'(z_{0})}$ (mod $\pi$).
\par

(ii) If $R(z_{0})+(u-z_{0})R'(z_{0}) = 0$ and $m\ge 2 $ is the smallest integer such that $R^{(m)}(z_{0}) \neq 0$, then $\mathfrak{tr}_{u}$ has $m$ intersecting branches at $z_{0}$. Their tangent slopes are:
$$
\frac{\theta_{0}}{m} +\frac{k\pi}{m},
$$
where $\theta_{0}$ is the argument of $\frac{R(z_{0})}{R^{(m)}(z_{0})}$ and $k \in \mathbb{Z}/m\mathbb{Z}$.
\end{lemma}

Before proving  Lemma~\ref{lem:RootTrailSlope} we prove the next two Lemmas.
\begin{lemma}\label{lem:trajcurvature}
        If $\gamma(t)$ is smooth planar curve, $\gamma(0)=z_0$, and $\dot\gamma(t)=G(\gamma(t))$ for some function $G$ holomorphic and non-vanishing at $z_0$ then the sign of the curvature of $\gamma(t)$ at $z_0$ coincides with the sign of $\Im G'(z_{0})$.
\end{lemma}
    
    Indeed, then $
    \ddot\gamma(t)=G'(\gamma(t))\cdot \dot\gamma(t)
    $. By definition,  the sign of the curvature of $\gamma(t)$ at $z_0$ coincides with the sign of 
    $\Im\frac{\ddot\gamma(t)}{\dot\gamma(t)}|_{t=0}=\Im G'(z_{0})$.

    \begin{lemma}\label{lem:branches and curvature}
        Let $F$ be a  function holomorphic at $z_0$ with $F(z_0)\in\R$ and let  $m=\operatorname{ord}_{z_0}\left(F-F(z_0)\right)$.   Then the germ of $I_F=\{\Im F=0\}$ at $z_0$ consists of $m$ smooth branches with tangent slopes $\frac{\theta_{0}}{m} +\frac{k\pi}{m}$, $k\in \mathbb{Z}/m\mathbb{Z}$, where $\theta_0=-\arg F^{(m)}(z_0)$.

        If $m=1$ and  $\gamma(t)$ is a parameterization of $I_F$ such that $F(\gamma(t))\equiv F(z_0)+t$ then the sign of the curvature of $\gamma(t)$ coincides with the sign of  $-\Im \left[\frac{F''(z_0)}{(F')^2(z_0)}\right]$.
    \end{lemma}
\begin{proof}
    Indeed, we have $F(z)=a_0+a_m(z-z_0)^m+...$, $a_0\in\R$, so the branches of $I_F$ are tangent to the $m$ lines satisfying equation $\Im a_m(z-z_0)^m=0$, which have slopes as stated.

    For the second claim, note that $\dot\gamma(t)=\frac{1}{F'(\gamma(t))}$, so the  claim follows from Lemma~\ref{lem:trajcurvature}.
 \end{proof}    
 \begin{remark}
    Similar results hold for  $F$ having a pole at $z_0$ by considering $\tfrac 1 {F}$.
 \end{remark}
\begin{proof}[Proof of Lemma~\ref{lem:RootTrailSlope}]
Note that by definition 
    $$
    \mathfrak{tr}_{u}=\left\{z\in\bC \text{ s.t. } \frac{R(z)}{u-z}\in \bR_+\right\}\subset\left\{\Im\frac{R(z)}{u-z}=0\right\},
    $$ 
    and  Lemma~\ref{lem:RootTrailSlope} follows from the Lemma~\ref{lem:branches and curvature}  with $F(z)=\frac{R(z)}{u-z}$ and the fact that $\arg R(z_0)=\arg(u-z_0)$.
\end{proof}

\begin{remark}
    The condition $R(z_{0})+(u-z_{0})R'(z_{0}) = 0$ means that the point $u=z_0-\frac{R(z_0)}{R'(z_0)}$ is obtained as the first iteration of   Newton's method of approximating roots of $R(z)$ with the starting point $z_0$. 
\end{remark}
When $u$ is a point at infinity in the extended plane $\mathbb{C} \cup \mathbb{S}^{1}$, the root trail $\mathfrak{tr}_{u}$ of $u$ is the closure of the points $z$ where the argument of $R(z)$ coincides with $u$.

\begin{lemma}\label{lem:RTSlopeINFINITY}
Consider a linear differential operator $T$ given by \eqref{eq:1stN} such that $R(z)$ is not constant. For any point $u$ at infinity  and any point $z_{0} \notin \mathcal{Z}(PQ)$ such that $z_{0} \in \mathfrak{tr}_{u}$, provided $R'(z_{0}) \neq 0$, the root trail $\mathfrak{tr}_{u}$ has a unique branch passing through $z_{0}$ and its tangent slope is the argument of $\frac{R(z_{0})}{R'(z_{0})}$ (mod $\pi$).
\par
If $R'(z_{0}) = 0$ and $m\ge 2$ is the smallest integer  such that $R^{(m)}(z_{0}) \neq 0$, then $\mathfrak{tr}_{u}$ has $m$ intersecting branches at $z_{0}$. Their tangent slopes are:
$$
\frac{\theta_{0}}{m} +\frac{k\pi}{m},
$$
where $\theta_{0}$ is the argument of $\frac{R(z_{0})}{R^{(m)}(z_{0})}$ and $k \in \mathbb{Z}/m\mathbb{Z}$.
\end{lemma}

\begin{proof}
    In this case $\mathfrak{tr}_{u}\subset\left\{\Im \left(R(z)/R(z_0)\right)=0\right\}$ and the claim follows again from Lemma~\ref{lem:branches and curvature} 
\end{proof}

\begin{remark}\label{rem:ROOTTRAIL}
From Lemmas~\ref{lem:RootTrailSlope} and~\ref{lem:RTSlopeINFINITY} it immediately follows   that if a root trail $\mathfrak{tr}_{u}$ can have $m \geq 2$ branches at some point $z_{0}$, then $z_{0}$ belongs to the curve of inflections $\mathfrak{I}_{R}$ (because $R(z_{0})$ and $u-z_{0}$ are real colinear).
\par
Besides, if $m \geq 3$, then $R^{(k)}(z_{0})=0$ for $2 \leq k \leq m-1$ and $z_{0}$ is a singular point of $\mathfrak{I}_{R}$.
\end{remark}

\subsubsection{Concavity of root trails}

\begin{proposition}\label{prop:RTconcave}
Let $u$ be a point of the extended plane $\mathbb{C} \cup \mathbb{S}^{1}$ and $z_{0}$ be a point of $\mathfrak{tr}_{u}$ such that $z_{0} \notin \mathcal{Z}(PQ) \cup \mathfrak{I}_{R}$ and $z_{0} \neq u$. We denote by $L$ the tangent line to  $\mathfrak{tr}_{u}$ at $z_{0}$. We define $f(z,u)$ to be:
\begin{itemize}
    \item $\frac{\left(R''(z_0)(u-z_0)^2+2R'(z_0)(u-z_0)+2R(z_0)\right](u-z_0)}{\left(R'(z_0)(u-z_0)+R(z_0)\right)^2}$ if $u \in \mathbb{C}$;
    \item $\frac{R''(z_{0})R(z_{0})}{R'(z_{0})^{2}}$ if $u$ is a point at infinity.
\end{itemize}
Then the germ of $\mathfrak{tr}_{u}$ at $z_{0}$ belongs to 

(i) the same half-plane bounded by $L$ as the associated ray $r(z_{0})$ if $\Im(f)$ and $\Im(R'(z_{0}))$ have opposite signs.

(ii) 
They belong to distinct half-planes bounded by $L$ if $\Im(f)$ and $\Im(R'(z_{0}))$ have the same sign.\newline

Finally, $\mathfrak{tr}_{u}$ has an inflection point at  $z_{0}$ if $\Im(f) = 0$.
\end{proposition}

\begin{proof}
    Let $F_u(z)=\frac{R(z)}{u-z}$ for $u\in \mathbb{C}$ and $F_u(z)=u^{-1}R(z)$ for $u\in\mathbb{S}^1$ so that $\mathfrak{tr}_u=\{\Im F_u(z)=0\}$. Let  $c=F_u'(z_0)$. 
    We have  
    $$
    L=\left\{    z_0+c^{-1}\R\right\}=\left\{z|\Im \left(c(z-z_0)\right)=0\right\}.
    $$
    If $\gamma(t)$ is a local parameterization of $\mathfrak(tr)_u$ at $z_0$ such that $F_u(\gamma(t))=F_u(z_0)+t$ then $\dot\gamma(0)=c^{-1}$. Moreover, $\mathfrak(tr)_u\subset L_+=\left\{\Im c(z-z_0)>0\right\}$ if the curvature of $\gamma(t)$ is positive and $\mathfrak(tr)_u\subset L_-=\left\{\Im c(z-z_0)<0\right\}$ otherwise. 

    The tangent ray  $r(z_0)=\{z_0+R(z_0)\R_+\}$ lies in $L_+$ if $\Im R(z_0)c>0$ and in $L_-$ otherwise.

    By Lemma~\ref{lem:branches and curvature} the sign of curvature of $\gamma(t)$ is opposite to the sign of $\Im \left[\frac{F_u''(z_0)}{\left(F_u'\right)^2(z_0)}\right]$.
    
    For $u\in \mathbb{C}$ we have $c= F_u'(z_0)=\frac{R'(z_0)(u-z_0)+R(z_0)}{(u-z_0)^2}$ and 
$$
F_u''(z_0)=\frac{R''(z_0)(u-z_0)^2+2R'(z_0)(u-z_0)+2R(z_0)}{(u-z_0)^3},
$$
so we are interested in signs of 
$$
\Im R(z_0)c=\Im R(z_0)\frac{R'(z_0)(u-z_0)+R(z_0)}{(u-z_0)^2}=\Im R'(z_0)
$$ (recall that $\frac{R(z_0)}{u-z_0}>0$) and 
\begin{equation}\label{eq:curvature of tr}
\Im \frac{F_u''(z_0)}{\left(F_u'\right)^2(z_0)}=\Im\frac{\left(R''(z_0)(u-z_0)^2+2R'(z_0)(u-z_0)+2R(z_0)\right](u-z_0)}{\left(R'(z_0)(u-z_0)+R(z_0)\right)^2}.
\end{equation}

 For $u\in\mathbb{S}^1$ we have $c=u^{-1}R'(z_0)$ and we are interested in the signs of
 $\Im R'(z_0)$ and $\Im \frac{R''(z_0)R(z_0)}{\left(R'\right)^2(z_0)}$.
 
\end{proof}

In the transverse locus $\mathfrak{I}_{R}^{\ast}$ of the curve of inflections, the concavity of root trails with respect to the line containing the associated ray depends on the sign of some geometrically meaningful real function.

\begin{proposition}\label{prop:RTInflecCONCAVE}
Consider a point $z_{0} \in\mathfrak{I}_{R}^{\ast} \setminus \mathcal{Z}(PQ)$ and some point $u \in \mathbb{C} \cup \mathbb{S}^{1}$. Assume that $R(z_{0})+R'(z_{0})(u-z_{0}) \neq 0$ (or $R'(z_{0}) \neq 0$ if $u$ is a point at infinity). Let $L$ be the line containing the associated ray $r(z_{0})$.
\par
The germ of $\mathfrak{tr}_{u}$ at $z_{0}$ and the positive germ $\gamma^{+}_{z_{0}}$ of the integral curve of the field $R(z)\partial_{z}$ starting at $z_{0}$ belong to the same open half-plane bounded by $L$   if $R'(z_{0})+R(z_{0})/(u-z_{0})$ is negative ($R'(z_{0})<0$ if $u$ is a point at infinity).
\par
The germ of $\mathfrak{tr}_{u}$ at $z_{0}$ and $\gamma^{+}_{z_{0}}$ belong to opposite open half-planes bounded by $L$ if $R'(z_{0})+R(z_{0})/(u-z_{0})$ is positive ($R'(z_{0})>0$ if $u$ is a point at infinity).
\end{proposition}

\begin{proof}
Without loss of generality, we assume that $z_{0}=0$ and $R(z)= 1 + R'(0) z + (a+bi)z^{2} +o(z^{2})$ with $R'(0) \in \mathbb{R}$, $a \in \mathbb{R}$ and $b \in \mathbb{R}_{>0}$ ($b \neq 0$ because $z_{0}=0$ belongs to the transverse locus of the curve of inflections). Necessarily $u>0$. Since $b>0$, $\gamma^{+}_{0}$ belongs to the upper half-plane.

By Lemma~\ref{lem:RootTrailSlope} (if $u \in \mathbb{C}$) and Lemma~\ref{lem:RTSlopeINFINITY} (if $u \in \mathbb{S}^{1}$), $\mathfrak{tr}_{u}$ has a unique branch at $0$ tangent to $\R$.
Let $F_u(z)=\frac{R(z)}{u-z}$ for $u\in r(z_0)$ and $F_u(z)=R(z)$ for $u\in\mathbb{S}^1$, so $\mathfrak{tr}_{u}=\{\Im F_u(z)=0\}$. Choose a parameterization $\gamma(t)$ of this branch in such a way that $F_u(\gamma(t))=F_u(z_0) +t$.   Then 
$$
\dot\gamma(0)=\frac1{F_u'(0)}=\frac{{u-z_0}}{R'(0)+\tfrac{R(0)}{u-z_0}}  \qquad\text{ or }    \qquad \dot\gamma(0)=\frac1{R'(0)}
$$ for $ u\in \bC$  or  $u\in\mathbb{S}^1$ being  a point at infinity, respectively. 
Therefore  $\dot\gamma(0)>0$ if $R'(0)+R(0)/({u-z_0})>0$  (resp. $R'(0)>0$) and $\dot\gamma(0)<0$ otherwise.

By \eqref{eq:curvature of tr} the sign of the curvature of $\gamma(t)$ at $0$ is opposite to the sign of $\Im R''(0)=\Im b>0$, i.e. is negative. Thus $\gamma(t)$ lies in the lower half-plane (i.e. not in the same half-plane as $\gamma^{+}_{0}$) if $R'(0)+R(0)/({u-z_0})>0$ is positive and in the same half-plane as $\gamma^{+}_{0}$ if $R'(0)+R(0)/({u-z_0})<0$ ($R'(0)>0$ and $R'(0)<0$ resp. for $u\in\mathbb{S}^1$). Since $R'(z_0)\in \R$, the number $R'(z_0)+R(z_0)/({u-z_0})$ is invariant under the maps $z\mapsto az+b$ and $z\mapsto \bar z$ used for normalization, and the claim  follows.
\end{proof}

\subsubsection{Root trails and connected components of the minimal set}

When $\deg Q - \deg P =0$, root trails provide a bound on the number of connected components of the minimal set (in all other cases, it is known that $\minvset{CH}$ is connected).

\begin{proposition}\label{prop:CCrawBOUND}
Consider a linear differential operator $T$ given by \eqref{eq:1stN} and satisfying $\deg Q - \deg P =0$. Any connected component $C$ of $\minvset{CH}$ satisfies the following conditions:
\begin{itemize}
    \item $C$ contains at least one root of $P$;
    \item $C$ contains at least one root of $Q$;
    \item the sum of orders of zeros and poles  of $R(z)$ in $C$ vanishes.
\end{itemize}
\end{proposition}

\begin{proof}
We assume that a connected component $C$ of $\minvset{CH}$ is disjoint from $\mathcal{Z}(P)$. Note that $\deg Q-\deg P=0$ implies that the union of the zeros of $tQ(z)+P(z)(z-u)$ for any $u \in \mathbb{C}$, $T>0$ and $t\in [0,T]$ is bounded.

Hence, for any $u \in \minvset{CH}\setminus C$, the root trail of $u$ is disjoint from $C$, as otherwise there would be points in the complement of $\minvset{CH}$ belonging to the root trail of $u$. Since $\minvset{CH}$ coincides with the $T_{CH}$-extension of any point in $\minvset{CH}$ (see Lemma 2.2 of \cite{AHN+24}), it follows that $C$ cannot belong to the minimal invariant set.
\par
Suppose now that there is a component for which the sums of orders of the zeros of $Q$ does not equal  the sums of orders of the zeros of $P$. Then there is a component $C$ such that the sums of the orders of the zeros of $P$, say $d_0$ is strictly greater than the sums of the orders of the zeros of $Q$, say $d_1$. Taking $u\in C$ we have that for all $t$, the zeros of $tQ(z)+P(z)(z-u)$ belonging to $C$ have total degree $d_0+1$. However, when sending $t\to \infty$, $d_1$ of these zeros tend to the zeros of $Q$ belonging to $C$ and at most one tend to $\infty$. This implies that at least $d_0-d_1>0$ of the end points of the root trail of $u$ does not belong to $C$, a contradiction.
\end{proof}

We prove now that the interior $(\minvset{CH})^{\circ}$ of the minimal set satisfies (outside zeroes and poles of $R(z)$) a weak property of local connectedness.

\begin{lemma}\label{lem:ImKleinen}
For any linear differential operator $T$ given by \eqref{eq:1stN} we consider a point $\alpha$ of the boundary $\partial \minvset{CH}$ that is neither a zero nor a pole of $R(z)$. For any sufficiently small neighborhood $V$ of $\alpha$, the connected component $M_V$ of $V\cap \minvset{CH}$ containing $\alpha$ has connected interior.
\end{lemma}

\begin{proof}
If $\alpha$ does not belong to the regular locus of $\minvset{CH}$, then $M_V$ is a line segment and hence has empty interior.
\par
We have $R(z)=r_{\alpha}+o(z-\alpha)$ for some $r_{\alpha} \in \mathbb{C}^{\ast}$. Then, a continuity argument immediately shows that $M_V$ has connected interior, as otherwise points in the complement of $\minvset{CH}$ would have associated rays intersecting $\minvset{CH}$.
\par

\end{proof}

In the following, we prove that the closure of a connected component of the interior of $\minvset{CH}$ cannot be disjoint from $\mathcal{Z}(P)$.

\begin{lemma}\label{lem:BoundInterior}
For any linear differential operator $T$ given by \eqref{eq:1stN}, one of the following statements holds:
\begin{enumerate}
    \item $\minvset{CH}$ is fully irregular;
    \item $\minvset{CH}=\mathbb{C}$;
    \item the closure of any connected component of the interior $(\minvset{CH})^{\circ}$ of the minimal set contains a root of $P(z)$;
    \item the closure of any connected component of the interior $(\minvset{CH})^{\circ}$ of the minimal set contains an endpoint of a tail.
\end{enumerate} 
\end{lemma}

\begin{proof}
We suppose that we are not in the case (1), (2). Besides, we assume the existence of a connected component $C$ of the interior $(\minvset{CH})^{\circ}$ of the minimal set whose closure is disjoint from $\mathcal{Z}(P)$, contradicting statement (3).
\par
We first prove that $C$ cannot be the only connected component of $(\minvset{CH})^{\circ}$. Indeed, roots of $P(z)$ that do not belong to the regular locus of $\minvset{CH}$ (the closure of the interior) belong to tails (see Theorem~\ref{thm:TAIL}) and they are not zeros or poles of $R(z)$. Besides, $\minvset{CH}$ is assumed to be distinct from $\bC$. Consequently, we have $|\deg Q - \deg P| \leq 1$. The only case where the regular locus of $\minvset{CH}$ can be disjoint from $\mathcal{Z}(P)$ is when $R(z)$ is of the form $\lambda$ or $\lambda(z-\alpha)$. In the first case, $\minvset{CH}$ is known to be fully irregular. In the second case, either $\lambda \in \mathbb{R}_{>0}$ (and $\minvset{CH}$ is fully irregular, see Theorem~\ref{thm:irregular}) or $\lambda \notin \mathbb{R}_{>0}$ and $\minvset{CH}$ has no tails (and $P(z)$ has no root at all). We assume therefore that the interior $\minvset{CH}$ has several connected components.
\par
We denote by $A$ the set of points of $C$ that belong to the closure of another connected component of $(\minvset{CH})^{\circ}$. By assumption, these are not these points are not roots of $P(z)$ and Lemma~\ref{lem:ImKleinen} shows that each of them is a zero of $R(z)$.
\par
Since $\minvset{CH}$ is minimal, there is a point $u\in \overline{\minvset{CH}\setminus C}$ and a point $z_{0}\in \mathfrak{tr}_{u} \cap C$. As the root trail $\mathfrak{tr}_{u}$ changes continuously in $u$, $u$ may be chosen outside $A$. Since $|\deg Q-\deg P|\leq 1$, the zeros of $tQ(z)+P(z)(z-u)$ as $t \rightarrow 0$ tends to $\mathcal{Z}(P) \cup \lbrace{ u \rbrace}$. Further, we can assume that $\gamma(t)$ does not equal $\infty$ for some finite $t$, as this would imply $\deg Q-\deg P=1$ and $\lambda<0$, in which case $\minvset{CH}$ is equal to $\bC$. Hence, the minimal set $\minvset{CH}$ therefore contains  a continuous path $\gamma(t)$ from an element of $\mathcal{Z}(P) \cup \lbrace{ u \rbrace}$ to $z_{0}$ such that $\gamma(t)$ solves $tQ(\gamma(t))+P(\gamma(t))(\gamma(t)-u)=0$.
\par
The path $\gamma$ has to enter the component $C$ and can do so either through a tail or an element of $A$. The path $\gamma$ cannot contain any element $\alpha \in A$ because the equations $Q(\alpha)=0$ and $tQ(\alpha)+P(\alpha)(\alpha-u)=0$ (for some $t>0$) imply $P(\alpha)=0$, contradicting our assumption. Our assumption that neither (1), (2) nor (3) was satisfied thus implies (4).
\end{proof}

\subsection{Asymptotic geometry of Hutchinson invariant sets}\label{sec:asymptotic}

Let us recall the results of \cite{AHN+24} concerning  minimal Hutchinson invariant sets (see Theorems~1.11 and 1.12 of \cite{AHN+24}).

\begin{theorem}\label{thm:main} 
For any operator $T$ as in \eqref{eq:1stN} with a minimal set $\minvset{CH}$ having a nonempty interior, $\minvset{CH}$ is:
\begin{itemize}
    \item a compact contractible subset of $\mathbb{C}$ if $\deg Q -\deg P= 1$, and $Re(\lambda) \geq 0$;
    \item a noncompact non-trivial subset of $\mathbb{C}$ if $\deg Q -\deg P= 0$ or $-1$;
    \item trivial, i.e. equal to $\bC$ otherwise.
\end{itemize}
Besides, the closure $\overline{\minvset{CH}}$ in the extended plane $\mathbb{C} \cup \mathbb{S}^{1}$ is contractible, connected and compact.
\end{theorem}

Thus, the only interesting cases for the description of $\partial \minvset{CH}$ are those for which the values of $\deg Q -\deg P$ are $1$, $0$ or $-1$. In the latter two cases, we have more precise results given below.

\subsubsection{$\deg Q -\deg P=-1$}

The following statement has been proved in Corollary~6.2 of \cite{AHN+24}.

\begin{proposition}\label{prop:qp-1}
For an operator $T$ as in \eqref{eq:1stN} such that $\deg Q -\deg P=-1$. Then the complement of its minimal Hutchinson invariant set $\minvset{CH}$ in $\mathbb{C}$ has exactly two connected components $X_{1},X_{2}$.  
Each $X_{i}$ contains infinite cones whose intervals of directions are arbitrarily close to $\left(\frac{\phi_{\infty}-\pi}{2},\frac{\phi_{\infty}+\pi}{2}\right)$ and $\left(\frac{\phi_{\infty}+\pi}{2},\frac{\phi_{\infty}+3\pi}{2}\right)$ respectively.
\end{proposition}

\subsubsection{$\deg Q -\deg P=0$}

The following statement has been proven in Corollary~6.4 of \cite{AHN+24}.

\begin{proposition}\label{prop:qp0}
Take any operator $T$ as in \eqref{eq:1stN} such that $\deg Q -\deg P=0$. Then for any $\epsilon>0$, there exists an open cone $\mathcal{C}$ whose interval of directions is $(\phi_{\infty}+\pi-\epsilon,\phi_{\infty}+\pi+\epsilon
)$ and such that $\minvset{CH}$ is contained in $\mathcal{C}$.
\end{proposition}

\section{Local analysis of the boundary of \texorpdfstring{$\minvset{CH}$}{MCH}}\label{sec:localanalysis}

We consider an operator $T$ as in \eqref{eq:1stN} whose minimal set $\minvset{CH}$ has a nonempty interior.

\begin{notation}\label{notationlocal}
For any point $\alpha \in \partial \minvset{CH}$, we define $r_{\alpha} \in \mathbb{C}^{\ast}$, $m_{\alpha} \in \mathbb{Z}$ so that
\begin{equation}\label{eq:m_alpha}
    R(z)=\frac{Q(z)}{P(z)}=r_{\alpha}(z-\alpha)^{m_{\alpha}}+o(|z-\alpha|^{m_{\alpha}}).
\end{equation}
We also define $\phi_{\alpha}=arg(r_{\alpha})$ and $d_{\alpha}: \mathbb{S}^{1} \longrightarrow \mathbb{S}^{1}$ where $d_{\alpha}(\theta)=\phi_{\alpha}+m_{\alpha}\theta$.
\end{notation}

\subsection{Description of a tangent cone}

\begin{definition}\label{def:calKL}
For any $\alpha \in \partial \minvset{CH}$, we define $\mathcal{K}_{\alpha}$ as the subset of $\mathbb{S}^{1}$ formed by directions $\theta$ such that there is a sequence $(z_{n})_{n \in \mathbb{N}}$ satisfying the following conditions:
\begin{itemize}
    \item for any $n \in \mathbb{N}$, $z_{n} \in (\minvset{CH})^{c}$;
    \item $z_{n} \longrightarrow \alpha$;
    \item $arg(z_{n}-\alpha) \longrightarrow \theta$.
\end{itemize}
We also define $\mathcal{L}_{\alpha}$ as the subset of $\mathbb{S}^{1}$ formed by directions $\theta$ such that the half-line $\alpha+e^{i\theta}\mathbb{R}^{+}$ does not intersect the interior of $\minvset{CH}$.
\end{definition}

\begin{lemma}\label{lem:KLalpha}
For any $\alpha \in \partial \minvset{CH}$, the following statements hold:
\begin{enumerate}
    \item $\mathcal{K}_{\alpha}$ and $\mathcal{L}_{\alpha}$ are nonempty closed subsets of $\mathbb{S}^{1}$;
    \item $\mathcal{L}_{\alpha} \subset \mathcal{K}_{\alpha}$;
    \item for any $\theta \in \mathcal{K}_{\alpha}$, $d(\theta) \in \mathcal{L}_{\alpha}$. In particular, 
    $K_\alpha$ is invariant under $d_\alpha$;
    \item for any $\theta \in \mathcal{K}_{\alpha}$, there exists a closed interval $J \subset \mathcal{K}_{\alpha}$ of length at most $\pi$ containing both $\theta$ and $d_{\alpha}(\theta)$;
    \item $\mathcal{L}_{\alpha} \neq \mathbb{S}^{1}$.
\end{enumerate}
\end{lemma}

\begin{proof}
 From Definition~\ref{def:calKL} it  immediately follows that $\mathcal{K}_{\alpha}$ and $\mathcal{L}_{\alpha}$ are closed subsets of $\mathbb{S}^{1}$.
\par
If $\alpha \in \partial \minvset{CH}$, then we can find a sequence of points in the complement of $\minvset{CH}$ approaching $\alpha$. By compactness of $\mathbb{S}^{1}$, we can choose a subsequence for which the arguments converge to some limit. Thus $\mathcal{K}_{\alpha}$ is nonempty.
\par
Then, for any $\theta \in \mathcal{K}_{\alpha}$, we have a sequence $(z_{n})_{n \in \mathbb{N}}$ in the complement of $\minvset{CH}$ accumulating to  $\alpha$ with the limit slope $\theta$. The associated rays $r(z_{n})$ accumulate to $\alpha+e^{id_{\alpha}(\theta)}\mathbb{R}^{+}$. Since none of them intersects the interior of $\minvset{CH}$, the half-line $\alpha+e^{id_{\alpha}(\theta)}\mathbb{R}^{+}$ does not intersect it either and $d_{\alpha}(\theta) \in \mathcal{L}_{\alpha}$.
\par
Besides, in the case where $\theta \neq d_{\alpha}(\theta)$, (up to taking a subsequence of $(z_{n})_{n \in \mathbb{N}}$, there is a closed interval $J \subset \mathbb{S}^{1}$ such that:
\begin{itemize}
    \item the endpoints of $J$ are $\theta$ and $d_{\alpha}(\theta)$;
    \item the length of $J$ is at most $\pi$;
    \item for any $\eta$ in the interior of $J$, there is a bound $N(\eta)$ such that for any $n \geq N(\eta)$, the associated ray $r(z_{n})$ intersects the half-line $\alpha + e^{i\eta}\mathbb{R}^{+}$ at  some point $P_{\eta,n}$.
\end{itemize}
Existence of sequences $(P_{\eta,n})_{n \geq N(\eta)}$ proves that for any $\eta \in J$, one has  $\eta \in \mathcal{K}_{\alpha}$.
\par
Finally, $\mathcal{L}_{\alpha} \neq \mathbb{S}^{1}$ because in this case, $\minvset{CH}$ would have empty interior.
\end{proof}
\begin{remark}
    Note that in the case $\theta=d_\alpha(\theta)$, the interval $J$ is a singleton $\{\theta\}.$
\end{remark}

Let us deduce local description of $\mathcal{K}_{\alpha}$ and $\mathcal{L}_{\alpha}$ depending on the local invariants of $\alpha$.

\begin{corollary}\label{cor:KLalpha}
For any $\alpha \in \partial \minvset{CH}$, the following statements hold:
\begin{itemize}
    \item if $|m_{\alpha}| \geq 2$, then $\mathcal{K}_{\alpha}=\mathcal{L}_{\alpha}$ and they are contained in the finite set of arguments satisfying $\theta \equiv \frac{\phi_{\alpha}}{1-m_{\alpha}}~[\frac{2\pi}{1-m_{\alpha}}]$;
    \item if $m_{\alpha}=1$, then $\phi_{\alpha}=0$ and $\mathcal{K}_{\alpha}=\mathcal{L}_{\alpha}$;
    \item if $m_{\alpha}=0$, then $\phi_{\alpha} \in \mathcal{L}_{\alpha}$;
    \item if $m_{\alpha}=-1$, then $\mathcal{K}_{\alpha}=\mathcal{L}_{\alpha}$ and these sets are formed by at most two intervals, each of length at most $\pi$ and having their  midpoints at $\frac{\phi_{\alpha}}{2}$ and $\frac{\phi_{\alpha}}{2}+\pi$.
\end{itemize}

\end{corollary}

\begin{proof}
We consider  maximal interval $J$ in  $\mathcal{K}_{\alpha}$ (which is non-empty by Lemma~\ref{lem:KLalpha}). The images of $J$ under the iterated action of $d_{\alpha}$ belong to $\mathcal{L}_{\alpha}$.

\par
If $|m_{\alpha}| \geq 2$, then $J$ is a singleton since otherwise the union of its iterates would coincide with $\mathbb{S}^{1}$ (contradicting Lemma~\ref{lem:KLalpha}). Thus $J$ has to be a fixed point of the map $d_{\alpha}$.
\par
If $m_{\alpha}=1$ and $\phi_{\alpha} \neq 0$, then $J$ coincides with $\mathbb{S}^{1}$ because no other connected subset of the circle is preserved under the action of  nontrivial rotation. Therefore $d_{\alpha}$ is the identity map.
\par
If $m_{\alpha}=0$, then for any $\theta \in \mathcal{K}_{\alpha}$, $d_{\alpha}(\theta)=\phi_{\alpha}$. Therefore $\phi_{\alpha} \in \mathcal{L}_{\alpha}$.
\par
If $m_{\alpha}=-1$, then $J$ is invariant under the action of $\theta \mapsto \phi_{\alpha} - \theta$. Thus, either $\frac{\phi_{\alpha}}{2}$ or $\frac{\phi_{\alpha}}{2}+\pi$ is the bisector of $J$. If $J$ is of length strictly bigger than $\pi$, then Lemma~\ref{lem:KLalpha} shows that its complement (of length strictly smaller than $\pi$) is also contained in $\mathcal{L}_{\alpha}$. Therefore $L_{\alpha} = \mathbb{S}^{1}$ which is a contradiction.
\end{proof}

We obtain a bound on the number of petals of $\minvset{CH}$ that can be attached to a boundary point.

\begin{corollary}\label{cor:BOUNDINTERIOR}
For any linear differential operator $T$ given by \eqref{eq:1stN} we consider a point $\alpha$ of the boundary $\partial \minvset{CH}$. Then for any sufficiently small open subset $V \subset \mathbb{C}$ the interior of the connected component  $M_V$ of $V\cap \minvset{CH}$ containing $\alpha$ has at most:
\begin{itemize}
    \item $|1-m_{\alpha}|$ connected components if $m_{\alpha} \neq 1$;
    \item $\deg P$ connected components if $m_{\alpha} = 1$
\end{itemize}
where $R(z)=\lambda (z-\alpha)^{m_{\alpha}}+o((z-\alpha)^{m_{\alpha}})$ with $\lambda \in \mathbb{C}^{\ast}$ and $m_{\alpha} \in \mathbb{Z}$.
\end{corollary}

\begin{proof}
If $\alpha$ is not a zero or a pole of $R(z)$, then Lemma~\ref{lem:ImKleinen} proves the statement. Besides, if $m_{\alpha} \notin \lbrace{ 0,1 \rbrace}$, Corollary~\ref{cor:KLalpha} proves that $\alpha$ is in the closure of at most $1-m_{\alpha}$ components.
\par
In the remaining cases, $\alpha$ is a simple zero of $R(z)$. If $\alpha$ is also a root of degree $d$ of $P$, then it is a root of degree $d+1$ of $Q$.
    
We can divide $P$ and $Q$ by $(z-\alpha)^{d}$ while keeping the same minimal set $\minvset{CH}$ (because in this case $\mathcal{Z}(PQ)$ remains unchanged). Consequently, we can assume that $\alpha$ is not a root of $P$. Lemma~\ref{lem:BoundInterior} proves that for any connected component $C$ of $(\minvset{CH})^{\circ}$ such that $\alpha$ is in the closure of $C$, either some root of $P(z)$ belongs to the closure of $C$ or some tail is attached to $C$. If $\alpha$ is in the closure of several connected components of $(\minvset{CH})^{\circ}$, then a same root of $P$ cannot be in the closure of two of them because $\overline{\minvset{CH}}$ would fail to be contractible. Similarly a given tail is attached to only one connected component of $(\minvset{CH})^{\circ}$ (and contains at least one root of $P$). Therefore, $\alpha$ is in the closure of at most $\deg P$ components.
\end{proof}

\subsection{Curve of inflections}\label{sub:inflection}

In \S A.3 of \cite{AHN+24} we introduced the curve of inflections $\mathfrak{I}_{R}$ of an analytic vector field $R(z)\partial_{z}$. By definition, it is the closure in $\bC$ of the subset of $\bC\setminus \zeros(PQ)$ at each point of which the integral curve of the vector field $R(z)\partial_{z}$ passing through this point has zero curvature. Here and throughout, $\zeros(F)$ denotes the set of zeros of the function $F$. Below we provide some additional information about $\mathfrak{I}_{R}$. 

For an operator $T$ for which $R(z)$ is not of the form $\lambda$ or $\lambda(z-\alpha)$ for some $\lambda \in \mathbb{C}^{\ast}$ and $\alpha \in \mathbb{C}$, the function $R'(z)$ is a non-constant rational function. Therefore the curve of inflections  $\mathfrak{I}_{R}$ of $R(z)\partial_z$ (which is defined as the closure of the set of points for which $Im(R'(z))=0$) is a real plane algebraic curve. 

We first characterize the points at which several local branches of the curve of inflections intersect.

\begin{lemma}\label{lem:InflectionBranches}
A point $z_{0} \in \mathfrak{I}_{R}$ belongs to exactly $m \geq 2$ local branches of $\mathfrak{I}_{R}$ in the following cases:
\begin{enumerate}
    \item $z_{0}$ is a critical point of $R'(z)$ of order $m-1$  (including zeroes of order $m$ of $R(z)$);
    \item $z_{0}$ is a pole  of $R(z)$ of order $m-1$.
\end{enumerate}
The $2m$ limit slopes of the local branches at $z_0$ form a regular $2m$-gon in $\mathbb{S}^{1}$.
\end{lemma}

\begin{proof}
This follows immediately from Lemma~\ref{lem:branches and curvature}.
\end{proof}

\begin{corollary}\label{cor:inflectionmultiple}
The curve of inflections $\mathfrak{I}_{R}$ has at most $4 \deg P +  \deg Q -2$ singular points.
\end{corollary}

\begin{proof}
There are at most $\deg P$ poles of $R(z)$ and the critical points of $R'(z)$ are the zeroes of $R''(z)$.
\end{proof}

\begin{lemma}\label{lem:IRnumber of conncomp}
    Let $F(z):\bC\to\bC P^1$ be  a non-constant rational function of degree $d$. Then the real algebraic curve  $\Gamma=\overline{\left\{z\in\bC|\Im F(z)=0\right\}}$ is non-empty, has at most $d$ connected components and has exactly $d$ connected components for generic $F$.
\end{lemma}
\begin{proof}
Clearly, as $F^{-1}(x)\not=\emptyset$ for any $x\in \mathbb{R}\setminus\{F(\infty)\}$, $\Gamma\not=\emptyset$ as well. 

By the open mapping theorem, the map $F: \bar{\Gamma}\to \mathbb{R}P^1 $,  where $\bar{\Gamma}$ is the closure of $\Gamma$ in $\mathbb{C}P^1$, is onto on each connected component of $\bar{\Gamma}$. Since $F$ has degree $d$ this means that $\Gamma$ has at most $d$ components. 

Note that the ramification  points of  $F: \bar{\Gamma}\to \mathbb{R}P^1 $ coincide with the ramification points of  $F:\mathbb{C}P^1\to \mathbb{C}P^1$ lying on $\bar{\Gamma}$. Thus if the ramification values of $F$ are not in $\mathbb{R}P^1$ then the former map is an unramified cover of degree $d$, so has exactly $d$ connected components. This means that the bound is sharp.
\end{proof}

\subsubsection{Inflection domains}

\begin{definition}
The curve of inflections $\mathfrak{I}_{R}$ subdivides $\mathbb{C}$ into two open (not necessarily connected) domains: 
$\mathfrak{I}^{+}$ given by $Im(R'(z))>0$ and $\mathfrak{I}^{-}$ given by $Im(R'(z))<0$.
\end{definition}

Observe that in $\mathfrak{I}^{+}$ (resp. $\mathfrak{I}^{-}$), the integral curves of the vector field $R(z)\partial_{z}$ are turning counterclockwise (resp. clockwise).

\subsubsection{Circle at infinity}

Consider the closure of the curve of inflections $\mathfrak{I}_{R}$ in the extended complex plane $\mathbb{C} \cup \mathbb{S}^{1}$.

\begin{lemma}\label{lem:InflecInfinity}
The intersection $\mathfrak{I}_{R} \cap \mathbb{S}^{1}$ is:
\begin{itemize}
    \item is empty if $\deg Q - \deg P =1$ and $\lambda \notin \mathbb{R}$;
    \item coincides with the set $\lbrace{\frac{\phi_{\infty}}{2},\frac{\phi_{\infty}}{2}+\frac{\pi}{2},\frac{\phi_{\infty}}{2}+\pi,
    \frac{\phi_{\infty}}{2}+\frac{3\pi}{2}\rbrace}$ if $\deg Q - \deg P =-1$.
\end{itemize} 
 In the remaining cases:
\begin{itemize}
    \item $\deg Q - \deg P =1$ and $\lambda \in \mathbb{R}$; or
    \item $\deg Q - \deg P =0$; or
    \item $\deg Q-\deg P\not\in\{-1,0,1\}$
\end{itemize}
the set $\mathfrak{I}_{R} \cap \mathbb{S}^{1}$ consists of $2k$ points forming a regular $2k$-gon for some $k$ satisfying $k\leq  \max\{\deg P, \deg Q\}+1$.
\end{lemma}

\begin{proof}
If $k=\deg Q - \deg P \in\bZ\setminus\{0,1\}$, then  $R'(z)$ has an expansion  of the form ${k\lambda_kz^{k-1}}+o(z^{k-1})$ near $\infty$ from which the characterization of the infinite branches of the real locus of $R'(z)$ follows by Lemma~\ref{lem:branches and curvature} applied to either $R'(z)$ or to $\tfrac{1} {R'(z)}$ depending on whether $k>0$ or $k<0$ (clearly both have the same real locus outside their poles).

If $\deg Q - \deg P =0$, then $R(z)$ has an expansion  $\lambda+\frac{A}{z^{k}}+o(z^{-k})$ for some $A \in \mathbb{C}^{\ast}$ and $k \in \mathbb{N}^{\ast}$ near $\infty$. (The case when $R(z)$ is constant is ruled out by the genericity assumptions). Therefore $R'(z)$ has an expansion $-\frac{Ak}{z^{k+1}}+o(z^{-k-1})$. We conclude that $\mathfrak{I}_{R}$ has $2k$ infinite branches whose limit directions form a regular $2k$-gon.

If $\deg Q - \deg P =1$, then  $R(z)$ has an expansion $\lambda z + A + Bz^{-k}+o(z^{-k})$ for some $A \in \mathbb{C}$, $B \in \mathbb{C}^{\ast}$, and $k \in \mathbb{N}^{\ast}$. (The case when $R(z)$ is a linear function is ruled out by the genericity assumptions). We obtain that $R'(z)$ is of the form $\lambda-\frac{Bk}{z^{k+1}}+o(z^{-k-1})$. Consequently, unless $\lambda$ is real, the curve of inflections $\mathfrak{I}_{R}$ is compact in $\bC$. If $\lambda$ is real, the infinite branches of $\mathfrak{I}_{R}$ are asymptotically the same as that of the real locus of $-\frac{kB}{z^{k+1}}$. Therefore  $\mathfrak{I}_{R}$ has $2k$ infinite branches whose limit directions form a regular $2k$-gon.

In these last two cases, we have $R'(z)=\frac{M}{z^{k+1}}+o(z^{-k-1})$ for some $M \in \mathbb{C}^{\ast}$ and $k \geq 1$. 
The number $k$ is the ramification index of either $R-\lambda z$ (for $\deg Q-\deg P=1$) or $R$ (for $\deg Q-\deg P=0$) at infinity, thus $k$ cannot be bigger than the degree $\max\{\deg P, \deg Q\}$ of $R$. Therefore $k\leq \max\{\deg P, \deg Q\}+1$.

\end{proof}

\subsubsection{Singularities of the vector field}

Next we deduce from Corollary~\ref{cor:KLalpha} a proof of the statement that any root of $P(z)$ or $Q(z)$ belonging to $\partial \minvset{CH}$ automatically belongs to the curve of inflections.

\begin{corollary}\label{cor:SingInflec}
Consider an operator $T$ as in \eqref{eq:1stN} such that $\minvset{CH}$ does not coincide with $\mathbb{C}$ and has a nonempty interior. Let $\alpha$ be a zero or a pole of $R(z)$ such that $\alpha \in \partial \minvset{CH}$. Then $\alpha$ also belongs to the curve of inflections $\mathfrak{I}_{R}$. Additionally, the number of local branches of $\mathfrak{I}_{R}$ at $\alpha$ equals:
\begin{itemize}
    \item $a+1$ if $\alpha$ is a pole of order $a \geq 1$;
    \item $a-1$ if $\alpha$ is a zero of order $a \geq 2$;
    \item some integer  $b \geq 1$ if $\alpha$ is a simple zero.
\end{itemize}
\end{corollary}

\begin{proof}
The statement is proved by direct computation of $Im(R')$ in  case of a pole or a zero of order $a \geq 2$. If $\alpha$ is a simple zero of $R(z)$, then we have $R(\alpha+\epsilon)= R'(\alpha)\epsilon+o(\epsilon)$. If $\alpha \in \partial \minvset{CH}$, then $\phi_{\alpha}=arg(R'(\alpha))=0$ (see Corollary~\ref{cor:KLalpha}). Thus $\alpha \in \mathfrak{I}_{R}$.
\par
Unless $R(z)$ is linear, $R(z)$ is of the form $R'(\alpha)(z-\alpha)+M(z-\alpha)^{d}+o(|z-\alpha|^{d})$ for some $d \geq 2$ and $M \in \mathbb{C}^{\ast}$. Thus $R'(z)=R'(\alpha)+Md(z-\alpha)^{d-1}+o(|z-\alpha|^{d-1})$. Consequently, the number of local branches of the equation $Im(R')=0$ equals $d-1$.
\par
If $R(z)=\lambda(z-\alpha)$, then $Re(\lambda) \geq 0$ (otherwise $\minvset{CH}=\mathbb{C}$) and $Im(\lambda) \neq 0$ (otherwise $\minvset{CH}$ is fully irregular). It follows that $Im(R'(z))$ is a non-vanishing constant and the curve of inflections is empty. In this case, $\mathfrak{I}_{R}$ does not contain any zero or pole of $R(z)$ on the boundary of $\minvset{CH}$.
\end{proof}

\subsubsection{Tangency locus}\label{sub:tangency} 

\begin{definition}\label{defn:tangencylocus}
For the rational vector field $R(z)\partial_{z}$, the \textit{tangency locus} $\mathfrak{T}_{R}$ is the subset of the curve of inflections $\mathfrak{I}_{R}$ where $R(z)\partial_{z}$ is tangent to some branch of $\mathfrak{I}_{R}$.
\end{definition}

\begin{proposition}\label{prop:tangencybound}
For an operator $T$ as in \eqref{eq:1stN}, the \textit{tangency locus} $\mathfrak{T}_{R}$ is the union of:
\begin{itemize}
    \item at most $\max\{\deg Q,\deg P\}+1$ lines; and
    \item at most $2(3\deg P + \deg Q -1)^{2}$ points.
\end{itemize}
\end{proposition}

\begin{proof}
For any point $z \in \mathcal{T}_{R}$, an immediate computation involving the Taylor expansion of $R'(z)$ proves that $z$ belongs to the intersection of the curve of inflections (given by $\Im(R')=0$) with a real plane algebraic curve given by the equation $\Im(R''R)=0$. Indeed, the tangent line to $\mathfrak{I}_R$ at some $z_0\in\mathfrak{I}_R$ is given by the equation $R''(z_{0})\cdot(z-z_0)\in\mathbb{R}$, and the associated ray direction is $R(z_0)$. The degrees of these two curves are respectively $\deg Q + 3\deg P -1$ and $2\deg Q + 6\deg P -2$. Therefore, Bézout's theorem implies  that $\mathfrak{T}_{R}$ contains at most $2(\deg Q + 3\deg P -1)^{2}$ such points and some irreducible components corresponding to the common factors of the two equations.
\par
By definition of the tangency locus these irreducible components are the integral curves of $R(z)\partial_{z}$ contained in the curve of inflections. Such integral curves have identically vanishing  curvature and therefore  they are segments of straight lines. Therefore the relevant irreducible components are straight lines. 
 But $\mathfrak{I}_R$ intersects $\mathbb{S}^1$ at most $2\max\{\deg Q,\deg P\}+2$ points by Lemma~\ref{lem:InflecInfinity}. Thus the number of the lines is at most $\max\{\deg Q,\deg P\}+1$.
\end{proof}

We deduce an estimate on the number of connected components of the transverse locus $\mathfrak{I}_{R}^{\ast}$ of the curve of inflections. Denote $d=3\deg P+\deg Q-1=\deg \mathfrak{I}_R$.

\begin{corollary}\label{cor:transverseCOMPONENT}
For an operator $T$ as in \eqref{eq:1stN}, the \textit{transverse locus} $\mathfrak{I}_{R}^{\ast}$ of the curve of inflections is formed by at most $2d^{2}+6d+2$ connected components.

\end{corollary}

\begin{proof}
A connected component of $\mathfrak{I}_{R}^{\ast}$ is either a smooth closed loop (so a connected component of $\mathfrak{I}_{R}$) or an arc joining points at infinity, singular points of $\mathfrak{I}_{R}$ or isolated points of the tangency locus.
\par
Following Proposition~\ref{prop:tangencybound}, the tangent locus contains at most $2d^{2}$ isolated points. Each of them is the endpoint of two arcs of the transverse locus.
\par
Lemma~\ref{lem:InflecInfinity} proves that at most $2\max\{\deg P,\deg Q\}+2$ arcs of the transverse locus go to infinity.
\par
Lemma~\ref{lem:InflectionBranches} provides the analog result for the multiple points of the curve of inflections. In the "worst" case, poles of $R(z)$ and critical points of $R'(z)$ are simple. At most four arcs of the transverse locus are incident to such points. There are at most $4 \deg P +  \deg Q -2 \leq 2d$ such points (see Corollary~\ref{cor:inflectionmultiple}) so they are incident to at most $4d$ arcs.
\par
Adding these bounds, we obtain an upper bound $4d^{2}+10d+4$ on the number of ends of non-compact connected components of the transverse locus, i.e. there are at most $2d^{2}+5d+2$ non-compact connected components. By Lemma ~\ref{lem:IRnumber of conncomp} the number of the compact connected components (loops) of  $\mathfrak{I}_{R}$ is at most $d$, which gives the required upper bound. 
\end{proof}

\begin{corollary}\label{cor:R''RInflection}
On each connected component of the transverse locus $\mathfrak{I}_{R}^{\ast}$, the sign of $\Im(R''R)$ remains constant. If $\Im(R''R)$ is positive (resp. negative), then for any point $z$ of the component, the associated ray $r(z)$ points towards $\mathfrak{I}^{+}$ (resp. $\mathfrak{I}^{-}$).
\end{corollary}

\begin{proof}
Any regular point $z$ of the curve of inflections satisfying $\Im(R''(z)R(z))=0$ belongs to the tangency locus (see the proof of Proposition~\ref{prop:tangencybound}). A direct computation proves the rest of the claim.
\end{proof}

\subsection{Horns}\label{sub:horns}

In this section, we introduce some curvilinear triangles called \emph{horns} and find conditions under which we can conclude that they do not belong to the minimal set $\minvset{CH}$. Our aim is to prove that some parts of the boundary of the minimal sets are portions of integral curves of the vector field $R(z)\partial_{z}$.

\subsubsection{Definitions}	

Recall that $\sigma(q)$ is the argument of $R(q)$, i.e. $\sigma(q)=\Im \log R(q)$ and   $r(q)=q+R(q)\R_+$ is the  associated ray.
 
\begin{definition}\label{def:horn}
Assume that a segment  $\gamma_p^{p'}$ of the positive trajectory of  $R(z)\partial_z$ starting at $p\notin\Z(PQ)$ and ending at $p'$ doesn't intersect the  curve of inflections except possibly at $p$. Assume that the total variation of $\sigma$ along $\gamma_p^{p'}$ is less than  $\pi/2$.
\par     
We define the \emph{horn} $\horn{p}{p'}{p''}$ at $p$ as an open curvilinear triangle formed by $\gamma_p^{p'}$ and tangents to this trajectory at $p$ and $p'$ intersecting at a point $p''$. 
\end{definition} 

%To simplify notations we further assume that the ray $r(p)$ coincides with  $\R_+$.
\begin{definition}\label{def:small horn}
    A horn $\horn{p}{p'}{p''}$ is called \emph{small positive} (resp. \emph{small negative}) if 
\begin{enumerate}
    \item for any point $u\in \horn{p}{p'}{p''}$, the argument $\sigma(u+tR(u))$  is monotone increasing (resp. decreasing) in the variable $t$ as long as  $t\ge 0$ and $u+tR(u)\in \horn{p}{p'}{p''}$
    \item for any two points $u,v\in \horn{p}{p'}{p''}$, the scalar product  $\left(R(u),R(v)\right)$ is positive.
\end{enumerate} 
A horn $\horn{p}{p'}{p''}$ is called $\emph{small}$ if it is either small positive or small negative.
\end{definition}

\begin{remark}
A small positive horn becomes a small negative one after conjugation, i.e. after replacing $R(z)$ with $\overline{R(\bar{z})}$. Indeed,   
$$
\left(R(u),R(v)\right) = \Re R(u)\overline{R(v)}
$$ remains the same after the conjugation, and  
$$
\frac{d \sigma(u+tR(u))}{dt}(t)=\Im \frac{R'(u+tR(u))}{R(u+tR(u))}R(u)
$$ changes sign.
\end{remark}

\begin{lemma}\label{lem:intersecthornCI}
The  curve of inflections (given by $\Im R'=0$) does not intersect small horns.
\end{lemma}

\begin{proof}
We have that $\frac{d\sigma(u+tR(u)}{dt}\vert_{t=0}=\Im R'(u)\ge 0$. Assume that we have the equality at some $u\in \horn{p}{p'}{p''}$. Since $\horn{p}{p'}{p''}$ is open and $R'$ is an open map, this assumption will imply that $\frac{d\sigma(u+tR(u)}{dt}\vert_{t=0}$ changes sign in $\horn{p}{p'}{p''}$, which contradicts  the smallness assumptions.
\end{proof}

We define the cone complementary to $\horn{p}{p'}{p''}$ (in short, the \emph{complementary cone}) to be the open cone $\horninf{p''}$ with the apex $p''$  bounded by part of the  ray $r(p)$ starting at $p''$ and by the ray extending the segment $p'p''$.

\begin{lemma}\label{lem:horn out}
Consider a point $p$ which neither   belongs to $\mathcal{Z}(PQ)$ nor to  the interior of $\minvset{CH}$. Assume that the integral curve $\gamma$ of the vector  field $R(z)\partial_{z}$ containing $p$ is not a straight line. Then there exists a horn $\horn{p}{p'}{p''}$ such that both $\horn{p}{p'}{p''}$ and its complementary cone $\horninf{p''}$ do not intersect $\minvset{CH}$. 
\end{lemma}

\begin{proof}
Let  $D(p)=\{|z-p|<\delta_{\mathcal{Z}}(p)=\frac 1 2 \operatorname{dist}(p,\mathcal{Z}(PQ))\}$.

First, assume that $p\notin\minvset{CH}$. Then by definition, $r(p)\subset{\minvset{CH}^c}$. 

Choose some  $\delta>0$ and define $p_0=p$ and 
$$
p_{i}=p_{i-1}+\delta R(p_{i-1})\in r(p_i)\subset(\minvset{CH})^c\cap D(p), \quad i=1,...,N=N(\delta)=O\left(\frac{\delta_{\mathcal{Z}}(p)}{\delta}\right),
$$  
(we stop when $p_{N+1}\notin D(p)$).

The broken line $\hat{\gamma}_p^{p_N}=\cup_{i=1}^{N}[p_{i-1}, p_{i}]\subset(\minvset{CH})^c\cap D(p)$ is the Euler approximation to the positive trajectory  $\gamma_p^+$ of $R(z)\partial_z$ starting from $p$ and converges to it (more exact, to the connected component $\gamma_p^{p'}\subset D(p)$ of $\gamma_p^+\cap D(p)$ containing $p$)  as $\delta\to0$. Thus $
    \gamma_p^{p'}\subset\overline{(\minvset{CH})^c}.
$
Repeating this argument for all $\tilde{p}\notin\minvset{CH}$ sufficiently close to $p$ we see that 
\begin{equation}\label{eq:lem horn out}
    \gamma_p^{p'}\subset\left(\overline{(\minvset{CH})^c}\right)^o=(\minvset{CH})^c.
\end{equation}

If $\gamma_p^{p'}$ is a subset of the curve of inflections then it is a part of a straight line, which is excluded by our assumption.  Thus we can assume that for $p'$ sufficiently close to $p$ the curve $\gamma_p^{p'}$ intersects the curve of inflections only at $p$. Therefore  $\gamma_p^{p'}$ is convex and, choosing $p'$ closer to $p$  if needed, we can assume that $\gamma_p^{p'}$ is of angle smaller than $\pi$. Therefore
\begin{equation}\label{eq:lem horn out 2}
\left(\bigcup\nolimits_{s\in\gamma_p^{p'}}r(s)\right)^\circ=\horninf{p''}\,{\textstyle \bigcup }\,\horn{p}{p'}{p''}\subset(\minvset{CH})^c.
\end{equation}

Second, assume that $p\in\partial\minvset{CH}$ and let $\gamma_{p}^{p'}$ be a part of the connected piece of $\gamma_p^+\cap D(p)$ containing $p$ such that $\gamma_p^{p'}$ is convex and  of angle smaller than $\pi/2$. Let $p_i\notin\minvset{CH}$ be a sequence of points tending to $p$ and take $p_i'\in \gamma_{p_i}^+$ such that  $\gamma_{p_i}^{p_i'}$ converges to $\gamma_{p}^{p'}$. By analyticity  this convergence is uniform in $C^1$ sense as well. Therefore
\begin{equation}\label{eq:lem horn out 3}
  \left(\bigcup\nolimits_{s\in\gamma_p^{p'}}r(s)\right)^\circ \subset
  \bigcup_i\left(\bigcup\nolimits_{s\in\gamma_{p_i}^{p_i'}}r(s)\right)^\circ \subset
  (\minvset{CH})^c,
\end{equation}
which finishes the proof.

\end{proof}

\subsubsection{Small horns exist}

\begin{proposition}\label{prop:small horn exists}
    For any point $p\notin \Z(PQ)$ such that the trajectory $\gamma(p)$ of $R$ starting at $p$ is not a straight line,  there exists a small horn $\horn{p}{p'}{p''}$. 
\end{proposition}

\begin{proof}

Using an affine change of variables we can assume that $p=0$ and  $R(0)=1$. By assumption $R(z)$ is not a real rational function.
Let 
\begin{equation}\label{eq:R Taylor}
    R(u)=1+\rho(u) + ibu^m + O(u^{m+1}),\quad b>0, \, \rho\in\R[u],\, m\ge 1
\end{equation}
be the Taylor expansion  of $R(z)$ at $0$ (the case $m=1$ is covered by Lemma~\ref{lem:unif horns}). Here we can assume that $b>0$ by replacing $R(z)$ by $\overline{R(\bar{z})}$, if necessary.

First, we consider the case $m=1$, i.e. $p\notin \mathfrak{I}_{R}$. 

\begin{lemma}\label{lem:unif horns}
For every compact set $K$ not intersecting the  curve of inflections $\mathfrak{I}_{R}$, there is a $\delta=\delta(K)>0$ such that for every $p\in K$,  there is a small horn $\horn{p}{p'}{p''}$ of  diameter greater than $\delta$.
\end{lemma}

\begin{proof} Indeed, for any $p\in K$ the function $\Re R(u)\overline{R(v)}$ is positive and $\Im \frac{R'(u)}{R(u)}R(v)$ is non-zero at $(p,p)\in \mathbb{C}^2$, so this remains true for all $(u,v)\in \mathbb{C}^2$ such that  $\operatorname{dist}((p,p), (u,v))<\delta=\delta(p)$ by continuity.
This means that any $\horn{p}{p'}{p''}\subset U_{\delta(p)}(p)$  is a small horn. The uniform lower bound follows from the continuity of  $\delta(p)$. \end{proof}

From now on we assume that $m\ge 2$. Our next goal is to find the asymptotics  of $\gamma_0$ near $0$ and the $\horn{0}{p'}{p''}$. We abuse notation by writing the germ of $\gamma_0$ as $\gamma_0=\{x+i\gamma_0(x), x>0\}$.
\begin{lemma}\label{lem:asympt of traj at z in I_R}
\begin{equation}\label{eq:gamma  in m}
\gamma_0(x)=\tfrac{b}{m+1}x^{m+1}+O(x^{m+2})
\end{equation}
    and 
\begin{equation}\label{eq:horn at 0 asymptotics}
    \horn{0}{p'}{p''}\subset\left\{0<x<\epsilon, 0<y<\gamma_0(x)\right\}.
\end{equation}

\end{lemma}

\begin{proof}

Note that 
\begin{equation}\label{eq:1st integral of R}
    \gamma_0\subset\{\Im F=0\},\text{   where   } F'=\frac 1 R,F(0)=0.
\end{equation}
Indeed, 
$$\tfrac{d}{dt}\Im F(\gamma_0(t))=\Im \tfrac{d}{dt}F(\gamma_0(t))=\Im \left(F'\cdot\dot{\gamma_0}(t)\right)=0.$$
Now, 
\begin{equation}\label{eq:1 over R}
\frac{1}{R}=\frac{1}{1+\rho(u)}-\frac{ibu^m}{(1+\rho(u))^2}+O(u^{m+1}), 
\end{equation}

so 
$$
F(u)=u+\tilde{\rho}(u)-i\tfrac{b}{m+1}u^{m+1}+O(u^{m+2}),\quad \tilde{\rho}\in \R[u].
$$
For $u=x+iy$ we get 
$$
\Im F(u)=y(1+o(1))-\tfrac{b}{m+1}x^{m+1}+O(u^{m+2}).
$$ 
Recalling that $\gamma_0$ is tangent to the real axis, we have $y=o(x)$. 
Therefore 
\begin{equation*}
    \{\Im F=0\}=\left\{x+iy:y=\tfrac{b}{m+1}x^{m+1}+O(x^{m+2})\right\}\subset\{y\ge 0\}
\end{equation*} 
near the origin, and the claim of the Lemma follows since $r(0)=\R_+$.
\end{proof}

 Next, we have to check the two conditions in  Definition~\ref{def:small horn} for $\horn{0}{p'}{p''}$ with $p'$ sufficiently close to 0. The second condition is easy: 
since $R(0)=1$ then the scalar product $\left(R(u), R(v)\right)$ is positive for all $u,v\in \horn{0}{p'}{p''}$ by continuity.

To check the first condition set $u=x_1+iy_1, v=x_2+iy_2=u+tR(u)\in\horn{0}{p'}{p''}$ with  $t>0$. By the second property of the small horns, we have $x_2>x_1$. By \eqref{eq:horn at 0 asymptotics} we have  $y_i=O\left(x_i^{m+1}\right)$. Combining  \eqref{eq:1 over R} and 
\begin{equation}\label{eq:R prime}
     R'(v) =\rho'(v) + imbv^{m-1}+O(v^{m}), 
      \end{equation}  
    we get 
    \begin{align*}
        \frac{R'(v)}{R(v)}&=\left(\rho'(x_2)+imbx_2^{m-1}+O(x_2^{m})\right)\frac{1+\rho(x_2)-ibx_2^m+O(x_2^{m+1})}{\left(1+\rho(x_1)\right)^2}\\
                            &=\frac{\rho'(x_2)\left(1+\rho(x_2)\right) + imbx_2^{m-1}+O(x_2^{m})}{\left(1+\rho(x_2)\right)^2}.
    \end{align*}
     
 Thus, using \eqref{eq:R Taylor}, we get for $\Phi(u,v)=\left(1+\rho(x_2)\right)^2\Im R(u)\frac{R'(v)}{R(v)}$ the equation 
\begin{align} \label{eq:small horn exists bound}
\Phi&=\Im\left(\left[1+\rho(x_1)+ibx_1^m+O(x_1^{m+1})\right] \cdot\left[\rho'(x_2)\left(1+\rho(x_2)\right) + imbx_2^{m-1}+O(x_2^{m})\right]\right)\nonumber\\
& =mbx_2^{m-1}+O(x_2^{m})>0,
\end{align}
where we use $x_1\le x_2$. This proves the first requirement of Definition~\ref{def:small horn}.
\end{proof}

\begin{corollary}\label{cor:IR not between traj and r(p)}
    The germ of $\mathfrak{I}_R$ at $p$ cannot lie  between $\gamma_p^+$ and $r(p)$.
\end{corollary}
\begin{proof}
    This would mean that this germ lies inside $\horn{p}{p'}{p''}$ which is impossible by  Proposition~\ref{prop:small horn exists} and Lemma~\ref{lem:intersecthornCI}.
\end{proof}

\subsubsection{Removing small horns}

We will use the following general Lemma
\begin{lemma}\label{lem:excision}
Assume that for some open set $U\subset\mathbb{C}\setminus\Z(PQ)$ and  every point $u\in U$,    the associated ray $r(u)$ lies in the union $U\cup\left(\minvset{CH}\right)^c$. Then $\minvset{CH}\cap U=\emptyset$.  
\end{lemma}
\begin{proof}
    Indeed, if not then $\minvset{CH}\setminus U\subsetneq\minvset{CH}$ will be again invariant, which  contradicts minimality of $\minvset{CH}$.  
\end{proof}

The crucial property of small horns is the following Lemma.
		\begin{lemma}\label{lem:horn+min}
			For any $v\in\horn{p}{p'}{p''}$, one has $r(v)\subset\horn{p}{p'}{p''}\cup\horninf{p''}$. 
		\end{lemma} 
  \begin{proof}
		We prove the statement  assuming that the small horn $\horn{p}{p'}{p''}$ is positive, 
		the negative case will follow by conjugation.
 \begin{figure}
	 	\centering
	 	\includegraphics[width=0.7\linewidth]{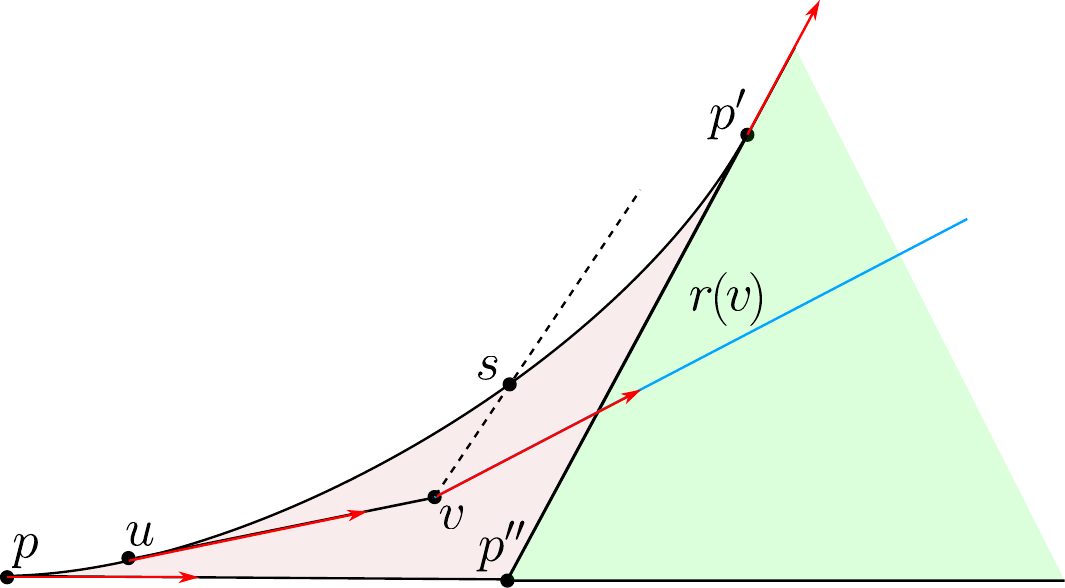}
	 	\caption{Removing small horns.}
	 	\label{fig:hornsfig0}
	 \end{figure}		
		
  Let  $u\in\gamma_p^{p'}$ be a point such that $v\in r(u)$. By definition of small horns, we have $\sigma(p)<\sigma(u)<\sigma(v)$, see Fig.~\ref{fig:hornsfig0}.

  The ray $r_{v}$  does not intersect  $\gamma_p^{p'}$. Indeed, assume that the ray $r(v)$  intersects  $\gamma_p^{p'}$ at a point $s$. Then  at the intersection point   the slope of $\gamma_p^{p'}$ should be smaller than the slope of $r(v)$, i.e. $\sigma(s)<\sigma(v)$ which contradicts the requirement that the slope is monotone increasing along the segment joining  $v$ and $s$.  
  
  Also $r(v)$ cannot intersect $pp''$ since $\sigma(v)>\sigma(u)$ and $\sigma(u)>\sigma(p)$, where $u\in\gamma_p^{p'}$ such that $v\in r(u)$.
		
		Thus $r(v)$  leaves $\horn{p}{p'}{p''}$ and enters $\horninf{p''}$ at some point of  $p''p'$ with the slope $\sigma(p)<\sigma(v)<\sigma(p')$. Thus $r(v)$ never leaves $\horninf{p''}$.
	\end{proof}
 
\begin{proposition}\label{prop:horn+min}
		Assume that  $\horn{p}{p'}{p''}$  is a small horn and  $\horninf{p''}\subset(\minvset{CH})^c$.
		Then $p$ is not in the interior of $\minvset{CH}$.
  \end{proposition}
 
	\begin{proof}
		
		Follows from Lemmas~\ref{lem:excision} and \ref{lem:horn+min}. 
  \end{proof}

\section{Boundary arcs}\label{sec:Boundaryarcs}

Recall that we consider an operator $T$ whose minimal set $\minvset{CH}$ is different from $\mathbb{C}$ and  has a nonempty interior. We want to describe its boundary in combinatorial and dynamical terms. To do this, we introduce two set-valued functions.
\par
Recall that in our terminology, $\overline{\minvset{CH}}$ is the closure of $\minvset{CH}$ in the extended plane $\mathbb{C} \cup \mathbb{S}^{1}$.

\subsection{The correspondences \texorpdfstring{$\Gamma$}{gamma} and \texorpdfstring{$\Delta$}{Delta}}\label{sub:correspondences}

\begin{definition}\label{def:DeltaGamma}
For any $x \in \partial\minvset{CH} \setminus \mathcal{Z}(PQ)$, we define:
\begin{itemize}
    \item $\Gamma(x) = \lbrace{ y \in \gamma^{+}_{x}~\vert~ y \neq x \rbrace} \cap  \overline{\minvset{CH}}$ where $\gamma^{+}_{x}$ is the positive trajectory of the vector field $R(z)\partial_{z}$ starting at $x$;
    \item $\Delta(x) = \lbrace{ y \in r(x)~\vert~ y \neq x \rbrace} \cap  \overline{\minvset{CH}}$.
\end{itemize}
\end{definition}
Note that  if $y\in \Gamma(x)$ or $y\in \Delta(x)$, and $x\in \partial \minvset{CH},$ then $y\in \partial \minvset{CH}$ as well.

Using correspondences $\Gamma$ and $\Delta$, we split the set of boundary points of $\minvset{CH}$ disjoint from the curve of inflections into the following three types.

\begin{definition}\label{defn:ArcClassification}
A point of $\partial\minvset{CH} \setminus (\mathcal{Z}(PQ) \cup \mathfrak{I}_{R})$ is a point of:
\begin{itemize}
    \item \emph{local type} if $\Gamma(z) \neq \emptyset$ and $\Delta(z)=\emptyset$;
    \item \emph{global type} if $\Gamma(z) = \emptyset$ and $\Delta(z) \neq \emptyset$;
    \item \emph{extruding type} if $\Gamma(z) \neq \emptyset$ and $\Delta(z) \neq \emptyset$.    
\end{itemize}
\end{definition}

By Proposition~\ref{prop:localarc}  these are the only possibilities for points in $\partial\minvset{CH}\setminus \mathfrak{I}_{R}$.

\subsection{Support lines}\label{sub:support}

In this section, we prove that for a given point $z$, the condition $\Delta(z) \neq \emptyset$ means that the associated ray $r(z)$ is a \textit{support line} of $\overline{\minvset{CH}}$.

For any oriented support line of $\overline{\minvset{CH}}$, we define the \textit{co-orientation} of its support  in the following way. The support point $x$ is:
\begin{itemize}
    \item a \textit{direct support point} if the standard orientation of $\partial \minvset{CH}$ and the orientation of the support line agree at $x$;
    \item an \textit{indirect support point} otherwise.
\end{itemize}

In particular, if the support line is the positively oriented real axis, a support point $x$ is called \textit{direct} if the intersection of $\minvset{CH}$ with a neighborhood of $x$ is contained in the upper half-plane (see Figure~\ref{fig:supportline} for examples of indirect support points).

\begin{definition}\label{def:Epartition}
Consider $z \in \mathbb{C}$ such that:
\begin{itemize}
    \item $z$ does not belong to the tangency locus $\mathcal{T}_{R}$ of the curve of inflections $\mathfrak{I}_{R}$;
    \item $z$ is not a root of $P$ or $Q$.
\end{itemize}
Then we say that $z \in \mathfrak{E}^{+}$ (resp. $\mathfrak{E}^{-}$) if the associated ray $r(z)$ is pointing inside the inflection domain $\mathfrak{I}^{+}$ (resp. $\mathfrak{I}^{-}$). This includes   $z\in \mathfrak{I}^{+}$ (resp. $\mathfrak{I}^{-}$).
\end{definition}
\begin{figure}[!ht]
    \centering
    \includegraphics[width=0.6\textwidth,page=2]{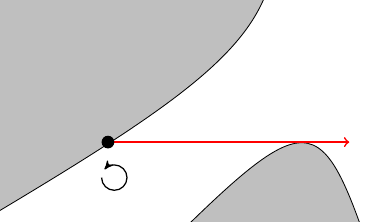}
    \caption{The point where the red arrow is tangent to $\minvset{CH}$ is an indirect support point. The circular arrow indicates that the black point belongs to $\mathfrak{E}^+$.}
    \label{fig:supportline}
\end{figure}

\begin{lemma}\label{lem:directsupport}
Consider $z \in \partial \minvset{CH} \setminus \mathcal{Z}(PQ)$ such that $z \in \mathfrak{E}^{+}$ (resp. $\mathfrak{E}^{-}$). If $y \in \Delta(z)$, then $y$ is an indirect support point (resp. a direct support point).
\end{lemma}

\begin{proof}
Without loss of generality, we can assume that  $z \in \mathfrak{E}^{+}$, $z=0$, $r(z)=\mathbb{R}_{>0}$ and $y=1$. 
This implies that $\gamma_0$ lies in the upper half-plane. By Lemma~\ref{lem:horn out} there is a neighborhood $V$ of $y$ such that $V \cap \left(\overline{\minvset{CH}}\right)^{\circ}$ is contained in the lower half-plane. Therefore $y$ is an indirect support point.
\end{proof}

\begin{lemma}\label{lem:intersectingrays}
Take $x,y \in \partial \minvset{CH}$ such that:
\begin{itemize}
    \item $x,y \in \mathfrak{E}^{-} \cup \mathfrak{E}^{+}$
    \item the associated rays $r(x)$ and $r(y)$ intersect at some point $m \in \mathbb{C}$;
    \item $\sigma(y) \in ]\sigma(x)-\pi,\sigma(x)[$.
\end{itemize}
Then the open cone $\Gamma$ with apex $m$ and the interval of directions $]\sigma(y),\sigma(x)[$ is disjoint from $\left(\minvset{CH}\right)^{\circ}$ and there are the following subcases:
\begin{itemize}
    \item either $y \in \mathfrak{E}^{+}$ or $\Delta(y) \subset [y,m]$;
    \item either $x \in \mathfrak{E}^{-}$ or $\Delta(x) \subset [x,m]$.    
\end{itemize}
\end{lemma}

\begin{proof}
The path formed by the concatenation of segments $[x,m]$ and $[m,y]$ can be approached by a family paths joining $x$ and $y$ or a family of paths joining $y$ and $x$ whose interior points are disjoint from $\minvset{CH}$. Lemma~\ref{lem:ArcCone} applies to one of these families of paths so $\Gamma$ is disjoint from $\left(\minvset{CH}\right)^{\circ}$.
\par
Then, we assume by contradiction that $y \in \mathfrak{E}^{-}$ and some point $z \in \Delta(y)$ does not belong to $[y,m]$. Since $\Gamma$ is disjoint from $\left(\minvset{CH}\right)^{\circ}$, it follows that $z$ is an indirect support point of the line containing $r(y)$ which  contradicts Lemma~\ref{lem:directsupport}. Consequently either $\Delta(y) \subset [y,m]$ or $y \notin \mathfrak{E}^{-}$.  
\par
An analogous argument proves that either $x \in \mathfrak{E}^{-}$ or $\Delta(x) \subset [x,m]$.
\end{proof}

\subsection{Local arcs}\label{sub:local}

In this section, we prove that local points (see Proposition~\ref{prop:localarc}) 
form \textit{local arcs}.

\begin{definition}\label{def:localarc}
 A \textit{local arc} of  $\partial \minvset{CH}$ is a maximal open arc of an integral curve of vector field $R(z)\partial_{z}$ that contains only local points. In particular, it is disjoint from $\mathcal{Z}(PQ)$ and $\mathfrak{I}_{R}$.
\par
Local arcs are oriented by the vector field $R(z)\partial_{z}$.
\end{definition}

Using the geometry of horns (see Section~\ref{sub:horns}), we can show that every local point actually belongs to a local arc of  $\partial \minvset{CH}$.

\begin{proposition}\label{prop:localarc}
Consider a point $p\in \partial \minvset{CH}$ and such that $\Delta(p) = \emptyset$ and $p \notin \mathcal{Z}(PQ) \cup \mathfrak{I}_{R}$. Then, the germ of the integral curve $\gamma_p$ of $R(z)\partial_{z}$ passing through $p$ belongs to  $\partial \minvset{CH}$. 
\end{proposition}

Without loss of generality, we assume that $p=0$, $r(p)=\R_+$ and $p\in\mathfrak{E}^{+}$, so $\gamma_p^{p'}$ lies in the upper half-plane. The proof consists of two steps illustrated by Figure~\ref{fig:hornsfig1} and Figure~\ref{fig:hornsfig2} respectively.

\begin{lemma}\label{lem:localarc step 1}
    $\minvset{CH}$ lies above the integral curve $\gamma_p$ of $R(z)\partial_z$ passing through $p$.
\end{lemma}
\begin{proof}[Proof: see Fig.~\ref{fig:hornsfig1}]
By Lemma~\ref{lem:horn out} there exists $p'\in\gamma_p^+$ such that the union $ \horn{p}{p'}{p''}\cup\horninf{p''}$ is outside of $\minvset{CH}$. 
Let $q'\in\gamma_p^{p'}$, $q'\neq p,p'$, and let $q''\in\R_+$ be the  intersection point of $\R_+$ with the line tangent to $\gamma_p^{p'}$ at $q'$. By Proposition~\ref{prop:small horn exists}  we can assume that 
 $\horn{p}{q'}{q''}$ is a \emph{small} horn at $p$, $\horn{p}{q'}{q''}\subset\horn{p}{p'}{p''}$. Clearly, $\sigma(q')<\sigma(p')$.

\par
The condition $\Delta(p) = \emptyset$ implies that $q''\in (\minvset{CH})^c$. Moreover, as $+\infty\notin\Delta(0)$, there is an open sector $S$ with vertex on $\R$, containing $[q'', +\infty)$ and  disjoint from $\minvset{CH}$.

For a point $\tilde{p}$ sufficiently close  to $p$ and lying below $\gamma_p$ consider a horn  $\horn{\tilde{p}}{\tilde{q}'}{\tilde{q}''}$ with vertices $\tilde{q}'$ and $\tilde{q}''$ close to $q'$ and $q''$, respectively. By continuity, the part of $r(\tilde{p})$ starting from $\tilde{q}''$ lies in $S$. Also, $\tilde{q}'$ lies in the horn $\horn{p}{p'}{p''}$, so $r(\tilde{q}')\cap\minvset{CH}=\emptyset$  by Lemma~\ref{lem:horn out}.
\par
Thus the complementary cone $\horninf{\tilde{q}''}$ of $\tilde{p}$ with vertex $\tilde{q}''$ lies outside of $ \minvset{CH}$. Therefore by Proposition~\ref{prop:horn+min} $\tilde{p}\notin\minvset{CH}^\circ$, so $\tilde{p}\in \overline{(\minvset{CH})^c}$. As this remains true for any point in a sufficiently small neighborhood of $\tilde{p}$, we conclude $\tilde{p}\in \left(\overline{(\minvset{CH})^c}\right)^\circ=(\minvset{CH})^c$. Thus near $p$ the set $ \minvset{CH}$ lies above  $\gamma_p$.
\par
\begin{figure}
	 	\centering
	 	\includegraphics[width=1\linewidth]{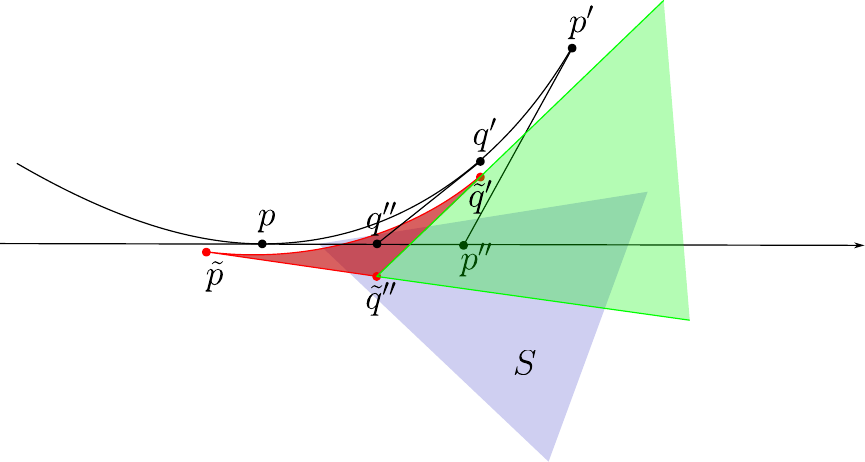}
	 	\caption{$\minvset{CH}$ lies above the trajectory $\gamma^{p'}_p$. }
	 	\label{fig:hornsfig1}
	 \end{figure}
\end{proof}

\begin{lemma}\label{lem:localarc step 2}
    The boundary $\partial \minvset{CH}$ coincides with the integral curve $\gamma_{p}$ in a neighborhood of $p$.
\end{lemma}
\begin{proof}[Proof: See Figure~\ref{fig:hornsfig2}]
Lemma~\ref{lem:localarc step 1} and its proof implies that $p$ lies on the boundary of a sector $S$ with a vertex $s\in \bR$, $s\not=p$, and disjoint from $\minvset{CH}$. 

Recall that by Lemma~\ref{lem:unif horns} there is a lower bound $\delta$ on the size of small horns for all points close to $p$.

Assume that a point $q\notin\minvset{CH}$ close to $p$ lies above $\gamma_p$ on a distance much smaller than $\delta$ and let $\horn{q}{q'}{q''}$ be its horn (necessarily small) of size $\delta/2$. Both $\horn{q}{q'}{q''}$ and $\horninf{q''}$ lie outside of $\minvset{CH}$.

\par
Let $\tilde{p}$ be a point on $\gamma_q$ close to $q$ and in the negative direction from $q$,   let $\tilde{p}''$ be the intersection of $r(\tilde{p})$ and the line $q'q''$. The horn $\horn{\tilde{p}}{q'}{\tilde{p}''}$ is smaller than $\delta$, so is small. 

The ray $\tilde{p}''q'$ lies outside of $\minvset{CH}^\circ$. Moreover, as long as the ray $p''+R(\tilde{p})\bR_+\subset r(\tilde{p})$ lies outside $\minvset{CH}^\circ$ we have  $\horninf{\tilde{p}''}\subset(\minvset{CH})^c$, so  $\tilde{p}\notin\minvset{CH}^\circ$ by Proposition~\ref{prop:horn+min}.

These arguments work for all points $\tilde{p}$ sufficiently close to $p$ and with slope $\sigma(\tilde{p})$ exceeding some negative number (namely the slope of the second side of $S$), in particular,  for points slightly above $\gamma_p^-$, the negative trajectory of $\gamma_p$. But $p$ lies in the horn of size $\delta$ of such a point,  which means that $p\notin\minvset{CH}$ by Lemma~\ref{lem:horn out}, a contradiction.
	
	\begin{figure}
		\centering
		\includegraphics[width=0.7\linewidth]{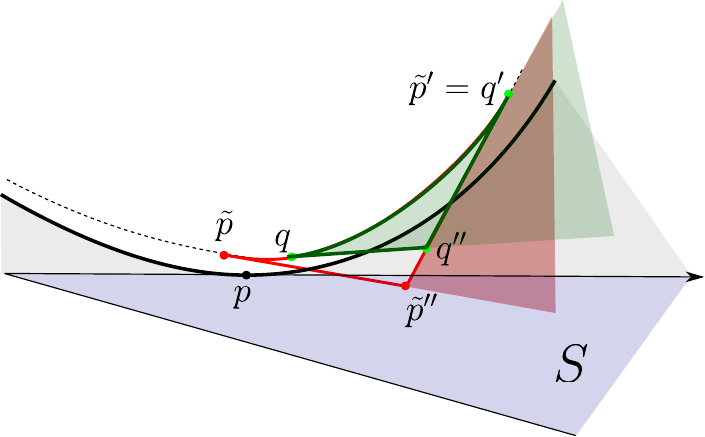}
		\caption{$\gamma_p^{p'}$ is boundary of $\minvset{CH}$}
		\label{fig:hornsfig2}
	\end{figure}
    
\end{proof}

Local analysis of horns (see Section~\ref{sub:horns}) leads to the following results about the correspondence $\Gamma$.

\begin{corollary}\label{cor:gammabasic}
Consider $z \in \partial \minvset{CH}$ such that $z \notin \mathcal{Z}(PQ) \cup \mathfrak{I}_{R}$. If $\Gamma(z) \neq \emptyset$ then $z$ is either the starting point or a point of a local arc.
\end{corollary}

\begin{proof}
We just have to prove that for some $y \in \Gamma(z)$ such that $y$ is close enough to $z$, the arc $\alpha$ of integral curve between $z$ and $y$ belongs to $\partial\minvset{CH}$. This follows from Lemma~\ref{lem:horn out}.
\end{proof}

\begin{proposition}\label{prop:localconvex}
Any local arc is a locally strictly convex real-analytic curve. Its orientation coincides with the standard topological orientation of $\partial\minvset{CH}$ if it is contained in $\mathfrak{I}^{+}$ (and with the opposite orientation otherwise).
\end{proposition}

\begin{proof}
As any integral curve of a real-analytic vector field, a local arc is a real-analytic curve in  $\mathbb{R}^{2}$. The arc has to be locally convex because otherwise, the associated ray (which is contained in the tangent line) at some point would cross the interior of $\minvset{CH}$. Besides,  direct computation shows that the curvature  of an integral curve becomes zero only at points belonging to $\mathfrak{I}_{R}$.
\end{proof}

Let us check that a local arc of $\partial\minvset{CH}$ cannot end inside an inflection domain. It cannot be periodic either. 

\begin{proposition}\label{prop:localEND}
Every local arc has an endpoint that belongs to $\mathcal{Z}(PQ) \cup \mathfrak{I}_{R}$.
\par
Besides, if such an endpoint belongs to $\mathcal{Z}(PQ)$, it is either a regular point or a simple pole of $R(z)$.
\end{proposition}

\begin{proof}
Assume that the local arc $\gamma$ is periodic and doesn't intersect $\mathcal{Z}(PQ) \cup \mathfrak{I}_{R}$.
Then following Proposition~\ref{prop:localconvex}, $\gamma$ is a strictly convex closed loop disjoint from $\mathcal{Z}(PQ)$ and $\minvset{CH}$ is a strictly convex compact domain bounded by $\gamma$ (in particular $\gamma$ encompasses every point of $\mathcal{Z}(PQ)$). A neighborhood of $\gamma$ is foliated by periodic integral curves $\gamma_{t}$ of the vector field $R(z)\partial_{z}$ that are also disjoint from $\mathcal{Z}(PQ)$ and $\mathfrak{I}_{R}$, so strictly convex as well. Each of them cuts out a strictly convex compact domain $\mathcal{D}_{t}$. For each point $z$ in the complement of some $\mathcal{D}_{t}$, $r(z)$ remains disjoint from $\mathcal{D}_{t}$, which by Lemma~\ref{lem:excision} contradicts the minimality of $\minvset{CH}$.

\par
Now, we show that a local arc cannot go to infinity. When $|\deg Q - \deg P| \leq 1$, integral curves going to infinity enter the cones disjoint from $\minvset{CH}$ and never leave them (see Section~\ref{sec:asymptotic}) and otherwise $\minvset{CH}$ is trivial.
\par
In the remaining cases, Poincaré-Bendixson theorem proves that a local arc $\gamma$ has an ending point $y \in \partial\minvset{CH}$. We assume by contradiction that $y \notin \mathcal{Z}(PQ) \cup \mathfrak{I}_{R}$. We consider an arc $\beta$ formed by a portion of the integral curve ending at $y$ and a portion of the associated ray $r(y)$. Provided that $\beta$ remains in the same inflection domain as $y$, the family of associated rays starting at the  points of the arc $\beta$ sweeps out a domain containing a cone (see Lemmas~\ref{lem:ArcCone} and \ref{lem:horn out}). Therefore, we have $\Delta(y) = \emptyset$. Proposition~\ref{prop:localarc} then proves that the local arc can be continued in a neighborhood of $y$.
\par
If $y \in \mathcal{Z}(PQ)$ and is a zero or a pole of $R(z)$, then $\mathcal{L}_{y}$ contains an interval of length at least $\pi$ (see Definition~\ref{def:calKL}).  Corollary~\ref{cor:KLalpha} proves that $y$ is either a simple pole or a simple zero satisfying $\phi_{y}=0$. In the latter case, $y$ is a repelling singular point of $R(z)\partial_{z}$ and therefore cannot be the endpoint of a local arc.
\end{proof}

As we will see in Section~\ref{sub:extruding}, in contrast  with the case of ending points, a local arc can start inside an inflection domain at a point of extruding type.

\subsection{Local connectedness of $\minvset{CH}$}
    
Here we show that $\minvset{CH}$ is locally connected, away from the part of the tangency locus that is formed by straight lines.

\begin{lemma}\label{lem:locallyconnected not IR}
$\minvset{CH}$ is locally connected outside of $\mathfrak I_R \cup \zeros(PQ)$.
\end{lemma}

\begin{proof}
If $z\in\partial \minvset{CH}$ is a point of local type, then the boundary locally coincides with the integral curve passing through $z$ by Lemma~\ref{lem:localarc step 2}. Next, let $z\in\partial \minvset{CH}\setminus (\zeros(PQ)\cup \mathfrak I_R)$ with $\Delta(z)\neq \emptyset$. Let $y\in \Delta(z)$. As $z\notin \zeros(PQ)\cup \mathfrak I_R$, by Lemma~\ref{lem:RootTrailSlope} or Lemma~\ref{lem:RTSlopeINFINITY} there is a unique germ of $\ttt_y$ passing through $z$, and it does so transversely to the integral curve of $R(z)\partial_{z}$ passing through $z$. Then taking the backward trajectories of $R(z)\partial_{z}$ of points in $\ttt_y$, all points on one side of $\ttt_y$ near $z$ belong to $\minvset{CH}$. Taking now as a neighborhood basis a family of decreasing curvilinear quadrilaterals with two of the sides being trajectories of $R(z)\partial_{z}$ and two sides being smooth curves on either side of $\ttt_y$, it follows that all points on the other side of $\ttt_y$ have backward trajectories of $R(z)\partial_{z}$ intersecting $\ttt_y$ inside these neighborhoods, provided they are sufficiently small. As for any $x\in \minvset{CH}$, its backward trajectory belongs to $\minvset{CH}$, it follows that $\minvset{CH}$ is locally connected at $z$.
\end{proof}

\begin{lemma}\label{lem:locallyconnectedZeroesANDPoles}
$\minvset{CH}$ is locally connected at zeros and poles of $R$.
\end{lemma}

\begin{proof}
If $z_0$ is a pole of $R$ one can show using Proposition 3.12 in \cite{AHN+24} and Corollary~\ref{cor:KLalpha} that $\minvset{CH}$ is locally connected at $z_0$. Next, for $z_0$ a zero of $R$ it follows from the same corollary and by using the local portrait of $R(z)\partial_{z}$ near $z_0$.
\end{proof}

\begin{lemma}\label{prop:locallyconnected not flat IR}
     $\minvset{CH}$ is locally connected at all $z$, such that $\gamma_z$ is not a straight line.
\end{lemma}
\begin{proof}
  
Since the integral curve $\gamma_z$ of vector field $R(z)\partial_{z}$ containing $z$ is not a straight line, some germ of the negative part $\gamma_z^-$ lies outside of $\mathfrak{I}_{R}$. Recall that if $y\in\minvset{CH}$ then the negative part $\gamma_y^-$ of $\gamma_y$ necessarily lies in $\minvset{CH}$. 

Let $\gamma_z^-(\epsilon)\subset \mathfrak{I}_{R}^c$ be the  part of the negative trajectory $\gamma_z^-$ lying strictly between $z$ and $y=g_R^{-\epsilon}(z)$
where $g_R^t$ is the flow of $R(z)\partial_z$. By Lemma~\ref{lem:locallyconnected not IR} there is a neighborhood $V_y$ of $y$ of size smaller than $\epsilon$ such that $V_y\cap\minvset{CH}$ is connected.
Now, let $U_z=\cup_{0\le t \leq \epsilon} g_R^t(V_y)$. Clearly $U_z$ is a neighborhood of $z$. We claim that $U_z\cap\minvset{CH}$ is connected. Indeed, if $w\in U_z\cap\minvset{CH}$ then 
$w=g_R^t(w')$ for some $w'\in V_y\cap\minvset{CH}$, $0\le t \le \epsilon$. Since $w'$ lies in the same connected component of $U_z\cap\minvset{CH}$ as $y$ and $w$ and $w'$ are jointed by $\gamma_w^-\subset U_z\cap\minvset{CH}$ this means that $U_z\cap\minvset{CH}$ is connected. As $\epsilon$ can be chosen arbitrarily small, the statement follows.
\end{proof}

We denote by $\mathcal{L}$ the union of all $R$-invariant lines.
\begin{corollary}\label{cor:para}
$\partial \minvset{CH}\setminus \mathcal L_R$ is parametrizable.
\end{corollary}
By Carathéodory's theorem, the boundary of an open set is parametrizable if its boundary is locally connected. For each $z \in \partial \minvset{CH} \setminus \mathcal{L}_R$ we have by Lemma~\ref{lem:locallyconnectedZeroesANDPoles} and Proposition~\ref{prop:locallyconnected not flat IR} a neighborhood basis $\mathcal{N}(z)$ consisting of sets such that $U \cap \minvset{CH}$ is connected. Define $U(z)\in \mathcal{N}(z)$ to be a set of the form $U(z) \subset B(z,\frac{1}{2}d(z,\mathcal{L}_R))$. The union $$\bigcup_{z\in \partial \minvset{CH} \setminus \mathcal{L}_{R}} U(z)$$ is an open cover of $\partial \minvset{CH} \setminus \mathcal{L}_{R}$ and being a subset of $\mathbb{C}$, it has a countable subcover $$\bigcup_{n\in \mathbb N} U(z_{n}).$$
For each $z_n, \overline{U(z_n)\cap \minvset{CH}}$ is locally connected and by Lemma~\ref{lem:ImKleinen} and the fact that all irregular points are contained in $\mathcal{L}_R$, its boundary is a Jordan curve away from the poles and zeros of $R(z)$. Hence $\partial \overline{U(z_n)\cap \partial\minvset{CH}}$ is parametrizable by Carathéodery's theorem, injectively away from the zeros and poles of $R(z)$. We start with a $z_0$ and use this parametrization of $\partial \overline{U(z_n)\cap \partial\minvset{CH}}$. Then for the smallest $n=n_1$ such that $U(z_n)\cap U(z_0)\neq \emptyset, \partial \minvset{CH}\cap U(z_n)\not\subset\partial \minvset{CH}\cap U(z_0)$, we glue together the parametrizations of $\partial \overline{(U(z_{n_1})\setminus U(z_0))\cap\partial \minvset{CH}}$ with that of $\partial \overline{U(z_0)\cap \partial\minvset{CH}}$ along the end points of the parametrizations. We then have a parametrization of $(U(z_0)\cup U(z_{n_1}))\cap \partial \minvset{CH}\coloneqq\mathcal{B}_1$. We then take the smallest $n=n_2$ such that $U(z_{n_2})\cap (U(z_0)\cup U(z_{n_1}))\neq \emptyset, \partial \minvset{CH}\cap U(z_n)\not\subset\mathcal{B}_1$ and in the same way find a parametrization of $\mathcal{B}_2\coloneqq(U(z_0)\cup U(z_{n_1})\cup U(z_{n_2}))\cap \partial \minvset{CH}.$ We proceed in this way inductively to get a parametrization of $\partial \minvset{CH}\setminus \mathcal L_R$, potentially pinched at the zeros and poles of $R(z)$ (but not anywhere else).

\subsection{Global arcs}\label{sub:global}

\subsubsection{Additional results about correspondence $\Delta$}

\begin{lemma}\label{lem:EchangeInflectionDelta}
Consider $z \in \partial \minvset{CH} \setminus \mathcal{Z}(PQ)$ such that $z \in \mathfrak{E}^{+}$ (resp. $\mathfrak{E}^{-}$). If $y \in \partial \minvset{CH}$ and $y\in \Delta(z)$, then one of the following statements holds:
\begin{itemize}
    \item $y \in \mathcal{Z}(PQ) \cup \mathfrak{I}_{R}$;
    \item $y \in \mathfrak{I}^{-}$ (resp. $\mathfrak{I}^{+}$).
\end{itemize}
\end{lemma}

\begin{proof}
Without loss of generality, we assume that $z=0$, $r(z)=\mathbb{R}^{+}$ and $z \in \mathfrak{E}^{+}$. 
\par
We consider $y \in \Delta(z)$ such that $y \notin \mathcal{Z}(PQ) \cup \mathfrak{I}_{R}$. If $\Delta(y) = \emptyset$, then Proposition~\ref{prop:localarc} shows that $y$ belongs to a local arc. Besides, $y$ is an indirect  support point of the associated ray $r(z)$ (see Lemma~\ref{lem:directsupport}). If $y \in \mathfrak{I}^{+}$, then the associated rays starting from a germ of the local arc at $y$ sweep out a neighborhood of $z$ and we get a contradiction. Therefore $y \in \mathfrak{I}^{-}$.
\par
Now we consider the case where $\Delta(y) \neq \emptyset$ and assume by contradiction that $y \in \mathfrak{I}^{+}$. If $Im(R(y))>0$, then Lemma~\ref{lem:intersectingrays} provides an immediate contradiction. If $Im(R(y))<0$, then $\mathcal{L}_{y}$ (see Definition~\ref{def:calKL}) contains an interval of length strictly larger than $\pi$ such that $\sigma(y)$ is one of the ends. It follows  from Corollary~\ref{cor:KLalpha} that $y$ is a simple zero of $R(z)$ (and therefore $y \in \mathcal{Z}(PQ)$).
\par
If $r(y)=y+\mathbb{R}^{-}$, then for some small $\epsilon >0$, points of the interval $]-\epsilon,\epsilon[$ are disjoint from the interior of $\minvset{CH}$. Associated rays starting from the points of  $]-\epsilon,\epsilon[$ sweep out an open cone containing a neighborhood of $y$. This contradicts the assumption $y \in \partial \minvset{CH}$. Therefore, $r(y)=y+\mathbb{R}^{+}$ and $r(y) \subset r(z)$. In this case, for some small $\epsilon' >0$, points of  $]y-\epsilon',y[$ are disjoint from the interior of $\minvset{CH}$ and their associated rays will sweep out a neighborhood of $y$ if $y \in \mathfrak{I}^{+}$. Therefore, in that case we get that $y \in \mathfrak{I}^{-}$. Similar result holds for $z \in \mathfrak{E}^{-}$.
\end{proof}

\begin{definition}\label{def:Deltaminmax}
For any point $z \in \partial \minvset{CH}$ such that $z \notin \mathcal{Z}(PQ) \cup \mathfrak{I}_{R}$ and $\Delta(z) \neq \emptyset$, we define $\Delta^{min}(z)$ (resp. $\Delta^{max}(z)$) as the infimum (resp. the supremum) in $\Delta(z)$ of the order induced by the orientation of the associated ray $r(z)$.
\par
Besides, we define $L_{z}=|\Delta^{min}(z)-z|$.
\end{definition}

\begin{lemma}\label{lem:gap}
For any $z \in \partial \minvset{CH}$ such that $z \notin \mathcal{Z}(PQ) \cup \mathfrak{I}_{R}$ and $\Delta(z) \neq \emptyset$, we have $\Delta^{min}(z) \neq z$ and $L_{z} \neq 0$.
\end{lemma}

\begin{proof}
Without loss of generality, we assume that $z=0$, $z \in \mathfrak{I}^{+}$ and $r(x)=\mathbb{R}_{>0}$. For any small enough real positive $\epsilon$, we have $Re(R(\epsilon)),Im(R(\epsilon))>0$ and $\epsilon \in \mathfrak{I}^{+}$. If such an $\epsilon$ belongs to $\Delta(z)$, then it contradicts Lemma~\ref{lem:EchangeInflectionDelta}.
\end{proof}

Since $\overline{\minvset{CH}}$ is compact in $\mathbb{C} \cup \mathbb{S}^{1}$, it follows immediately that for any $z$, $\Delta^{min}(z)$ is actually a point of $\partial \overline{\minvset{CH}}$.

\begin{definition}\label{def:calUz}
For any point $z \in \partial \minvset{CH}$ such that $z \notin \mathcal{Z}(PQ) \cup \mathfrak{I}_{R}$ and $\Delta(z) \neq \emptyset$, we define $\mathcal{U}(z)$ as the connected component of $(\minvset{CH})^{c} \setminus [z,\Delta^{min}(z)]$  incident to:
\begin{itemize}
    \item the right side of $[z,\Delta^{min}(z)]$ if $z \in \mathfrak{I}^{+}$;
    \item the left side of $[z,\Delta^{min}(z)]$ if $z \in \mathfrak{I}^{-}$.
\end{itemize}
i.e. in the half-plane bounded by $r(z)$ different from that containing  the germ of the trajectory of $R(z)\partial_z$ starting at $z$.
\end{definition}

\begin{lemma}\label{lem:boundaryUz}
Consider $z \in \partial \minvset{CH}$ such that $z \notin \mathcal{Z}(PQ) \cup \mathfrak{I}_{R}$, $\Delta(z) \neq \emptyset$ and $z\in \mathfrak{I}^{+}$ (resp. $\mathfrak{I}^{-}$). For any $y \in \partial \minvset{CH} \cap \partial{\mathcal{U}(z)}$ such that $y \in \mathfrak{I}^{+}$ (resp. $\mathfrak{I}^{-}$) and $\Delta(y) \neq \emptyset$, we have $\mathcal{U}(y) \subset \mathcal{U}(z)$.
\par
Besides, if $y \neq z$, we have $\mathcal{U}(y) \subsetneq \mathcal{U}(z)$.
\end{lemma}

\begin{proof}
By connectedness of $\minvset{CH}$ in the case $\deg Q-\deg P=\pm 1$ and the asymptotic geometry of $\minvset{CH}$ in the case $\deg Q-\deg P=0$, it follows that the associated ray $r(y)$ intersects the associated ray $r(z)$. Applying Lemma~\ref{lem:intersectingrays} to $r(z)$ and $r(y)$, we see that $\Delta(y) \subset \partial \mathcal{U}(z)$. Thus  $\mathcal{U}(y) \subset \mathcal{U}(z)$.

When $\mathcal{U}(y) = \mathcal{U}(z)$, the associated ray $r(y)$ has to coincide with $r(z)$ (with the same orientation since $y$ and $z$ belong to the same inflection domain). It follows that $y=z$.
\end{proof}

\subsubsection{Orientation of global arcs}\label{sub:orientationglobal}
By Corollary~\ref{cor:para}, the following notion is well-defined.
\begin{definition}\label{def:globalarc}
\emph{A global arc} in $\partial \minvset{CH}$ is a maximal open connected arc formed by points of global type.
\end{definition}
Furthermore, for a global arc $\alpha$ defined on $(t_{min},t_{max})$, its end point is defined as long as $$\omega_+(\alpha)\coloneqq\bigcap_{t_0\in (t_{min},t_{max})}\overline{\{\alpha(t):t>t_0\}}$$ is a singleton (and equals this element). The starting point is analogously defined if $$\omega_-(\alpha)\coloneqq\bigcap_{t_0\in (t_{min},t_{max})}\overline{\{\alpha(t):t<t_0\}}$$ is a singleton. If $\omega_+(\alpha)$ is not a singleton, then it can only contain points contained in $R$-invariant lines, again by Corollary~\ref{cor:para} and similarly for $\omega_-(\alpha)$. Regardless if they are singletons or not, we call the sets $\omega_\pm(\alpha)$ \emph{end accumulation} and the \emph{start accumulation}. In case they are in fact singletons, we will also call them end and starting points respectively.
We have a geometrically meaningful way to define orientation on global arcs.

\begin{lemma}\label{lem:globalorientation}
Any global arc $(\alpha_{t})_{t \in I}$ can be oriented in such a way that for $t'>t$, we have:
\begin{itemize}
    \item $\alpha_{t} \in \partial\minvset{CH} \cap \partial \mathcal{U}(\alpha_{t'})$;
    \item $\mathcal{U}(\alpha_{t'}) \supset \mathcal{U}(\alpha_{t})$.
\end{itemize}
In particular, in $\mathfrak{I}^{+}$, the orientation of global arcs coincides with the standard topological orientation of $\partial \minvset{CH}$ (it coincides with the opposite orientation in $\mathfrak{I}^{-}$).
\par
In particular, a global arc is an interval, i.e. it cannot be a closed loop.
\end{lemma}

\begin{figure}[!ht]
    \centering
    \includegraphics[width=0.6\textwidth]{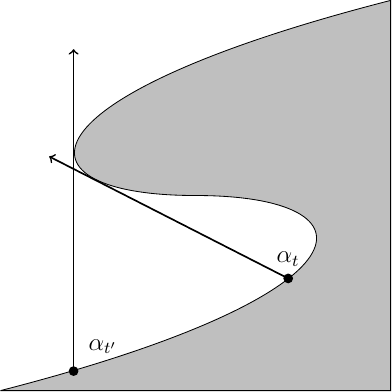}
    \caption{Two associated rays from the same global arc.}
    \label{fig:Lemma422}
\end{figure}

\begin{proof}
Removal of $\alpha_{t}$ from  $\alpha$ cuts the arc into two pieces, one of which is contained in $\partial \mathcal{U}(\alpha_{t})$ (see Figure~\ref{fig:Lemma422}). Lemma~\ref{lem:boundaryUz} then proves the inclusion of the sets of the form $\mathcal{U}(\alpha_{t})$ as $t$ sweeps out the interval $I$ which provides a meaningful orientation on the global arc.
\end{proof}

\begin{lemma}\label{lem:globalparameter}
Along a global arc $\alpha$, the function $\sigma(z)= \arg(R(z))$ is a monotone mapping of $\alpha$ to an interval in $\mathbb S^1$ with length at most $\pi$.
\par
Besides, if  $\sigma(\alpha_{t})=\sigma(\alpha_{t'})$ for some $t>t'$, then $\Delta(\alpha_{t})$ coincides with the point at infinity $\sigma(\alpha_{t})=\sigma(\alpha_{t'})$ that also belongs to $\Delta(\alpha_{t'})$.
\end{lemma}

\begin{proof}
Consider two points $\alpha_{t}$ and $\alpha_{t'}$ of a global arc satisfying $t>t'$ for the canonical orientation. By Lemma~\ref{lem:globalorientation}, $\alpha_{t'} \in \partial \mathcal{U}(\alpha_{t})$.
\par
Without loss of generality, we assume that $\alpha$ is contained in $\mathfrak{I}^{+}$, $\alpha_{t}=0$ and $r(\alpha_{t})=\mathbb{R}^{+}$. If $\sigma (\alpha_{t'}) \in [-\pi,0[$, any associated ray starting in a small enough neighborhood of $\alpha_{t'}$ will cross
$\minvset{CH}$. If $\sigma(\alpha_{t'})=0$, then the interior of the strip bounded by $r(\alpha_{t}),r(\alpha_{t'})$ and the portion of global arc between $\alpha_{t}$ and $\alpha_{t'}$ is disjoint from $\minvset{CH}$. It follows that $\Delta(\alpha_{t})$ contains only the point at infinity. In the remaining case, we have $\sigma (\alpha_{t'}) \in ]0,\pi[$.
\end{proof}

\begin{proposition}\label{prop:globalarc}
Consider $z \in \partial \minvset{CH}$ such that $z \notin \mathcal{Z}(PQ) \cup \mathfrak{I}_{R}$ and $\Delta(z) \neq \emptyset$. Then, $z$ is either the endpoint or a point of a global arc.
\end{proposition}

\begin{proof}
We consider an arbitrarily small open arc $\alpha$ of $\partial \minvset{CH} \cap \mathcal{U}(z)$ ending at  $z$. By assumptions, $\alpha$ is disjoint from $\mathcal{Z}(PQ) \cup \mathfrak{I}_{R}$. If some point $y\in\alpha$ satisfies $\Gamma(y) \neq \emptyset$, then $\alpha$ partially coincides with a local arc. Since the ending point of any local arc belongs to $\mathcal{Z}(PQ) \cup \mathfrak{I}_{R}$ (see Proposition~\ref{prop:localEND}), comparison of the orientation of local arcs and the orientation of $\partial \minvset{CH}$ in a given inflection domain (see Lemma~\ref{prop:localconvex}) proves that $z$ also belongs to this local arc. This is a contradiction. Therefore, any point $y$ in the arc $\alpha$ satisfies $\Gamma(y) = \emptyset$.
Proposition~\ref{prop:localarc} then implies that each point of the arc $\alpha$  satisfies $\Delta(y) \neq \emptyset$ and is thus a point of global type. Therefore, $z$ is either the endpoint or a point of a global arc containing $\alpha$.
\end{proof}

\begin{proposition}\label{prop:globalpoint}
If a point $z \in \partial \minvset{CH}$ satisfies:
\begin{itemize}
    \item $z \notin \mathcal{Z}(PQ) \cup \mathfrak{I}_{R}$;
    \item $\Delta(z) \neq \emptyset$;
    \item $\Gamma(z) = \emptyset$;
\end{itemize}
then $z$ belongs to a global arc.
\end{proposition}

\begin{proof}
Following Proposition~\ref{prop:globalarc}, $z$ is either the ending point or a point of a global arc.  We consider a connected neighborhood $V$ of $z$ in $\partial \minvset{CH}$ that is disjoint from $\mathcal{Z}(PQ) \cup \mathfrak{I}_{R}$. Without loss of generality, we assume that $V$ belongs to $\mathfrak{I}^{+}$.
\par
We consider a point $y \in V$ such that the oriented arc from $y$ to $z$ in $\partial \minvset{CH}$ has the same orientation as the standard topological orientation of the boundary. If $\Gamma(y) \neq \emptyset$, then $y$ is either a point or the starting point of a local arc (Corollary~\ref{cor:gammabasic}) that can be continued til $z$ (see Proposition~\ref{prop:localEND}) since $V$ is disjoint from $\mathcal{Z}(PQ) \cup \mathfrak{I}_{R}$. Therefore, $\Gamma(y) = \emptyset$ and it follows then from Proposition~\ref{prop:localarc} that $\Delta(y)\neq \emptyset$. Thus, any such point $y$ is a global point belonging to global arc $\alpha$.
\par
Then, we consider points $y \in V$ such that the oriented arc from $y$ to $z$ in $\partial \minvset{CH}$ has the opposite orientation as the standard topological orientation of the boundary. If such point $y$ satisfies $\Gamma(y) \neq \emptyset$, then it is a point or the starting point of a local arc. Since $V$ is connected and disjoint from $\mathcal{Z}(PQ) \cup \mathfrak{I}_{R}$, it contains at most one local arc starting at a point of extruding type. The complement of the closure of this local arc in $V$ coincides with global arc $\beta$. By hypothesis, $z$ is not a point of extruding type so it belongs to a global arc $\beta$.
\end{proof}

\begin{proposition}\label{prop:globalEndNONTRIVIAL}
If $z_{0} \in \mathbb{C}$ is the endpoint of a global arc $\alpha$ and  is neither a zero nor a pole of $R(z)$, then $\Delta(z_{0}) \neq \emptyset$.
\end{proposition}

\begin{proof}
We assume that $\alpha$ is parameterized by the interval $]0,1[$ (with the correct orientation) and $\alpha(t) \rightarrow z_{0}$ as $t \rightarrow 1$. For any $n \geq 2$, we pick a point $\beta_{n} \in \Delta(\alpha(1-1/n))$. Since $\mathbb{C} \cup \mathbb{S}^{1}$ is compact, the sequence $(\beta_{n})_{n \geq 2}$ has an accumulation point $\beta$. Since $z_{0}$ is the endpoint of  $\alpha$, the point $\beta$ cannot coincide with $z_{0}$ (see Lemma~\ref{lem:globalorientation}). It follows that a family of associated rays accumulates on a half-line starting at $z_{0}$ and containing $\beta$. Since $\arg(R(z))$ is continuous in a neighborhood of $z_{0}$, we get that  this half-line coincides with the associated ray $r(z_{0})$.
\end{proof}

%\subsubsection{Tangent directions of global arcs}

%\begin{definition}
  %On any global arc, define the \textit{normalized radius} $\rho$ as a positive function given  by $\rho(z)=\frac{\Delta^{min}(z)-z}{R(z)}$.
%\end{definition}

%\begin{conjecture}
%The normalized radius $\rho$ is an increasing positive function (relatively to the orientation defined in Lemma~\ref{lem:globalorientation}) along any global arc.
%\end{conjecture}

\begin{definition}
    A point $z\in \partial \minvset{CH}$ is a \emph{non-convexity point} if there is a cone $\mathcal{C}$ at $z$ of angle strictly bigger than $\pi$ and a neighborhood $V$ of $z$ such that $\mathcal{C} \cap V \subset \minvset{CH}$.
\end{definition}

For a point $z_{0}$ for which  $\Delta(z_{0})$ consists of   a single point $u$ satisfying the condition $R(z_0)+(u-z_0)R'(z_0)\neq0$, Lemma~\ref{lem:RootTrailSlope} proves that: (i) the root trail $\mathfrak{tr}_{u}$ has a unique branch at $z_{0}$, (ii) it is contained in $\minvset{CH}$,  and (iii) its tangent slope is the argument of $\frac{R^{2}(z_{0})}{R(z_{0})+(u-z_{0})R'(z_{0})}$ (mod $\pi$). These lemma yields that if at some point $z_{0}$, $\Delta(z_{0})$ contains more than one point, then $\partial \minvset{CH}$ cannot be smooth at $z_{0}$:

\begin{proposition}\label{prop:globalmultiple}
At a point $z \notin \mathfrak{I}_{R} \cup \mathcal{Z}(PQ)$ such that $\Delta(z)$ contains at least two points, the boundary $\partial\minvset{CH}$ has a non-convexity point.
\end{proposition}

\begin{proof}
First assume that $\Delta(z)$ contains two points $u,v$ both of which are not points at infinity. Lemma~\ref{lem:RootTrailSlope} proves that $z$ belongs to two distinct root trails. Assuming that $R(z)+(u-z)R'(z)$ and $R(z)+(v-z)R'(z)$ are nonzero, the tangent slopes of these root trails at $z$ are determined by the argument of $\frac{R^{2}(z)}{R(z)+(u-z)R'(z)}$ and $\frac{R^{2}(z)}{R(z)+(v-z)R'(z)}$. By hypothesis, we have $Im(R'(z)) \neq 0$ and these two branches intersect transversely at $z$ and the claim follows, taking the backward trajectories of the two root-trails. If $R(z)+(u-z)R'(z)=0$, then two branches of the root trail intersect transversely.
\par
In the remaining case, $\Delta(z)$ contains exactly one point $u$ satisfying the condition $R(z)+(u-z)R'(z)\neq0$ and a point $\sigma(z)$ at infinity. Then the root trail of $u$ at $z$ has a slope given by the argument of $\frac{R^{2}(z)}{R(z)+(u-z)R'(z)}$ (or $\frac{R(z)}{R'(z)}$ if $u$ is at infinity, see Lemma~\ref{lem:RTSlopeINFINITY}). Similarly, $R'(z) \notin \mathbb{R}$ so these curves intersect transversely at $z$. Summarizing we see that in all possible cases, the boundary $\partial\minvset{CH}$ has a non-convexity point.
\end{proof}

\subsection{Points of extruding type}\label{sub:extruding}

Outside the local and the global arcs, the only singular boundary points in the complement of $\mathcal{Z}(PQ) \cup \mathfrak{I}_{R}$ which can occur are points of extruding type.

\begin{proposition}\label{prop:extruding}
Let $z$ be a point of extruding type in $\partial \minvset{CH}$. Then $z$ is both the endpoint of a global arc and the starting point of a local arc.
\par
The boundary $\partial \minvset{CH}$ is not $C^{1}$ at $z$ and $z$ is a non-convexity point.
\end{proposition}

\begin{proof} See Figure~\ref{fig:ExtrudingProof} below.
By definition of the correspondence $\Gamma$, and local considerations of $R$,  $z$ is the starting point of a local arc. Propositions~\ref{prop:globalarc} and ~\ref{prop:globalpoint} show that $z$ is the ending point of a global arc.
\par
For any point $u \in \Delta(z)$, the root trail $\mathfrak{tr}_{u}$ has a unique local branch at $z$ and its tangent direction is the argument of $\frac{R^{2}(z)}{R(z)+(u-z)R'(z)}$ (mod $\pi$), see Lemma~\ref{lem:RootTrailSlope}. Indeed, $R(z)+(u-z)R'(z) \neq 0$ because $u-z$ is real collinear to $R(z)$ while $\Im(R'(z)) \neq 0$. Since $z \notin \mathfrak{I}_{R}$, this branch transversely intersects the integral curve of $R(z)\partial_{z}$ containing $z$. Both of these branches are (semi)analytic curves  contained in $\minvset{CH}$ and the associated rays of the points lying on the negative part of $\gamma_z$ intersect $\mathfrak{tr}_{u}\subset\minvset{CH}$. Thus the negative part of $\gamma_z$ is disjoint from  $\partial\minvset{CH}$.
\end{proof}
\begin{figure}
    \centering
    \includegraphics[width=0.5\linewidth]{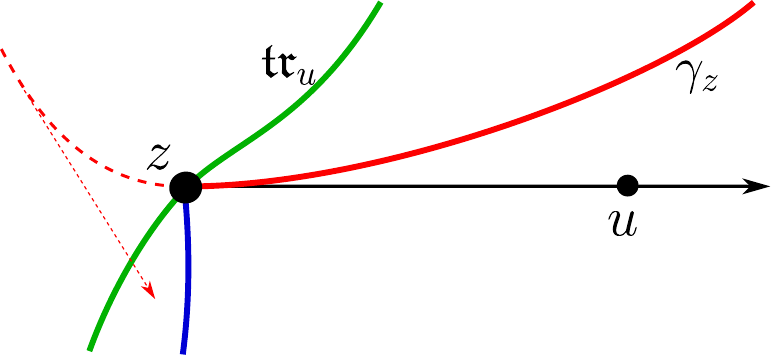}
    \caption{Negative part of $\gamma_z$ cannot be on the boundary as $\mathfrak{tr}_{u}\subset\minvset{CH}$.}\label{fig:ExtrudingProof}
\end{figure}

\subsection{Boundary arcs in inflection domains}

\begin{proposition}\label{prop:boundaryarc}
For any connected component $\mathcal{D}$ of the complement of the curve of inflections  $\mathfrak{I}_{R}$, $\partial \minvset{CH} \cap \mathcal{D}$ is a union of disjoint topological arcs. In each of them, local and global arcs have the same orientation. If $\Im(R')$ is positive (resp. negative) in $\mathcal{D}$ then the latter orientation coincides with (is opposite to) the topological orientation of $\partial \minvset{CH}$.
\end{proposition}

\begin{proof}
The statement about orientation follows from Proposition~\ref{prop:localconvex} and Lemma~\ref{lem:globalorientation}. Proposition~\ref{prop:extruding} shows that a point of extruding type is incident to a local and a global arcs. It remains to prove that any point of $\mathcal{Z}(PQ) \cap \mathcal{D}$ is incident to at most two arcs.
\par
Such a point $z_{0}$ is neither a zero nor a pole of $R(z)$, see Corollary~\ref{cor:SingInflec}. Therefore, Corollary~\ref{cor:BOUNDINTERIOR} together with the fact that all irregular points are contained in $\mathcal I_R$ proves the statement.
\end{proof}

\section{Singular boundary points on the curve of inflections}\label{sec:SingInflection}

At points belonging to the curve of inflections the boundary $\partial\minvset{CH}$ can display more complicated behaviours. In this section, we classify  boundary points that belong to the transverse locus $\mathfrak{I}_{R}^{\ast}$ of the curve of inflections (see Definition~\ref{defn:transverse}). For the following definition, recall the Definition~\ref{def:-+0}.

\begin{definition}\label{defn:TRANSVERSEClassification}
A point of $\partial\minvset{CH} \setminus \mathcal{Z}(PQ)$ belonging to the transverse locus $\mathfrak{I}_{R}^{\ast}$ is a point of:
\begin{itemize}
    \item \emph{bouncing type} if $\Delta^{+} \neq \emptyset$ and $\Gamma \cup \Delta^{-} \neq \emptyset$;

    \item \emph{switch type} if $\Delta^{+} \neq \emptyset$ and $\Gamma \cup \Delta^{-} = \emptyset$;   

    \item \emph{$C^{1}$-inflection type} if $\Delta^{+} = \emptyset$, $\Delta^{-} \neq \emptyset$ and $\Gamma = \emptyset$;
        
    \item \emph{$C^{2}$-inflection type} if $\Delta^{+} = \emptyset$ and either $\Delta^{-} = \emptyset$ or $\Gamma \neq \emptyset$.
    \end{itemize}
\end{definition}

\subsection{Horns at points of the transverse locus}\label{sub:HornTransverse}

At a point $p \in \mathfrak{I}_{R}^{\ast}$, the  curve of inflections  is smooth and the vector field $R(z)\partial_z$ is transversal to it. This means that by \eqref{eq:R Taylor} we can assume that
   \begin{equation}\label{eq:R at generic IC}
       R(u)=1+\rho u +(a+ib)u^2+...
   \end{equation}
where we assumed that $p=0$. The condition $\Im R'(0)=0$ means that $\rho\in \R$, and the transversality condition is equivalent to $b\not=0$. Without loss of generality we can assume that  $b>0$. In other words,   $m=2$ in \eqref{eq:R Taylor} which implies that the integral curves  locally look like  cubic curves with inflections at these points.

We define the \emph{diameter} of a horn $\horn{p}{p'}{p''}$ to be the least upper bound  $t_0>0$ of all $t'>0$ such that  there is $u\in \horn{p}{p'}{p''}$ such that $u+tR(u)\in \horn{p}{p'}{p''}$ for all $t\in(0,t')$.

\begin{lemma}\label{lem:unif horns at I_R^*}
For $p\in\mathfrak{I}_{R}^{\ast}$,  there exists a neighborhood $\Omega$ of $p$ and $\epsilon>0$ such that for all points $u\in\overline{\Omega_+}$, where  $\Omega_+=\mathfrak{I}^+\cap \Omega$,  there exists a small horn of diameter greater than $\epsilon$.
\end{lemma}

\begin{proof}
Without loss of generality we assume $p=0$.
Consider the function  
$$
T(u,t)=\frac{d \sigma(u+tR(u))}{dt}(t)=\Im \left[\frac{R'(u+tR(u))}{R(u+tR(u))}R(u)\right]
$$
defined in $\bC_u\times\R_t$. Note that by definition $\mathfrak{I}^+=\{T(u,0)>0\}$. Since $T(0,t)= 2bt+O(t^2)$, we have $\frac{\partial T}{\partial t}(0,0)=2b>0$. Therefore  $\tfrac{\partial}{\partial t}T(u,t)>b>0$ in some sufficiently small neighborhood $\widetilde{U}\times(-\tilde{\epsilon},\tilde{\epsilon})$ of $(0,0)$. Let $\widetilde{U}_+=\mathfrak{I}^+\cap \widetilde{U}$. By definition of $\mathfrak{I}^+$,  
we have $\widetilde{U}_+\times \{0\}\subset \{T>0\}$. Taken together, this implies that  $ \widetilde{U}_+\times[0,\tilde{\epsilon}]\subset \{T>0\}$ for some $\tilde{\epsilon}>0$. This means that the argument of  $R(u+tR(u))$ is monotone increasing for  $0<t<\tilde{\epsilon}$ and for every $u\in\widetilde{U}_+$. 

By transversality of $\mathfrak{I}_R$ and $R$ at $0$ we have \begin{equation*}
    \frac{\partial}{\partial t} T\big(g_R^t(0)+sR(g_R^t(0)\big), 0)|_{s=t=0}=\frac{\partial}{\partial s} T\big(g_R^t(0)+sR(g_R^t(0)), 0\big)|_{s=t=0}>0,
\end{equation*}
 where $g_R^t$ is the flow of $R(z)\partial_z$. 
 
Thus there is a neighborhood  $\Omega\subset\widetilde{U}$ of $0$ and $\epsilon<\tilde{\epsilon}$ such that $T(g_R^t(u)+sR(g_R^t(u)), 0)>0$ as soon as $u\in  \Omega_+=\mathfrak{I}^+\cap \Omega$ and $t,s\in[0,{\epsilon}]$. Decreasing $\Omega,\epsilon$ if needed, we can assume that
\begin{equation*}
g_R^t(u)+sR(g_R^t(u))\subset \widetilde{U}_+ \quad\text{for all  }u\in \Omega_+ \,\text{ and}\,  t,s\in[0,{\epsilon}].
\end{equation*}

Thus for every $u \in \Omega_{+}$, any horn of diameter at most $\epsilon$ lies in $\widetilde{U}_{+}$ and is therefore a small horn.

If $u \in \overline{\Omega_+}$ then the horn $\horn{u}{u'}{u''}$ of diameter $\epsilon$ lies in a union of small horns $\horn{\tilde u}{\tilde u'}{\tilde u''}$ of   diameter $\epsilon$, where $\tilde u \in \horn{u}{u'}{u''}$. Thus $\horn{u}{u'}{u''}$ is a small horn as well. 

\end{proof}

\subsection{Points of bouncing type}\label{sub:bouncing}

\begin{proposition}\label{prop:bouncing}
Let $z$ be a point of bouncing type in $\partial \minvset{CH}$: 
$$\Delta^{+}(z) \neq \emptyset\quad\text{and}\quad \Gamma(z) \cup \Delta^{-}(z) \neq \emptyset.
$$ 
Then $z$ is the ending point of a global arc and also the starting point of another arc which can  either be local or global.
\par 
Further, 
$z$ is a point of nonconvexity and for small enough neighborhoods $V$ of $z$, one has $V \cap \partial \minvset{CH} \cap \mathfrak{I}_{R} = \lbrace{ z \rbrace}$.
\end{proposition}

\begin{proof}
Without loss of generality, we can assume that $z=0$,  $R(0)=1$,  and $R'(0) \in \mathbb{R}\setminus\{0\}$. 
%We have $R'(\epsilon)=R'(0) + R''(0)\epsilon + o(\epsilon)$.
Since $z$ belongs to $\mathfrak{I}_{R}^{\ast}$, we get  $R''(0)=a+bi$ with $a \in \mathbb{R}$ and $b \in \mathbb{R}^{\ast}$. Without loss of generality, we can  assume that $b>0$ and therefore $\gamma_z$ lies in the upper half-plane.

%Recall that by Lemma~\ref{lem:horn out} a germ of the domain bounded by the integral curve of $R(z)\partial_z$ starting at $z$ and $r(z)$ lies in the complement of $\minvset{CH}$.
\par
%Let us first consider the case when $\Delta^{0}(z)= \emptyset$. 
Let  $u_+\in\Delta^{+}(0) \neq \emptyset$.  If $u_+\notin\Delta^0(0)$ then the germ of $\mathfrak{tr}_{u_+}$ at $0$ is   smooth and  tangent to $\R$ at $0$ and contained in the 
{lower}
half-plane, see Lemma~\ref{lem:RootTrailSlope} and Proposition~\ref{prop:RTInflecCONCAVE}). Otherwise it   consists of two branches transversal to $\R$ and orthogonal one to another. Denote by $\alpha_+\subset\mathfrak{tr}_{u_+}\subset\minvset{CH}$ the arc starting at $0$ and lying in the lower  right quadrant.

Similarly, since  $\Gamma(0) \cup \Delta^{-}(0)$ is nonempty there exists an arc $\alpha_-$ (either portion of  an integral curve $\gamma_0$ or a root trail $\mathfrak{tr}_{u_-}$ of a $u_-\in\Delta^{-}(0)$) starting at $0$,
%with the horizontal tangent, 
 contained in the upper right quadrant, and belonging to $\minvset{CH}$. 

Denote by $\alpha=\alpha_-\cup\alpha_+$, and let $V$ be a small neighborhood  of $0$ such that $\alpha$ cuts  $V$ into two parts. 
The part $V_-$ of $V$ to the left of $\alpha$ is entirely contained in $\minvset{CH}$ (as $r(u)$  intersects $\alpha\subset\minvset{CH}$ for any $u\in V_-$). This domain contains the  intersection of $V$ with a  cone with vertex at $0$ and of angle strictly larger than $\pi$. It also contains all the intersection of $V$ with $\mathfrak{I}_{R}$ excepted for the point $0$, see Lemmas~\ref{lem:RootTrailSlope},~\ref{lem:RTSlopeINFINITY} and Remark~\ref{rem:ROOTTRAIL}.
\par
%If $\Delta^{0}(z)$ is nonempty, then it contains the unique point $u=-\frac{1}{R'(0)}$. In such a case, the root trail $\mathfrak{tr}_{u}$ has exactly two branches intersecting at $z$ (see Lemmas~\ref{lem:RootTrailSlope},~\ref{lem:RTSlopeINFINITY} and Remark~\ref{rem:ROOTTRAIL}). The tangent slopes of these branches (which intersect orthogonally) are $\frac{\theta_{0}}{2}$ (mod $\pi/2$) where $\theta_{0}$ is the argument of $\frac{1}{R''(0)}$. In contrast, the tangent slope of $\mathfrak{I}_{R}$ at $z$ is $\theta_{0}$ (mod $\pi$). Since $R''(0) \notin \mathbb{R}$, we have that $\theta_{0} \notin \pi\mathbb{Z}$ and these branches intersect transversely. Consequently, there is a neighborhood $V$ of $z$ such that the intersection of $V$ with $\minvset{CH}$ contains the intersection of $V$ with a cone of angle arbitrarily close to $\frac{3\pi}{2}$. In particular, it contains the intersection of $V$ with $\mathfrak{I}_{R}$ except for $z$ itself.
%\par
It remains to prove that in a neighborhood of $z$, $\partial \minvset{CH}$ is formed by exactly two arcs. 
\begin{lemma}\label{lem:Delta- not empty Gamma is empty bouncing}
    Assume that $\Gamma(0)=\emptyset$ and $\Delta^-(0)\not=\emptyset$. Then $0$ is a starting point of a global arc. 
\end{lemma}
\begin{proof}

    By Lemma~\ref{lem:horn out} the points lying below the forward trajectory $\gamma^+_0$ of $0$ are not in $\minvset{CH}$. Let $q$ be a point lying slightly above $\gamma_0$ and near $0$ such that $q\notin\minvset{CH}$ (it exists since otherwise $\gamma^+_0\subset \partial\minvset{CH}$). Denote by 
    $\tilde{p}=\tilde{p}(q)$ the first point on the backward trajectory $\gamma^-_q$ starting from $q$ such that $r(\tilde{p})\cap \minvset{CH}\neq\emptyset$, in particular $\tilde{p}\in\minvset{CH}$ and therefore $\gamma^-_{\tilde{p}}\subset\minvset{CH}$. 
     This point exists since the point of intersection of $\gamma^-_q$ with $\alpha$ has this property by definition of $\alpha$. 
     
     Let $\gamma_{\tilde{p}}^q$ be the closed piece of trajectory of $R$ joining $\tilde p$ and $q$. By definition of $\tilde p$, all point of the (evidently closed) set $\gamma_{\tilde{p}}^q\cap  \minvset{CH}$ except $\tilde p$ are necessarily in $\partial\minvset{CH}$ and of local type.  Proposition~\ref{prop:localarc} then implies  $\gamma_{\tilde{p}}^q\cap \partial \minvset{CH}=\{\tilde p\}$: indeed, otherwise   the set $\gamma_{\tilde{p}}^q\cap  \minvset{CH}$  cannot be closed.

    The trajectory $\gamma^+_0$ is convex, so taking $q$ sufficiently close to $\gamma^+_0$ we can assume that  $r(q)\subset(\minvset{CH})^c$ intersect $\gamma^+_0$ and therefore all trajectories of $R\partial_z$ close to and above $\gamma^+_0$.  The points $\tilde{p}(q')$,
    where $q'\in r(q)$ lies above $\gamma^+_0$, form a global arc of $\partial\minvset{CH}$ starting at $0$ and lying between $\alpha_-$ and $\gamma_0$ in the upper right quadrant.

\end{proof}

If $\Gamma(0) \neq \emptyset$, then a local arc whose germ is contained in the upper half-plane starts at $0$.

In both cases the portion of the boundary  $\partial \minvset{CH}$ in a neighborhood of $0$ in the lower right quadrant  is a global arc ending at $z$. Indeed, let  $q\in\horn{z}{z'}{z''}\subset(\minvset{CH})^c$  and denote again by $\tilde{p}=\tilde{p}(q)$ the first point on $\gamma_q$ such that $r(\tilde{p})\cap \minvset{CH}\neq\emptyset$, in particular $\tilde{p}\in\minvset{CH}$. As before, it  lies on $\gamma^-_q$ between $q$  and $\gamma^-_q\cap\alpha$.  Repeating the arguments of Lemma~\ref{lem:Delta- not empty Gamma is empty bouncing} we see that $\gamma_{\tilde{p}}^q\cap\minvset{CH}=\{\tilde{p}\}$ and the  points $\tilde{p}$ form a global arc of $\partial\minvset{CH}$ lying in the lower half-plane and ending at $z$. 

\begin{figure}\label{fig:BouncingProof}
    \centering
    \includegraphics[width=0.8\linewidth]{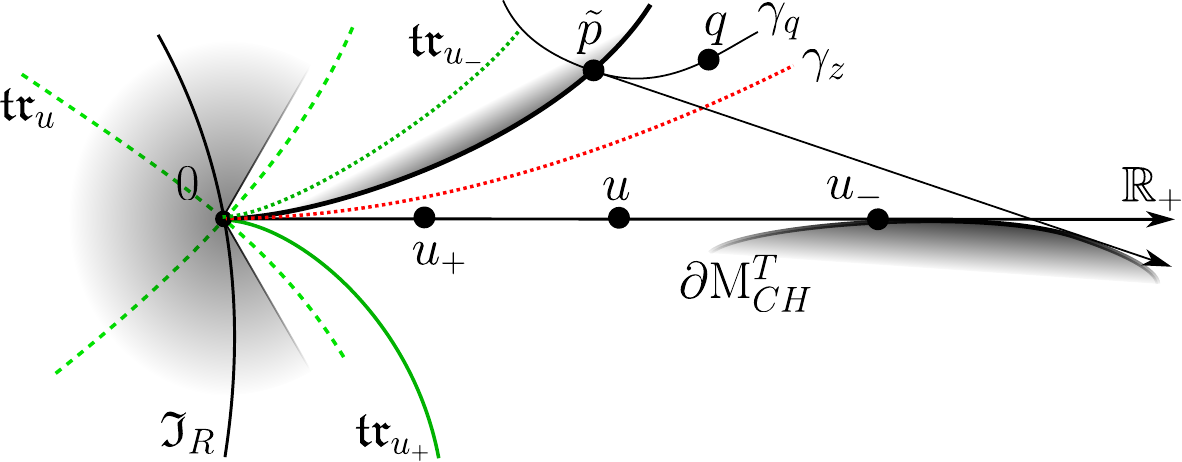}
    \caption{Bouncing type}
\end{figure}
\end{proof}

\subsection{Points of \texorpdfstring{$C^{2}$}{C2}-inflection type}\label{sub:InflectionTYPE}

Recall that a point $z\in \partial \minvset{CH}$ is called a point of \texorpdfstring{$C^{2}$}{C2}-inflection type if  $\Delta^{+}(z) = \emptyset$ and either $\Delta^{-}(z) = \emptyset$ or $\Gamma(z) \neq \emptyset$.

\begin{proposition}\label{prop:INFLECTIONTYPE}
Consider a point $p$ of $\partial\minvset{CH} \setminus \mathcal{Z}(PQ)$ belonging to the transverse locus $\mathfrak{I}_{R}^{\ast}$. If $\Delta(p) = \emptyset$, then $\Gamma(p) \neq \emptyset$ and $p$ is the starting point of a local arc.
\end{proposition}

\begin{proof}
We keep the previous normalization $p=0$, $R(0)=1$, $\Im R''(0)=b>0$.
The positive trajectory $\gamma^+_0$ cuts $\mathfrak{I}^+$ into two parts, one containing the convex hull of $\gamma_0$ (denoted by $\Omega_{++}$) and another one denoted by $\Omega_{+-}$. 
    
The arguments of Lemma~\ref{lem:localarc step 1} (using  Lemma~\ref{lem:unif horns at I_R^*}  instead of Lemma~\ref{lem:unif horns}) can be repeated word by word as long as $\tilde{p}\in\mathfrak{I}_{R}^{\ast}\cap\partial\Omega_{++}$, see Figure~\ref{fig:hornsfig1}. This proves that   $\minvset{CH}$ does not intersect $\Omega_{+-}$. Together with $\Delta(0)=\emptyset$ this implies  that some small sector $S_-=\{-\epsilon<\arg z<0\}$ does not intersect $\minvset{CH}$.
\par
Now, assume by contradiction, that there exists  $q \in \Omega_{++}\setminus \minvset{CH}$. 
 Shrinking $\Omega_{++}$ if needed we can repeat  the arguments of Lemma~\ref{lem:localarc step 2} as long as  $\tilde{p}\in\mathfrak{I^+}$, see Figure~\ref{fig:hornsfig2}, to conclude that $\Omega_{++}\cap\minvset{CH}=\emptyset$.

Therefore there is a neighbourhood $U_+$ of $0$ in $\mathfrak{I}^+$ which is disjoint from $\minvset{CH}$. We can assume that $U_+$ is the  intersection of $\mathfrak{I}^+$ with a small disk centered at $0$.
\par
For sufficiently small $\epsilon>0$, the set  $U=U_+\cup\{-\epsilon<\arg z<\epsilon\}$ is disjoint from $\minvset{CH}$;  the part lying in the lower half-plane is in $U_+\cup S_-$ and the part lying in the upper half-plane is in $U_+\cup\horn{0}{p'}{p''}\cup\horninf{p''}$.  

Now, take  a small neighborhood $U_-$ of $0$ in $\Omega_-$ bounded by a convex curve transversal to $R$. For any $u\in U_-$,  the ray $r(u)$, being close to $R_+$, lies inside $U\cup U_-$. By Lemma~\ref{lem:excision}, this implies that $U_-\subset(\minvset{CH})^c$, and therefore  $0\notin\minvset{CH}$,  a contradiction. Thus $\Omega_{++}\subset\minvset{CH}$ and $\gamma_0\subset\partial\minvset{CH}$.
%\begin{figure}
%    \centering
%    \includegraphics[width=0.75\linewidth]{HornsFigICgen.v2.pdf}
%    \caption{A point $p\in\partial\minvset{CH} \cap \mathfrak{I}_{R}^{\ast}$ with $\Delta(p)=\emptyset$  is a starting point of a local arc.}
 %   \label{fig:ICgenDelta=0}
%\end{figure}
\end{proof}

\begin{proposition}\label{prop:INFLglobal}
Consider a point $p$ of $C^{2}$-inflection type. Then there is a neighborhood $V$ of $p$ in which $\partial \minvset{CH}$ is formed by:
\begin{itemize}
    \item a portion of a local arc $\gamma_p$ parameterized by an interval $[0,\epsilon[$,  $\epsilon>0$ with $\gamma_p(0)=p$;
    \item a portion of a global arc $\tilde{\gamma}_p$ parameterized by $[0,\epsilon[$ and such that $\tilde{\gamma}_p(0)=p$ and $\Delta(\tilde{\gamma}_p(t))= \lbrace{ \gamma_p(t) \rbrace}$.
\end{itemize}
In particular, $p$ is simultaneously the starting point of a local arc and the starting point of a global arc.
\end{proposition}

\begin{proof}
We again assume that $p=0$ and $R(0)=1$. Following the definition of a point of $C^{2}$-inflection type (see Theorem~\ref{thm:MAINClassification}), we have that $\Delta^{+}(0) = \emptyset$ and either $\Delta^{-}(0)=\emptyset$ or $\Gamma(0) \neq \emptyset$. By  Proposition~\ref{prop:INFLECTIONTYPE}, if $\Delta^{-}(0)=\emptyset$, then $\Gamma(0)$ is also nonempty. Therefore, in both cases $0$ is the starting point of a local arc $\gamma_0$ and  we deduce the shape of the boundary close to $0$ from the assumption  $\Delta^{+}(0)=\emptyset$. 
\par

The curve $\gamma_0$ divides the domain $\mathfrak{I}^+$ into two parts, and as before we denote the one containing a horn of $0$ by  $\Omega_{+-}$. 

\begin{lemma}\label{lem:Omega+- in C2 is out}
 $\Omega_{+-}\cap \minvset{CH}=\emptyset$.
\end{lemma}
\begin{proof} 
At first, consider the case  $R'(0)<0$.   
Set $I_-\coloneqq\partial\Omega_{+-}\cap\mathfrak{I}_R\setminus\{0\}$.  We claim that  $r(u)\cap \minvset{CH}=\emptyset$ for all $u\in I_-$ sufficiently close to $0$.
   
   Set $\rho\coloneqq-\left(R'(0)\right)^{-1}$. By assumption $\minvset{CH}\cap \bR_+=\Delta^-(0)$ is a compact subset of $(\rho, +\infty]$, so $\minvset{CH}\cap \bR_+\subset[\rho', +\infty]$, $\rho'>\rho$. Again, by closedness of $\minvset{CH}$ and by Lemma~\ref{lem:horn out} we can assume that   for any $\epsilon>0$ and all  sufficiently small $0<\delta'<\delta'(\epsilon)$
  \begin{equation}\label{eq:c2 omega+- 2nd eq}
      \minvset{CH}\cap \{|\Im z|<\delta'\}\subset \{\Re z>\rho'-\epsilon, \Im z\le 0\}\cup \{\Re z<\epsilon\}.
  \end{equation}
This means that $(\minvset{CH})^c$ contains not only the horn $\horn{0}{p'}{p''}$ but also the box 
$$
\Pi=\{-\delta'<\Im z\le 0, \epsilon<\Re z<\rho+\epsilon\}
$$ 
(we take $\epsilon <\frac{\rho'-\rho}{2}$).

   For all  $u\in I_-$, the slope $\sigma(u)$ is positive: 
\begin{equation}\label{eq:C2 slope on I-}
    \Im R(u)=R'(0)\Im u+O(u^2)>0,
\end{equation}
   as $\Im u<0$ and $\Re u=O(\Im u)$ by transversality of $\mathfrak{I}_R$ and $r(0)$. Moreover,  by \eqref{eq:C2 slope on I-} we have $\Im (u+tR(u))=0$ for $t=\rho+O(u)$,  thus   $r(u)\cap\bR_+=\rho+O(u)$. Therefore for any $\epsilon>0$ and for any point  $u\in I_-$ sufficiently close to $0$,  the ray $r(u)$ has an  arbitrarily small slope and  $r(u)\cap\bR_+\in (\rho-\epsilon, \rho+\epsilon)$. Therefore
   \begin{equation}\label{eq:c2 omega+- 1st eq}
   r(u)\cap \{\Im z>0\}\subset \horn{0}{p'}{p''}\cup\horninf{p''}\subset (\minvset{CH})^c
   \end{equation}
   for all $u\in I_-$ sufficiently close to $0$.

   Let $u\in I_-$, $|\Im u|<\delta'$, and take $u''\in r(u)$ with $\Re u''=\epsilon$. 
   If $\epsilon$ is sufficiently small then by Lemma~\ref{lem:unif horns at I_R^*} we can assume that the horn $\horn{u}{u'}{u''}$ is small.

   The complementary cone $\horninf{u''}$ lies above $r(u)$ and to the right of $\{\Re z>\epsilon\}$. Thus
   $$
   \horninf{u''}\cap\{\Im z<0\}\subset\Pi
   $$
   and therefore $\horninf{u''}\cap\{\Im z\le0\}\subset (\minvset{CH})^c$. Choosing $\epsilon$ sufficiently small we can assume that the angle of the sector $\horninf{u''}$ is small enough to conclude from \eqref{eq:c2 omega+- 1st eq} that $\horninf{u''}\cap\{\Im z>0\}\subset (\minvset{CH})^c$ as well. This implies that  $u\in (\minvset{CH})^c$ and therefore  $\Omega_{+-}\subset (\minvset{CH})^c$ as well.

\begin{figure}
    \centering
    \includegraphics[width=1\linewidth]{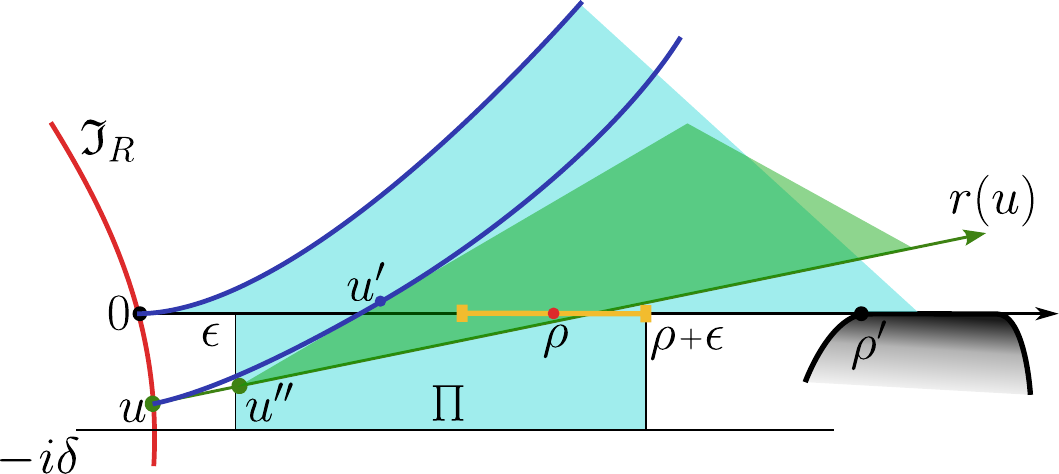}
    \caption{$\horninf{u''}$ lies in the complement to $\minvset{CH}$}
    \label{fig:C2Omega+-}
\end{figure}

If  $R'(0)\ge 0$ then $\Delta^-(0)\subset\Delta^+(0)=\emptyset$, so $\Delta(0)=\emptyset$.    Then the arguments of Lemma~\ref{lem:localarc step 1} are applicable for all $\tilde{p}\in\Omega_{+-}$, which proves  the required claim in this case as well.

\end{proof}

Let $\tilde{\gamma}_0$ be the set of points in $\mathfrak{I}^-$ whose associated rays are tangent to the positive trajectory $\gamma_0$ of $R(z)\partial_z$ starting at $0$.

\begin{lemma}\label{lem:INFRatio}
For $R$ as in \eqref{eq:R at generic IC} the local arc $\gamma_0$ is described by the relation $y(x) = \tfrac b 3 x^{3} + o(x^{3})$, $x \geq 0$ and the curve $\tilde{\gamma}_0$ is described by $y(x)=5\frac b 3 x^{3}+o(x^{3})$, $x \leq 0$.
\end{lemma}

\begin{proof}
Using \eqref{eq:gamma  in m} for \eqref{eq:R at generic IC} we see that 
$$
\gamma_0(t)=t+o(t)+i\left(\frac b 3 t^3+O(t^4)\right)
$$ 
with $a \in \mathbb{R}$ and therefore the slope $\sigma(\gamma_0(t))=bt^2+O(t^3)$.

The point $u\in\tilde{\gamma}_0$ whose associated ray $r(u)$ is tangent to $\gamma_0$ at $\gamma_0(t)$ has the form 
$$
u=\gamma_0(t)-sR(\gamma_0(t))=t-s+o(t)+i\left(-sbt^2+\frac b 3 t^3+O(t^4)\right)
$$ 
with the condition $\sigma(u)=\sigma(\gamma_0(t))=bt^2+O(t^3)$. The latter condition means that $s=2t+o(t)$ and therefore 
$$
u=-t+o(t)-i\left(\frac 5 3 bt^3+O(t^4)\right).
$$ 
\end{proof}

\begin{lemma}\label{lem:ralpha is outside C2}
    For any $u\in\tilde{\gamma}_0$, the ray $r(u)$ does not intersect $\tilde{\gamma}_0\cup\gamma_0$ between $u$ and  the point of tangency $z=z(u)$ of $r(u)$ and $\gamma_0$.
\end{lemma}
\begin{proof}
 By Lemma~\ref{lem:INFRatio}, there is function $\gamma(x)$ such that $\gamma_0\cup\tilde{\gamma}_0=\{y=\gamma(x)\}$, with $\gamma(x)=\frac b 3 x^3+o(x^3)$ for $x>0$  and $\gamma(x)=\frac {5b} 3 x^3+o(x^3)$ for $x<0$. Both expressions, being power series, can be differentiated and produce asymptotic formulae for $\gamma'(x), \gamma''(x)$ as well.
 In particular, $\gamma''(x)$ is continuous and monotone  on the interval $[\Re\alpha, \Re z]$. 
 
Assume  $r(u)\subset\{y=kx+b\}$. By construction, $\hat{\gamma}=\gamma(x)-kx-b$ vanishes at $\Re u$ and has a double zero at $\Re z$.
Any other point of intersection of $r(u)$ and $\tilde{\gamma}_0\cup\gamma_0$ will imply the  existence of another zero  of $\hat{\gamma}$ on $[\Re u,\Re z]$, thus $\hat{\gamma}$ will have four zeros on this interval counting multiplicities. By Rolle's Theorem this will imply the existence of two zeros of $\hat{\gamma}''=\gamma''$ on $[\Re u,\Re z]$, which contradicts monotonicity of $\gamma''$.
\end{proof}
By Lemma~\ref{lem:INFRatio} the curve $\tilde{\gamma}_0$ is tangent to $\bR$, thus transversal to $\mathfrak{I}_R$ and divides $\Omega_-$ into two parts. Denote by $\Omega_{--}$ the closed part consisting of points whose  associated rays do not intersect $\gamma_0$ and let $\Omega_{-+}$ denotes the second part. Clearly $\Omega_{-+}\subset\minvset{CH}$.

\begin{corollary}\label{cor:r_+(alpha) out C2}
    Let $u\in\tilde{\gamma}_0$ and set $u_+\coloneqq r(u)\cap \mathfrak{I}_R$ and  $r_+(u)\coloneqq r(u)\setminus\Omega_{--}=u_+ +R(u)\bR_+$.  Then   $r_+(u)\cap\left(\minvset{CH}\right)^\circ=\emptyset$.
\end{corollary}

\begin{proof}
    Indeed, the piece of $r_+(u)$ between $u_+$ and the point of tangency $z=z(u)$ of $r(u)$ and $\gamma_0$ lies in $\Omega_{+-}$ by Lemma~\ref{lem:ralpha is outside C2}, and the remaining piece coincides with $r(z)$. Thus the claim follows from Lemmas~\ref{lem:Omega+- in C2 is out} and ~\ref{lem:horn out}.
\end{proof}

\begin{lemma}\label{lem:ru C2}
For any $u\in\Omega_{--}\setminus\tilde{\gamma}_0$, one has 
\begin{enumerate}
    \item $r(u)\cap \tilde{\gamma}_0=\emptyset$,
    \item  $r(u)\setminus\Omega_{--}\subset (\minvset{CH})^c$.
\end{enumerate}
  
\end{lemma}
%We prove that $r(u)$ is  contained in $\Omega_{--}\cup\Omega_{+-}\cup S$ assuming that  both $\sigma(u)$ and $u$ are small (this can be achieved by taking smaller $\Omega_{--}$). 

\begin{proof}
   Let  $\gamma_\alpha^\beta$ be the piece of trajectory of $R(z)\partial_z$ containing $u$ and with endpoints  $\alpha=\alpha(u)\in \tilde{\gamma}_0$ and $\beta=\beta(u)\in\mathfrak I_R$: by Lemma~\ref{lem:INFRatio}  $\tilde{\gamma}_0$
is transversal to the trajectories of $R(z)\partial_z$ near $0$ except $\gamma_0$.
Let $D(u)=\cup_{z\in\gamma_\alpha^\beta}r(z)$ be the domain sweeped by the rays  associated to the points of  $\gamma_\alpha^\beta$. Since  $\gamma_\alpha^\beta$ is convex, $\partial D(u)=r(\alpha)\cup\gamma_\alpha^\beta\cup r(\beta)$. By Lemmas~\ref{lem:ralpha is outside C2}, ~\ref{lem:INFRatio}, and ~\ref{lem:Omega+- in C2 is out},  $\partial D(u)\cap \tilde{\gamma}_0=\{\alpha\}$, so $r(u)\cap \tilde{\gamma}_0=\emptyset$. 

The boundary of $D_+(u)=D(u)\setminus\Omega_{--}$ consists of $r(\beta)$, the piece of $\mathfrak{I}_R$ lying between $\beta$ and the point $\alpha_+ $ and $r_+(\alpha)$ which  
do not intersect $\minvset{CH}^\circ$ by Lemma~\ref{lem:Omega+- in C2 is out} and Corollary~\ref{cor:r_+(alpha) out C2}. This implies the second claim of the Lemma.
\end{proof}

\begin{proof}[Proof of Proposition~\ref{prop:INFLglobal}]
    Take some $u\in\Omega_{--}^\circ$ and let $\Omega_{--}(u)$ be the curvilinear triangle bounded by $\gamma_\alpha^\beta$, $\tilde{\gamma}_0$ and $\mathfrak{I}$. The ray $r(u)$ does not cross  $\gamma_\alpha^\beta$ by convexity and does not cross $\tilde{\gamma}_0$ by Lemma~\ref{lem:ru C2}.  Therefore it leaves the domain $\Omega_{--}(u)$ through $\mathfrak{I}$, with $r(u)\setminus \Omega_{--}(u)\subset (\minvset{CH})^c$, see Lemma~\ref{lem:ru C2}. Thus $r(u)\subset \Omega_{--}(u)\cup (\minvset{CH})^c $.
    
    As $\Omega_{--}(u')\subset\Omega_{--}(u)$ for any $u'\in \Omega_{--}(u)$, this means that $r(u')\subset \Omega_{--}(u)\cup (\minvset{CH})^c $ for all $u'\in \Omega_{--}(u)$, and therefore $\Omega_{--}\cap\minvset{CH}=\emptyset$ by Lemma~\ref{lem:excision}.

    As $\Omega_{-+}\subset\minvset{CH}$, we see that $\tilde{\gamma}_0\subset\partial\minvset{CH}$.
    
\end{proof}

 \begin{figure}
    \centering
    \includegraphics[width=0.75\linewidth]{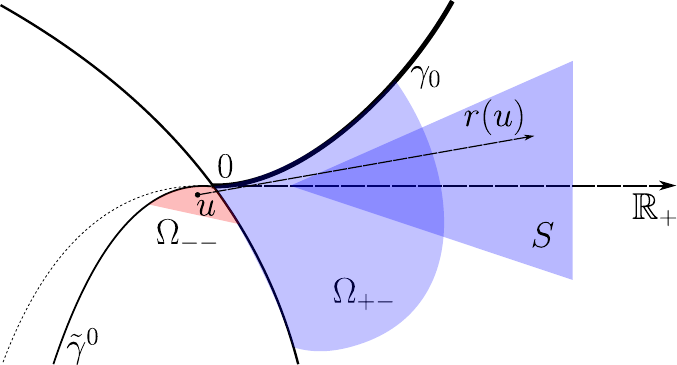}
    \caption{A point in the transverse locus of $\partial\minvset{CH} \cap \mathfrak{I}_{R}$ with empty $\Delta$ correspondence is the starting point of a global arc.}
    \label{fig:ICgenDelta=1}
\end{figure}
\end{proof}

\subsection{Points of \texorpdfstring{$C^{1}$}{C1}-inflection type}\label{sub:C1}
Recall that a point $p\in \partial\minvset{CH} \setminus \mathcal{Z}(PQ)$ belonging to the transverse locus $\mathfrak{I}_{R}^{\ast}$ is called a \emph{point of \texorpdfstring{$C^{1}$}{C1}-inflection type} if  $\Delta^{+}(p) =\Gamma(p) = \emptyset$ and $\Delta^{-}(p) \neq \emptyset$.  

\begin{proposition}\label{prop:C1INFLECTION}
A point $p\in\partial\minvset{CH}\cap\mathfrak{I}_R^{\ast}$ of $C^{1}-$inflection type is the starting point of two global arcs (one in each of the incident inflection domains).
\end{proposition}

\begin{proof}
We assume that $p=0$, $R(0)=1$ and $\gamma_0\subset\mathfrak{I}^+$ lies in the upper half-plane. We use $z=x+iy$ notations.
   
Exactly as in the case of bouncing type, the conditions $\Delta^-\not=\emptyset$ and $\Gamma=\emptyset$ imply that $0$ is a starting point of a global arc, see Lemma~\ref{lem:Delta- not empty Gamma is empty bouncing}. Denote this arc by $\eta=\{y={\xi}(x)\}$.
    The arguments of Lemma~\ref{lem:Delta- not empty Gamma is empty bouncing} show that $\eta$ lies between $\mathfrak{tr}_{+\infty}$ and  $\mathfrak{tr}_{u_\infty}$, where $u_\infty=\sup \Delta^-$. Thus $\eta$ has second order tangency with $\bR_+$. 
    Let $k(x)$ be the point of intersection of $r\left((x, \xi(x))\right)$ with $\bR_+$. One can easily see that $k(x)$ is an increasing function, and $k(x)\to u_{\infty}$ as $x\to 0$. This means that $\lim_{x\to0+}\xi(x)/x^2$ exists and is positive.

We repeat the arguments of the $C^2$-inflection case above one by one, with $\gamma_0$ replaced by $\eta$. 
The Lemma~\ref{lem:Omega+- in C2 is out} can be repeated  verbatim (necessarily $R'(0)<0$),  and we get $\Omega_{+-}\subset\minvset{CH}$.

Let $\tilde{\eta}\subset\mathfrak{I}^-$ be the curve of points bounding (the germ at $0$ of) the domain $\Omega_{--}$ consisting of points of $\mathfrak{I}^-$ whose associated rays do not intersect  $\eta$. In particular, $r(u)$ is a supporting line to $\eta$ for all $u\in\tilde{\eta}$.
If $\gamma_u^\beta$ is a trajectory of $R(z)\partial_z$ starting at $u\in\tilde{\eta}$ and ending at $\beta\in \mathfrak{I}_R$ then $\gamma_u^\beta\subset\Omega_{--}$ by convexity of  $\gamma_u^\beta$, as in Lemma~\ref{lem:ru C2}.

\begin{lemma}
For any $u\in\tilde{\eta}$ the ray $r(u)$ does not intersect $\tilde{\eta}$ except at $u$ itself.
\end{lemma}

\begin{proof}
Note first that $\eta$ lies above $r(u)$ for any $\alpha\in\tilde{\eta}$.
Assume the opposite to the claim of the lemma, and let $u'\in r(u)\cap\tilde{\eta}$ be such a point of intersection.
    If $\sigma(u')<\sigma(u)$ then the ray $r(u')$ lies below $r(u)$ and therefore cannot intersect $\eta$. Similarly, if $\sigma(u')>\sigma(u)$
    then $r(u)$ cannot intersect $\eta$. Therefore $\sigma(u')=\sigma(u)$, so $u'$ is a point of tangency of $r(u)$ and $R(z)\partial_{z}$ just like $u$.

Now, recall that $0\in\eta$ lies above $r(u)$. Moreover, as $u\in\mathfrak{I}^-$, the germ of $r(u)$ at $u$ lies in $\Omega_{--}$: for $v\in r(u)$ close to $u$ we have $\sigma(v)<\sigma(u)$, so $r(v)$ doesn't intersect $\eta$. Thus the germ of $\tilde{\eta}$ at $u$ lies above $r(u)$ as well.  Thus both ends of the part $\tilde\eta_u^0$ of   $\tilde\eta$ between $u$ and $0$ lie above $r(u)$, so $\tilde\eta_u^0$ intersects  $r(u)$ in even number of points not including $u$, i.e. at least at three points including $u$, all of them the points of tangency of $r(\alpha)$ and $R(z)\partial_z$.

As $u\to 0$ the line $r(u)$ converges to $\bR$ and the  points of tangency  necessarily converge to $0$ as well. This  means that $R\partial_z$ has point of tangency of order at least $3$ with $\bR$ at $0$. However, $p\in \mathfrak{I}_R^{\ast}$ means that $\Im R''(0)\neq 0$, so the order of tangency of $R\partial_z$  with $\bR$ at $0$ is  exactly $2$ (as $\Im R(0)=\Im R'(0)=0$). This contradiction proves the Lemma.
    
\end{proof}

The same arguments as in Lemma~\ref{lem:ru C2} and in the proof of Proposition~\ref{prop:INFLglobal} now show that $\Omega_{--}$ satisfies the conditions of Lemma~\ref{lem:excision} and is therefore disjoint from $\minvset{CH}$.
\end{proof}

\subsection{Points of switch type}\label{sub:Switch}

Recall that a point $p\in \partial\minvset{CH} \setminus \mathcal{Z}(PQ)$ belonging to the transverse locus $\mathfrak{I}_{R}^{\ast}$ is called a \emph{point of switch type} if  $\Delta^{+}(p) \neq \emptyset$ and $\Gamma(p) \cup \Delta^{-}(p) = \emptyset$.

\begin{proposition}\label{prop:SWITCH-LOCAL}
 The negative part $\gamma_0^-(t)$ of the integral curve $\gamma_0(t)$ is a part of the boundary  $\partial \minvset{CH}$.
\end{proposition}

    As before, we assume that $p=0$, $R(0)=1$ and $\Im R''(0)>0$. 
    The trajectory $\gamma_0(t)$ and the curve $\mathfrak{I}_R$ divide $U$ into four domains $\Omega_{\pm,\pm}$, with $\Omega_{+-}$ intersecting the $\horn{0}{p'}{p''}$.

 \begin{lemma} In the above notation, 
      $\Omega_{++}\subset(\minvset{CH})^c$.
 \end{lemma}
\begin{proof}
    
Recall that by Lemma~\ref{lem:horn out} the union $H(0)=\horn{0}{p'}{p''}\cup\horninf{p''}$ does not intersect $\minvset{CH}$.

 Assume first that $R'(0)\le0$ and set $\rho\coloneqq-(R'(0))^{-1}\in \R_+\cup\{+\infty\}$.
The set $\Delta(0)$ is a relatively closed subset of $(0,\rho]$ not containing $\rho$ (as $\Delta^-(0)=\emptyset$). Thus there exists a $\rho'<\rho$ such that the ray  $[\rho', +\infty]$ is disjoint from  $\minvset{CH}$. As both $[\rho', +\infty]$ and $\overline{\minvset{CH}}$ are closed in $\bC\cup\mathbb{S}^1$ there is some neighborhood $\widetilde{U}$ of $[\rho', +\infty]$ in $\bC\cup\mathbb{S}^1$ disjoint from  $\overline{\minvset{CH}}$.
Choose  an  open sector $S\subset\tilde U$ with  the ray  $(\rho', +\infty)$ as bisector,  necessarily disjoint from $\minvset{CH}$.

    By shrinking  $\Omega_{\pm\pm} $ we can assume that  $|\sigma(z)|< |\sigma(S)|$ for all $z\in \Omega_{\pm\pm} $, where $\pm\sigma(S) $ are the slopes of the sides of $S$. Moreover, we restrict ourselves to a neighborhood $U$ of $0$ so small that  rays $r(z)$ do not intersect the interval $[0, \rho']$ for $z\in U\cap\{ \Re z \ge0\}$ and $z\neq 0$: this is possible since   the trails of these points  lie in the lower half-plane.
    In particular,    for any  $z\in \Omega_{++}\subset U\cap\{ \Re z \ge0\}$  the ray $r(z)$ intersects the boundary of $H(0)\cup S$ only once at $\gamma_0^+(t)$.

    The same conclusion holds in the case $R'(0)>0$. In this case $\sigma(z)>0$ for all $z\in U$, $\Im z>0$. 
    
    Since $\Gamma(0)=\emptyset$ there exists a point $q\in \Omega_{++}\setminus\minvset{CH}$. We can now repeat  the  arguments of Lemma~\ref{lem:localarc step 2} for $U$: the crucial fact used in Lemma~\ref{lem:localarc step 2} was that $r(\tilde p)$ doesn't intersect $\minvset{CH}$, and this holds for points of $\Omega_{++}$. Therefore  $\Omega_{++}\cap\minvset{CH}=\emptyset$.

\end{proof}

    \begin{proposition}\label{prop:H(p) is outside} 
      Take $p=\gamma_0(t)$ for some $t<0$ sufficiently close to $0$. The  open curvilinear triangle  $H(p)\subset\Omega_{-+}$ bounded by $r(p)$, the curve $\gamma^-_0(t)$ and the inflection curve lies outside $\minvset{CH}$.
    \end{proposition}
    Denote $x=\Re z$, $y=\Im z$.
    \begin{lemma}\label{lem:mono in t} One has 
       $\frac{\partial}{\partial x}\sigma(x+iy)<0$ as long as $z=x+iy$ lies in a sufficiently small sector $\{|z|<\delta, x<0, |y|<-\epsilon x\}$ for some $\epsilon,\delta>0$ depending on the rational function $R(z)$ only.
    \end{lemma}
    \begin{proof}
        Recall that $R(z)=1+az+bz^2+..$, with  $\Im b>0$. Then $$\log R=az+\left(b-\frac{a^2}{2}\right)z^2+...$$ and $$\left(\log R\right)'=a+(2b-a^2)z+...$$
        Therefore 
        $$
        \frac{\partial}{\partial x}\sigma(x+iy)= \frac{\partial}{\partial x}\Im\left(\log R\right)=\Im\left(\log R\right)'=\Im  (2b-a^2)z+ o(z) <0
        $$
         if  $|\arg z-\pi|< \pi-\arg(2b-a^2)$ and $|z|$ is sufficiently small (note that  $0<\arg(2b-a^2)<\pi$ as $\Im b>0$ and $a\in\bR$).
    \end{proof}
    \begin{lemma}\label{lem:H(p)1} The associated ray 
        $r(z)$ does not intersect $\gamma_0^-(t)$ for any $z\in H(p)$.
    \end{lemma}
    \begin{proof}
        First consider  the case $\Im z<0$. As $\gamma_0^-(t)$ is tangent to $\bR$, we can assume that this part of $H(p)$ satisfies the conditions of Lemma~\ref{lem:mono in t}, so $\sigma(z)>\sigma(z_+)>0$, where $z_+=z+t\in\gamma_0^-(t)$. Thus  the ray $r(z)$ lies in the half-plane bounded by the line tangent to $\gamma_0^-(t)$ at $z_+$ and containing $z$. Therefore $r(z)$ does not intersect $\gamma_0^-(t)$: by convexity of $\gamma_0^-(t)$it lies in the other half-plane bounded by the line tangent to $\gamma_0^-(t)$ at $z_+$.

        Now, assume that $\Im z\ge0$. Recall that we have chosen $U$ so small that for any $z\in U$, $\Im z>0$, the intersection $r(z)\cap \R\subset (\rho',+\infty)\subset\R_+$.  Thus $r(z)\cap \{\Im w<0\}\subset \{\Re z>\rho'>0\}$ which is disjoint from $\gamma^-$.
    \end{proof}
 \begin{lemma}\label{lem:H(p)2} The associated ray 
        $r(z)$ does not intersect $r(p)$ for any $z\in H(p)$.
    \end{lemma} 
\begin{proof}
    Let $\gamma^-_z$ be a part of the integral curve of $R \partial_z$ ending at $z$ and let $z'\in\gamma^-_z\cap r(p)$ be the first (from $z$) point where $\gamma^-_z$ enters $H(p)$. This point must necessarily lie on $r(p)$ because  $\gamma^-_z$ intersects neither the curve of inflection points $\mathfrak{I}_R$ nor $\gamma^-$ by uniqueness of solutions of ODE.  Thus  $\sigma(z')\le\sigma(p)$.

     Assume now that the ray $r(z)$ intersects $r(p)$. Then $\sigma(z)>\sigma(p)\ge \sigma (z')$ as $z$ lies below $r(p)$. Therefore the slope $\sigma(w)$, $w\in\gamma^-_z$, is not monotone, implying that  $\gamma^-_z$ has an inflection point. But this is impossible since $\gamma^-_z$ does not intersect the  curve of inflections.
\end{proof}

\begin{proof}[Proof of Proposition~\ref{prop:H(p) is outside}]
       By minimality, it is enough to prove that $r(z)\subset H(p)\cup (\minvset{CH})^c$ for any $z\in H(p)$. 
       By Lemmas~\ref{lem:H(p)1},~\ref{lem:H(p)2}  the ray $r(z)$ does not intersect $r(p)$ and $\gamma_0^-(t)$. Thus $r(z)$ leaves $H(p)$ through the curve of inflections $\mathfrak{I}_R$ with a small slope. 
       
       If the slope is positive then $r(z)\setminus H(p)\subset U_{++}\cup H(0)$. In particular, this is the case  for all $z\in H(p)$, $\Im z<0$ by Lemma~\ref{lem:mono in t}.

       Assume now that $\sigma(z)<0$ (and therefore $\Im z>0$). Recall that we have chosen  $U$ so small that for any $z\in U$, $\Im z>0$, the intersection $r(z)\cap \R\subset (s',+\infty)\subset\R_+$. Thus $r(z)\setminus H(p) \subset \Omega_{++}\cup H(0)\cup S$. 

       Taken together, $r(z)\subset H(p)\cup (\minvset{CH})^c$ for all $z\in H(p)$. Therefore  $ H(p)\subset (\minvset{CH})^c$ by Lemma~\ref{lem:excision}.
        
    \end{proof}

\begin{proposition}\label{prop:SWITCHREGULAR}
Consider a point $p\in \partial\minvset{CH}$ of switch type. Then there is a neighborhood $V$ of $p$ such that $\minvset{CH} \cap V$ is contained in a half-disk centered at $p$. Besides, no neighborhood of $p$ in $\minvset{CH}$ can be contained in a cone centered at $p$ with an angle strictly smaller than $\pi$. A point of switch type is the ending point of both a local  and a global arcs.
\end{proposition}

\begin{proof}
We essentially repeat the arguments of Lemma~\ref{lem:localarc step 1}. Take $u\in\Delta(0)\neq\emptyset$ and let $\ttt_u^+$ be a germ of the branch of $\ttt_u\subset\minvset{CH}$ lying in the lower-right quadrant.  For any 
$q\in\ttt_u^+$   let $z=z(q)\in \gamma_{q}^+$ be the last point such that $r(z)\cap \minvset{CH}\neq\emptyset$. This point exists since $\gamma^+_q\neq\emptyset$ eventually enters $\horn{0}{p'}{p''}\subset(\minvset{CH})^c$.
As in Lemma~\ref{lem:Delta- not empty Gamma is empty bouncing},  $z\in\partial\minvset{CH}$ and therefore  $\Im z<0$, as otherwise $z\in\horn{0}{p'}{p''}$. The points $z(q)$ form a global arc ending at $0$.
\end{proof}

\subsection{Classification of boundary points}\label{sub:classification}

Here, we summarize several results obtained in the previous sections to finalize our classification of boundary points on $\partial\minvset{CH}$.

\begin{proof}[Proof of Theorem~\ref{thm:MAINClassification}]
It follows from Corollary~\ref{cor:inflectionmultiple} that there are at most $4 \deg P +  \deg Q -2 \leq 2d$ singular points in the curve of inflections ($d=3\deg P + \deg Q -1$). Proposition~\ref{prop:tangencybound} proves that the tangency locus $\mathfrak{T}_{R}$ is formed by at most $2d^{2}$ isolated points and $d$ lines.
\par
The classification of points in $\mathfrak{I}_{R}^{\ast}$ is trivial. If $\Delta^{+}$ is nonempty, then a point is of bouncing or switch type depending on whether $\Gamma \cup \Delta^{-}$ is empty or not. If $\Delta^{+}$ is empty, then a point is of $C^{1}$-inflection or $C^{2}$-inflection type depending whether the conjunction of $\Delta^{-} \neq \emptyset$ and $\Gamma = \emptyset$ is satisfied or not.
\par
Finally, for points outside $\mathcal{Z}(PQ) \cup \mathfrak{I}_{R}$, we just have to check that $\Gamma$ and $\Delta$ cannot both be empty. This is proved in Proposition~\ref{prop:localarc}.
\end{proof}

\subsection{Estimates concerning local and global arcs}\label{sub:estimateArc}

We introduce the following notations:
\begin{itemize}
    \item $|\mathcal{L}|$ is the number of local arcs;
    \item $|\mathcal{G}|$ is the number of global arcs;
    \item $|\mathcal{B}|$ is the number of points of bouncing type;
    \item $|\mathcal{E}|$ is the number of points of extruding type;
    \item $|\mathcal{I}_{1}|$ is the number of points of $C^{1}-$inflection type;
    \item $|\mathcal{I}_{2}|$ is the number of points of $C^{2}-$inflection type;
    \item $|\mathcal{S}|$ is the number of points of switch type.
\end{itemize}

We prove that the number of points of switch type provides an estimate for the number of local arcs (up to an error term depending only on $\deg P$ and $\deg Q$).

\begin{lemma}\label{lem:EstimateLocal}
In the boundary $\partial\minvset{CH}$, the ending point of every local arc, except at most $d(2d+1)$ of them, is a point of switch type where $d= 3\deg P + \deg Q -1$. Conversely every point of switch type is the endpoint of some local arc.\newline
In other words, we have $|\mathcal{S}| \leq |\mathcal{L}| \leq |\mathcal{S}|+d(2d+1)$.
\end{lemma}

\begin{proof}
Proposition~\ref{prop:SWITCHREGULAR} proves that every point of switch type is the endpoint of some local arc. It remains to list all possible endpoints for  local arcs.
\par
Following Proposition~\ref{prop:localEND}, every local arc has an endpoint that either belongs to $\mathcal{Z}(PQ)$ or to $\mathfrak{I}_{R}$. For any point $\alpha$ which is the endpoint of a local arc, $\mathcal{L}_{\alpha}$ contains an interval of length at most $\pi$. It follows from Corollary~\ref{cor:KLalpha} that such a point is either a simple pole of $R(z)$ or a point which is neither a zero or a pole of $R(z)$. Only two local arcs can have the same simple pole as their  endpoint. Any other point is the endpoint of at most one local arc (because only one integral curve passes through such a point). Consequently, at most $3\deg P + \deg Q$ local arcs have an endpoint in $\mathcal{Z}(PQ)$.

\par
It remains to count local arcs one endpoint of which belongs to $\mathfrak{I}_{R} \setminus \mathcal{Z}(PQ)$. Any such point is incident to a unique integral curve which implies that  it can be the endpoint of only one local arc. If such a point belongs to the transverse locus of the curve of inflections, then it is a point of switch type (see Propositions~\ref{prop:bouncing},~\ref{prop:INFLglobal},~\ref{prop:C1INFLECTION} and~\ref{prop:SWITCHREGULAR}). There are $|\mathcal{S}|$ of them. Following Proposition~\ref{prop:tangencybound}, the tangency locus of $\mathfrak{I}_{R}$ is formed by at most $2d^{2}$ points and $d$ straight lines where $d=3\deg P + \deg Q -1$. As $\arg(R(z))$ is constant on each such line, they contain at most one endpoint of a local arc. Thus  $|\mathcal{L}| \leq |\mathcal{S}|+d(2d+1)$.

\end{proof}

Similarly, we prove an estimate on the number of global arcs that do not start at a point of the transverse locus of the curve of inflections.

\begin{lemma}\label{lem:EstimateGlobal}
In the boundary $\partial\minvset{CH}$, the starting point of every global arc, except at most $12d+5d^{2}$ of them, is either a point of $C^{1}$-inflection, $C^{2}$-inflection or of bouncing type.\newline
Besides, we have $2|\mathcal{I}_{1}|+|\mathcal{I}_{2}| \leq |\mathcal{G}| \leq |\mathcal{B}|+2|\mathcal{I}_{1}|+|\mathcal{I}_{2}|+12d+5d^{2}$.
\end{lemma}

\begin{proof}
Proposition~\ref{prop:INFLglobal} shows that every point of $C^{2}$-inflection is the starting point of a global arc while Proposition~\ref{prop:C1INFLECTION} shows that every point of $C^{1}$-inflection is the starting point of two global arcs. It follows that $2|\mathcal{I}_{1}|+|\mathcal{I}_{2}| \leq |\mathcal{G}|$.
\par
Then, we list every possible starting point for a global arc (Lemma~\ref{lem:globalorientation} proves that global arcs cannot be closed loops).
\par
Since points of extruding type are not starting points of global arcs (see Proposition~\ref{prop:extruding}), every global arc either starts at a point at infinity or  at a point of $\mathcal{Z}(PQ) \cup \mathfrak{I}_{R}$. 
\par
We first count the number of global arcs which can start at infinity. If $\deg Q - \deg P =1$, then by  Theorem~\ref{thm:main} we know  that $\minvset{CH}$ is compact. If $\deg Q - \deg P =-1$, then Proposition~\ref{prop:NOFINGER} proves that $\minvset{CH}$ has only one connected component while its complement has two connected components. Therefore, we have at most four infinite global arcs in this case. If $\deg Q - \deg P =0$, the complement of $\minvset{CH}$ is connected so each connected component has at most two infinite global arcs. Following Proposition~\ref{prop:CCrawBOUND}, $\minvset{CH}$ has at most $\deg P + \deg Q$ connected components.  Therefore the number of global arcs starting at infinity is at most $2\deg P + 2\deg Q$. In the only case where $4 > 2\deg P + 2\deg Q$ while $\deg Q - \deg P =-1$, $Q(z)$ is constant while $\deg P = 1$. In this case, $\minvset{CH}$ is a straight line (see Proposition~\ref{prop:qp-1Convex}).
\par
Let us  consider the points of the transverse locus of $\mathfrak{I}_{R}$. Each point of $C^{1}-$inflection type is the starting point of exactly two global arcs (see Proposition~\ref{prop:C1INFLECTION}). Each point of bouncing or $C^{2}-$inflection type is the starting point of exactly one global arc (see Propositions~\ref{prop:bouncing} and~\ref{prop:INFLglobal}). No global arc starts at a point of switch type (see Proposition~\ref{prop:SWITCHREGULAR}).
\par
Now let us  consider the tangency locus of $\mathfrak{I}_{R}$. It is formed by at most $2d^{2}$ points and $\deg P + \deg Q + 1$ $R$-invariant lines (see Definition~\ref{def:Rinvariant}). For each such line, at most 4 global arcs can have start accumulation $\omega_-(\alpha)$ belonging to it, because each line has two sides and rays have two possible directions. Otherwise, the rays starting from these global arcs intersect the interior of $\minvset{CH}$ near the other global arcs. Using Corollary~\ref{cor:BOUNDINTERIOR}, we conclude that each of the $2d^{2}$ remaining points of the tangent locus is the starting point of at most two global arcs. For the same reasons, each singular point of $\mathfrak{I}_{R}$ that does not belong to $\mathcal{Z}(PQ)$ is the starting point of at most two global arcs. There are $4 \deg P +  \deg Q -2 \leq 2d$ such points (see Corollary~\ref{cor:inflectionmultiple}).
\par
It remains to estimate the number of global arcs that can start at a root $\alpha$ of $P(z)$ or $Q(z)$ in terms of the local degree $m_{\alpha}$ of $R(z)$ in $\alpha$. Corollary~\ref{cor:BOUNDINTERIOR} proves that $\alpha$ is the starting point of at most:
\begin{itemize}
    \item two arcs if $m_{\alpha}=0$;
    \item $2(1-m_{\alpha})$ arcs if $m_{\alpha} \leq -1$;
    \item $2\deg P$ arcs if $m_{\alpha} \geq 1$.
\end{itemize}
Consequently, in the worst case scenario, the roots of $P(z)$ and $Q(z)$ are simple and disjoint so at most $2\deg P (\deg P + 2)$ global arcs can start at these points.
\par
Therefore, the number of global arcs whose starting point does not belong to $\mathfrak{I}_{R}^{\ast}$ is at most $(2 \deg P + 2\deg Q) + 4(\deg P + \deg Q + 1) + 4d^{2} + 4d + 2\deg P (\deg P + 2)$.
\par
If $\deg P =0$, then $\deg Q =1$ (otherwise $\minvset{CH}$ is trivial) and $\minvset{CH}$ is fully irregular (and has therefore no global arcs) so we can replace the obtained bound by the slightly weaker (but more practical) upper bound $12d+5d^{2}$.
\end{proof}

\subsection{Long arcs}\label{sub:longarc}

In order to prove Theorem~\ref{thm:MainBound}, we introduce a new decomposition of the boundary $\partial \minvset{CH}$.

\begin{definition}\label{defn:longarcs}
For any linear differential operator $T$ given by \eqref{eq:1stN}, we define the \textit{long arcs} as the maximal arcs formed by gluing local and global arcs along points of extruding or bouncing type (see Sections~\ref{sub:extruding} and~\ref{sub:bouncing}).
\end{definition}

In particular, a long arc belongs to the closure of a unique inflection domain. Consequently, local and global arcs of a same long arc share the same orientation (see Section~\ref{sub:orientationglobal}). This defines the orientation a long arc.

\subsubsection{Estimates concerning long arcs}

Drawing on the estimates of Section~\ref{sub:estimateArc}, we prove that the number of long arcs corresponds to the number of intersections between $\partial \minvset{CH}$ and the transverse locus of $\mathfrak{I}_{R}$ that are not of bouncing type (in other words, where the boundary of the minimal set crosses the curve of inflections).

\begin{lemma}\label{lem:LongArcs}
Every long arc except at most $28d^{2}+52d$ of them goes from a point of switch type to a point of $C^{1}$-inflection or $C^{2}$-inflection type.
The number $|\mathcal{A}|$ of long arcs satisfies the following inequalities:
$$
2|\mathcal{S}| \leq |\mathcal{A}| \leq 2|\mathcal{S}| + 26d+14d^{2}
;
$$
$$
2|\mathcal{I}_{1}|+2|\mathcal{I}_{2}| \leq |\mathcal{A}| \leq 4|\mathcal{I}_{1}|+2|\mathcal{I}_{2}|+ 26d+14d^{2}
.
$$
\end{lemma}

\begin{proof}
Inequality $2|\mathcal{S}| \leq |\mathcal{A}|$ follows from the fact that every point of switch type is the ending point of two long arcs. Similarly, points of $C^{1}$-inflection or $C^{2}$-inflection type are the starting points of two long arcs so we obtain $2|\mathcal{I}_{1}|+2|\mathcal{I}_{2}| \leq |\mathcal{A}|$.
\par
We denote by $|\mathcal{A}_{L}|$ the number of long arcs containing a local arc. We already know that the endpoint of a local arc cannot be a point of extruding or bouncing type. Therefore, the endpoint of a local arc contained in a long arc is also the endpoint of the long arc. We deduce then from Lemma~\ref{lem:EstimateLocal} that the endpoints of all but at most $d(2d+1)$ these long arcs containing a local arc are points of switch type: $|\mathcal{A}_{L}| \leq |\mathcal{S}|+d(2d+1)$.
\par
Then, denote by $|\mathcal{A}_{G}|$ the number of long arcs that do not contain a local arc. The start accumulations of these long arcs are in particular start accumulations of a global arc and cannot be points of bouncing type. We deduce from the proof of Lemma~\ref{lem:EstimateGlobal} that the start accumulations of these long arcs, except at most $12d+5d^{2}$ of them are in fact starting points and they are points of $C^{1}$-inflection or $C^{2}$-inflection type. In other words, we have $|\mathcal{A}_{G}| \leq 2|\mathcal{I}_{1}|+|\mathcal{I}_{2}|+12d+5d^{2}$.
\par
Finally, we deduce that the total number of long arcs satisfies $|\mathcal{A}| \leq |\mathcal{S}|+2|\mathcal{I}_{1}|+|\mathcal{I}_{2}|+13d+7d^{2}$. Combining this inequality with $2|\mathcal{S}| \leq |\mathcal{A}|$, we obtain that $|\mathcal{S}| \leq 2|\mathcal{I}_{1}|+|\mathcal{I}_{2}|+13d+7d^{2}$ and therefore $|\mathcal{A}| \leq 4|\mathcal{I}_{1}|+2|\mathcal{I}_{2}|+26d+14d^{2}$. We obtain similarly that $|\mathcal{A}| \leq 2|\mathcal{S}| + 26d+14d^{2}$. This also provides a bound on the number of long arcs that do not go from a point of switch type to a point of $C^{1}$-inflection or $C^{2}$-inflection type.
\end{proof}

\subsubsection{Admissible long arcs}

We will refer to a long arc going from a point of switch type to a point of $C^{1}$-inflection or $C^{2}$-inflection type (or the opposite) as an \textit{admissible long arc}. 

\begin{definition}\label{defn:symbol}
We associate to each admissible long arc $\alpha$ a combinatorial symbol $s_{\alpha}$ that contains the following information:
\begin{itemize}
    \item the connected component of the transverse locus $\mathfrak{I}_{R}^{\ast}$ containing the starting point of $\alpha$;
    \item the connected component of $\mathfrak{I}_{R}^{\ast}$ containing the endpoint of $\alpha$;
    \item the sign of the inflection domain $\alpha$ belongs to.
\end{itemize}

Chains of consecutive long arcs define \textit{patterns} formed by the concatenation of the combinatorial symbols of their long arcs. 
\end{definition}

Chains of consecutive long arcs are given the orientation induced by the topological orientation of $\minvset{CH}$. Therefore, the orientation of a chain coincides with the orientation of the long arc inside positively oriented domains of inflection (and does not coincide with the orientation of the long arc inside negatively oriented domains of inflections).

\begin{remark}\label{rem:ChainOrientation}
In particular, at a point $z$ of switch, $C^{1}$-inflection or $C^{2}$-inflection type in a chain, the associated ray $r(z)$ and the orientation of the chain points towards the same domain of inflection.
\end{remark}

\subsubsection{Bounding the number of transverse intersection points between \texorpdfstring{$\partial \minvset{CH}$}{dMCH} and the transverse locus of \texorpdfstring{$\mathfrak{I}_{R}$}{IR}}\label{sub:bound}

Using the fact that an associated ray cannot cross the curve of inflections more than $d$ times (where $d=3\deg P +\deg Q -1$), we prove a bound on the number of chains of admissible long arcs that can realize a given pattern.

\begin{lemma}\label{lem:LONGintersection}
In the boundary $\partial \minvset{CH}$ of the minimal set, there cannot be $2d+2$ disjoint chains of $2d$ admissible long arcs that realize the same pattern.
\end{lemma}

\begin{proof}
We assume by contradiction the existence of $2d+2$ disjoint chains $\gamma_{1},\dots,\gamma_{2d+2}$ of $2d$ admissible long arcs realizing the same pattern. We refer to the admissible long arc of $\gamma_{i}$ corresponding to the $j^{th}$ symbol as $\alpha_{i,j}$.
\par
By definition of the combinatorial symbol, for a given $j$, the arcs $\alpha_{i,j}$ for $1 \leq i \leq 2d+2$ belong to the same inflection domain $\mathcal{D}_{j}$.
\par
We denote by $\beta_{1},\dots,\beta_{2d+1}$ the connected components of the transverse locus $\mathfrak{I}_{R}^{\ast}$ at the endpoints of long arcs ordered according to the orientation of the chains.
\par
We are going to prove the existence of a point $z$ of $\beta_{d+1}$ such that its associated ray $r(z)$ also intersects $d$ arcs among $\beta_{1},\dots,\beta_{2d+1}$, defining a straight line that intersects transversely a real algebraic curve of degree at most $d$ in at least $d+1$ points, obtaining the desired contradiction.
\par
A first observation is that the chains $\gamma_{1},\dots,\gamma_{2d+2}$ cannot cross each other (because the interior of $\minvset{CH}$ is connected near endpoints of admissible arcs). Since the complex plane is simply connected, this fact implies that for a given $j$, the intersection points of the chains $\gamma_{1},\dots,\gamma_{2d+2}$ with $\beta_{j}$ determine a cyclic order that does not depend on $j$. We assume therefore that the indices of $\gamma_{1},\dots,\gamma_{2d+2}$ are elements of $\mathbb{Z}/(2d+2)\mathbb{Z}$ and correspond to the previously defined cyclic ordering.
\par
In a given inflection domain $\mathcal{D}_{j}$, it may happen that the linear orders of these intersection points with $\beta_{j}$ and $\beta_{j+1}$ respectively are different. They may differ by a rotation if $\beta_{j}$ and $\beta_{j+1}$ do not belong to the same connected components of the boundary of $\partial \mathcal{D}_{j}$ (if $\mathcal{D}_{j}$ is not simply connected).
\par
It follows that for any $1 \leq j \leq 2d$, for every $k \in \mathbb{Z}/(2d+2)\mathbb{Z}$ except possibly one, there is a quadrilateral in $\mathbb{C}$ bounded by $\alpha_{j,k}$, $\alpha_{j,k+1}$, one portion of $\beta_{j}$ and one portion of $\beta_{j+1}$. The exception correspond to the case where the linear orderings do not match. Since we have $2d+2$ chains, it follows that we can assemble $2d$ of these strips into a unique long strip $\mathcal{S}$ bounded by some $\gamma_{k}$, $\gamma_{k+1}$ and portions of $\beta_{0}$ and $\beta_{2d+1}$.
\par
The chains $\gamma_{k}$ and $\gamma_{k+1}$ are oriented in such a way that for any point $z \in \partial \mathcal{S} \cap \beta_{0}$, associated ray $r(z)$ points inside $\mathcal{S}$. Since the orientation of $\gamma_{k}$ and $\gamma_{k+1}$ coincides with the topological orientation of $\minvset{CH}$, we can find a point $z$ of $\partial \mathcal{S} \cap \beta_{d+1}$ in the complement of $\minvset{CH}$. Thus, for any such point $z$, associated ray $r(z)$ cannot cross chains $\gamma_{k}$ and $\gamma_{k+1}$ and has to leave $\mathcal{S}$ through $\beta_{0}$ or $\beta_{2d+1}$. Therefore, $r(z)$ has to cross either $\beta_{1},\dots,\beta_{d+1}$ or $\beta_{d+1},\dots,\beta_{2d+1}$. This ends the proof.
\end{proof}

We deduce a bound on the number of long arcs.

\begin{corollary}\label{cor:longarc}
For any linear differential operator $T$ given by \eqref{eq:1stN}, the number $|\mathcal{A}|$ of long arcs in $\partial \minvset{CH}$ satisfies
$$
|\mathcal{A}| \leq 2e^{16d\ln(d)}+92d^{3}
$$
where $d=3\deg P +\deg Q -1$.
\end{corollary}

\begin{proof}
Since the number of connected components of the transverse locus $\mathfrak{I}_{R}^{\ast}$ is at most $2d^{2}+6d+2$ (see Corollary~\ref{cor:transverseCOMPONENT}), the number of possible combinatorial symbols for an admissible long arc is $2(2d^{2}+6d+2)^{2}$ (see Definition~\ref{defn:symbol}). Therefore, the number of possible patterns for a chain of $2d$ admissible long arcs is at most $2^{2d}(2d^{2}+6d+2)^{4d}$. Since $d \geq 3$ (see Remark~\ref{rem:degree}), we have $2d^{2}+6d+2 \leq \frac{5}{2}d^{2}$ and we obtain the weaker (but simpler) upper bound $(25/2)^{2d}d^{8d}$.
\par
Then, using Lemma~\ref{lem:LONGintersection}, we deduce that the number of disjoint chains of $2d$ admissible long arcs is at most $(25/2)^{2d}(2d+1)d^{8d}$.
\par
The number of non-admissible long arcs is bounded by $28d^{2}+52d$ (Lemma~\ref{lem:LongArcs}). It follows that in the worst case, each non-admissible arc is located between two chains of $2d-1$ admissible long arcs. Consequently, the number of long arcs is bounded by
$$
(25/2)^{2d}(2d)(2d+1)d^{8d}
+ (28d^{2}+52d)
+ (28d^{2}+52d+1)(2d-1)
.
$$
Since $d \geq 3$, this upper bound can be weakened to $2e^{16d\ln(d)}+92d^{3}$.
\end{proof}

Theorem~\ref{thm:MainBound} then follows from the fact that $|\mathcal{S}|$ and $2|\mathcal{I}_{1}|+|\mathcal{I}_{2}|$ are bounded by the number of long arcs (see Lemma~\ref{lem:LongArcs}). Corollary~\ref{cor:MainLocal} then follows from the combination of Theorem~\ref{thm:MainBound} with Lemma~\ref{lem:EstimateLocal}.

\section{Global geometry of minimal sets}\label{sec:GlobalGeometry}

At present we do not know a general recipe how to describe non-trivial $\minvset{CH}$. Nevertheless we can prove some general statements about their global geometry  and provide some illuminating examples.
\par
Recall that $\minvset{CH}$ is nontrivial if and only if $\deg Q-\deg P \in \{-1,0,1\}$.
\par
In some cases, description of the convex hull $Conv(\minvset{CH})$ is easier to obtain. The following has been proved as Corollary~5.16 in \cite{AHN+24}.

\begin{proposition}
Consider a linear differential operator $T$ given by \eqref{eq:1stN}.Then  the intersection of all \textit{convex} Hutchinson invariant set coincides with the convex hull $Conv(\minvset{CH})$ of the minimal set $\minvset{CH}$.
\end{proposition}

The local analysis of boundary points carried on in the previous sections provides interesting partial results towards a characterization of points where $\partial \minvset{CH}$ is locally convex.

\subsection{Local convexity of the boundary}\label{sub:convex}

Local analysis in terms of correspondences $\Gamma$ and $\Delta$ shows that corner points of $\minvset{CH}$ have to satisfy very specific conditions.

\begin{corollary}\label{cor:CONVEXCORNER}
For a linear differential operator $T$ given by \eqref{eq:1stN}, consider a point $\alpha$ which is a corner point of the boundary $\partial \minvset{CH}$. In other words, there is a neighborhood $V$ of $\alpha$ such that $V \cap \minvset{CH}$ is contained in a cone with apex $\alpha$ and with the opening strictly smaller than $\pi$. Then one of the following statements hold:
\begin{itemize}
    \item $\alpha$ is a simple zero of $R(z)$ satisfying $\phi_{\alpha}=0$ (see \ref{notationlocal});
    \item $\alpha$ is a common root of $P(z)$ and $Q(z)$ of the same multiplicity (i.e. $\alpha$ is neither a zero nor a pole of $R(z)$).
\end{itemize}
Besides, if $\alpha$ is a cusp (neighborhoods of $\alpha$ in $\minvset{CH}$ can be included in cones of arbitrarily small opening angle), then one of the following statements holds:
\begin{itemize}
    \item $\alpha$ is a common root of $P(z)$ and $Q(z)$ of the same multiplicity;
    \item $\minvset{CH}$ is fully irregular and contained in a half-line.
\end{itemize}
\end{corollary}

\begin{proof}
Corollary~\ref{cor:KLalpha} immediately implies that $\alpha$ cannot be a pole or a multiple zero of $R(z)$. Besides, if $\alpha$ is a simple zero, it has to satisfy the condition $\phi_{\alpha}=0$. Now we assume that $\alpha$ is neither a zero nor a pole of $R(z)$. It remains to prove that $\alpha \in \mathcal{Z}(PQ)$.
\par
Assume that $\alpha \notin \mathcal{Z}(PQ)$. In this case if some point of the forward trajectory of $R(z)\partial_{z}$ starting at $\alpha$ belongs to $\minvset{CH}$, then a germ of the integral curve starting at $\alpha$ is contained in $\minvset{CH}$ (see Proposition~\ref{prop:IntegralCurve}) and $\alpha$ cannot be a corner point. We conclude that $\Gamma(\alpha) = \emptyset$. 
\par
If $\Delta(\alpha)$ contains some point $y$, then a branch of the root trail $\mathfrak{tr}_{y}$ containing $\alpha$ belongs to   $\minvset{CH}$ (see Lemmas~\ref{lem:RootTrailSlope} and~\ref{lem:RTSlopeINFINITY}). Thus  in this case $\alpha$ cannot be a corner point and we get $\Delta(\alpha)= \emptyset$.
\par
Now we take a cone $\mathcal{C}$ with apex at $\alpha$,  of angle at least $\pi$ which is locally disjoint from $\minvset{CH}$. If $r(\alpha)$ is contained in $\mathcal{C}$, but is not one of the two limit rays, we can freely remove a neighborhood of $\alpha$ from $\minvset{CH}$ and still get an invariant set. In any other case, we can find an arc contained in a neighborhood of $\alpha$ and the complement of $\minvset{CH}$ whose associated rays sweep out a domain containing $\alpha$. Thus we get a contradiction in this case as well which implies that $\alpha$ has to be in $\mathcal{Z}(PQ)$.
\par
Finally if $\alpha$ is  a simple zero of $R(z)$ and a cusp, then we have $\mathcal{L}_{\alpha}=\mathbb{S}^{1}$ (see Definition~\ref{def:calKL}). It follows that $\minvset{CH}$ has empty interior. All such cases have been completely classified in Section~7 of \cite{AHN+24}. 
\end{proof}

Further local analysis provides necessary conditions under which  boundary points belong to locally convex parts of $\partial \minvset{CH}$.

\begin{proposition}\label{prop:charaCONVEX}
For a linear differential operator $T$ given by \eqref{eq:1stN},  consider a point $\alpha \in \partial \minvset{CH}$ such that there is a neighborhood $V$ of $\alpha$ with the property that $V \cap \minvset{CH}$ is contained in a closed half-plane whose boundary contains $\alpha$.
\par
If $\alpha \in \mathcal{Z}(PQ)$, then one of the following statements holds:
\begin{itemize}
    \item $\alpha$ is a simple pole of $R(z)$;
    \item $\alpha$ is a simple zero of $R(z)$ satisfying $\phi_{\alpha}=0$ (see \ref{notationlocal});
    \item $\alpha$ is a common root of $P(z)$ and $Q(z)$ of the same multiplicity (i.e. $\alpha$ is neither a zero nor a pole of $R(z)$).
\end{itemize}
\par
If $\alpha \in \mathfrak{I}_{R}^{\ast} \setminus \mathcal{Z}(PQ)$, then $\alpha$ is a point of switch type.
\par
If $\alpha \notin \mathfrak{I}_{R} \cup \mathcal{Z}(PQ)$, then one of the following statements holds:
\begin{itemize}
    \item $\alpha$ is a point of local type;
    \item $\alpha$ is a point of global type and for any $u \in \Delta(\alpha)$, either $\Im(f(u,\alpha))=0$ or $\Im(f(u,\alpha))$ and $\Im(R'(\alpha))$ have opposite signs (for $f$ defined as in Proposition~\ref{prop:RTconcave}).
\end{itemize}
\end{proposition}

\begin{proof}
The case $\alpha \in \mathcal{Z}(PQ)$ follows from Corollary~\ref{cor:KLalpha}. If $\alpha \in \mathfrak{I}_{R}^{\ast} \setminus \mathcal{Z}(PQ)$ and $\Delta^{-}(\alpha) \neq \emptyset$, then $\minvset{CH}$ contains both the germ of an integral curve of the field $-R(z)\partial_{z}$ at $\alpha$ and the germ of the root trail $\mathfrak{tr}_{u}$ for some $u \in \Delta^{-}(\alpha)$. Proposition~\ref{prop:RTInflecCONCAVE} implies  that $\minvset{CH}$ cannot be convex at $\alpha$. Besides, if $\Gamma(\alpha) \neq \emptyset$, then $\minvset{CH}$ cannot be convex in $\alpha$ either because a germ of an integral curve having an inflection point at $\alpha$ is contained in $\minvset{CH}$.
In the remaining cases, we have $\Gamma(\alpha) \cup \Delta^{-}(\alpha) =\emptyset$. If $\Delta^{+}(\alpha) \neq \emptyset$, this characterizes points of switch type (see Theorem~\ref{thm:MAINClassification}). If $\Delta^{+}(\alpha) = \emptyset$, then we obtain a point of $C^{2}$-inflection type, $\alpha$ is the starting point of a local arc and $\Gamma(\alpha)$ is therefore nonempty (see Proposition~\ref{prop:INFLECTIONTYPE}).\
\par
Now we consider the case $\alpha \notin \mathfrak{I}_{R} \cup \mathcal{Z}(PQ)$. If $\Gamma(\alpha) \neq \emptyset$, then $\alpha$ is a point of local type ($\alpha$ cannot be a point of extruding type because of Proposition~\ref{prop:extruding}). If $\Gamma(\alpha)=\emptyset$, then it follows from Proposition~\ref{prop:localarc} that $\Delta(\alpha) \neq \emptyset$. Proposition~\ref{prop:RTconcave} then provides the necessary condition.
\end{proof}

\subsection{Case $\deg Q-\deg P=-1$}

We  have a rational vector field $R(z)\partial_z$ satisfying $R(z)=\frac{\lambda}{z}+\frac{\mu}{z^{2}}+o(1/z^{2})$ with $\lambda \in \mathbb{C}^{\ast}$ and $\mu \in \mathbb{C}$.

\subsubsection{Horizontal locus and special line}

We define the following loci.

\begin{definition}
The \textit{horizontal locus} $\mathcal{H}_{R}$ is the closure in $\mathbb{C}$ of the set formed by points $z \notin \mathcal{Z}(PQ)$, for which $\sigma(z) = \frac{\arg(\lambda) \pm \pi}{2}$.
\par
We also denote by $\mathcal{L}_{R}$ the \textit{special line} formed by points $z$ given by the equation $Im(z/\lambda)=Im(\mu/\lambda^{2})$.
\end{definition}

For the sake of simplicity, the vector field $R(z)\partial_{z}$ is normalized by an affine change of variable as $R(z)=-\frac{1}{z}+o(z^{-2})$ ($\lambda=-1$ and $\mu=0$). The line $\mathcal{L}_{R}$ then coincides with the real axis $\mathbb{R}$.
\par

\begin{lemma}\label{lem:qp-1Horizon}
$\mathcal{H}_{R}$ is a real plane algebraic curve of degree at most $\deg P + \deg Q$. It has two asymptotic infinite branches. The line $\mathcal{L}_{R}$ is the asymptotic line for both of them.
\end{lemma}

\begin{proof}
Curve $\mathcal{H}_{R}$ can be seen as the pull-back of the real axis under the mapping $R(z): \mathbb{CP}^{1} \to \mathbb{CP}^{1}$. We have $R(\infty)=0$ and $\infty$ is a simple root of $R(z)$. Therefore, $\mathcal{H}_{R}$ is smooth near $\infty$.
\par
It remains to show that the tangent line to $\mathcal{H}_{R}$ at infinity coincides with the real axis. Actually the tangent line is the line at which the linearization of $R(z)$ at $\infty$ attains real zeroes. Since this linearization is exactly $-\frac{1}{z}$, the result follows.
\end{proof}

\begin{corollary}\label{cor:qp-1Horizon}
The closure $\overline{\minvset{CH}}$ of the minimal set in the extended plane contains asymptotic directions $0$ and $\pi$. Besides, the curve $\mathcal{H}_{R}$ is contained in the minimal set $\minvset{CH}$.
\end{corollary}

\begin{proof}
Looking at separatrices of the vector field $R(z)\partial_{z}$ and using Proposition~\ref{prop:IntegralCurve} we get that the closure $\overline{\minvset{CH}}$ in the extended plane contains asymptotic directions $0$ and $\pi$. The associated rays of points of $\mathcal{H}_{R}$ are thus asymptotically tangent to $\minvset{CH}$ and $\mathcal{H}_{R}$ is contained in the minimal set.
\end{proof}

\begin{proposition}\label{prop:qp-1Convex}
Consider a linear differential operator $T$ given by \eqref{eq:1stN} such that $\deg Q-\deg P=-1$. Then the minimal convex Hutchinson invariant set $Conv(\minvset{CH})$ is a bi-infinite strip (domain bounded by two parallel lines).
\par
More precisely, $Conv(\minvset{CH})$ is the smallest strip containing $\mathcal{H}_{R} \cup \mathcal{Z}(PQ)$.
\end{proposition}

\begin{proof}
The minimal convex Hutchinson invariant set $Conv(\minvset{CH})$ is the complement of the union of every open half-plane disjoint from $\minvset{CH}$. Since $\mathcal{H}_{R}$ is contained in $\minvset{CH}$ (Corollary~\ref{cor:qp-1Horizon}), these open half-planes have to be disjoint from $\mathcal{H}_{R}$. Conversely, any open half-plane $H$  disjoint from $\mathcal{H}_{R}$ is such that $Im(R(z))$ is either positive or negative for every $z \in H$. Therefore, provided $H$ does not contain any zero or pole of $R(z)$, one can conclude that it can be removed  from any Hutchinson invariant set. In other words, $Conv(\minvset{CH})$ is the complement to  the union of all half-planes disjoint from
$\mathcal{H}_{R} \cup \mathcal{Z}(PQ)$. Since $\mathcal{H}_{R}$ has asymptotically horizontal infinite branches, the boundary line of every half-plane disjoint from $\mathcal{H}_{R}$ has to be horizontal.
\par
It remains to prove that such half-planes exist. It follows from the asymptotic description of $\mathcal{H}_{R}$ in Lemma~\ref{lem:qp-1Horizon} that $|Im(z)|$ is bounded on $\mathcal{H}_{R}$. Therefore  we can find two (disjoint) open half-planes that are also disjoint from $\mathcal{H}_{R}$. These half-planes contain half-planes which, in  addition, are disjoint from $\mathcal{Z}(PQ)$.
\end{proof}

\subsubsection{Asymptotic geometry of the minimal set}

Following Proposition~\ref{prop:qp-1Convex}, $Conv(\minvset{CH})$ is the smallest horizontal strip containing the curve $\mathcal{H}_{R} \cup \mathcal{Z}(PQ)$. The closure of the projection of $Conv(\minvset{CH})$ on the vertical axis is an interval $[y^{-},y^{+}]$ where $y^{-} \leq 0 \leq y^{+}$.

\begin{lemma}\label{lem:qp-1HoriSlope}
For $0<y<y_{0}$,  denote by $M_{t}$ the intersection point between the associated ray $r(t+iy)$ and the horizontal line  $Im(z)=y_{0}$. Then the following statements hold:
\begin{itemize}
    \item for $t \longrightarrow + \infty$, $\Re(M_{t}) \longrightarrow - \infty$ if $y<\frac{y_{0}}{2}$;
    \item for $t \longrightarrow + \infty$, $\Re(M_{t}) \longrightarrow + \infty$ if $\frac{y_{0}}{2}<y<y_{0}$.  
\end{itemize}
Analogous statements hold for $t \longrightarrow -\infty$ or $y_{0}<y<0$.
\end{lemma}

\begin{proof}
For large values of $t$, we have $\Re(R(z))=-\frac{1}{t} + o(t^{-1})$ and $\Im(R(z))=\frac{y}{t^{2}}+o(t^{-2})$. Provided $t$ is large enough, $\Im(R(z))$ is positive and the associated ray $r(z)$ intersects the line  $Im(z)=y_{0}$. Then the real part of the intersection point equals $t-(y_{0}-y)\frac{t}{y}+o(t)$. After simplification, we obtain $\frac{(2y-y_{0})t}{y}+o(t)$. The sign of the main term is then determined by the sign of $2y-y_{0}$.
\end{proof}

\begin{proposition}\label{prop:NOFINGER}
If $\deg Q- \deg P=-1$, the minimal set $\minvset{CH}$ is connected in $\mathbb{C}$.
\end{proposition}

\begin{proof}
Following Proposition~\ref{prop:qp-1}, the complement $(\minvset{CH})^{c}$ of $\minvset{CH}$ in $\mathbb{C}$ has exactly two connected components and it has been proved in Proposition~\ref{prop:qp-1Convex} that each of them contains a half-plane. We refer to the domain containing an upper half-plane as $\mathcal{D}^{+}$ and to the domain containing a lower half-plane as $\mathcal{D}^{-}$. Since $\minvset{CH}$ contains $\mathcal{H}_R$, we deduce that $\Im(R(z))$ is positive on $\mathcal{D}^{+}$ and negative on $\mathcal{D}^{-}$.
\par
Proving that $\minvset{CH}$ is connected in $\mathbb{C}$ amounts to showing that $\mathcal{D}^{+}$ and $\mathcal{D}^{-}$ have only one topological end. We will prove this statement for $\mathcal{D}^{+}$ (the proof for $\mathcal{D}^{-}$ is identical). We assume by contradiction that $\mathcal{D}^{+}$ has a topological end $\kappa$ distinct from the end of the upper half-plane contained in $\mathcal{D}^{+}$ (we will refer to this end as the \textit{main end} of $\mathcal{D}^{+}$).
\par
For any sequence $\{z_{n}\}$ of points  in $(\minvset{CH})^{c}$ approaching $\kappa$, we have (up to taking a subsequence) the sequence $\{\arg(z_{n})\}$ converging either to $0$ or to $\pi$ (since otherwise, $\kappa$ would not be distinct from the main end). Let's assume without loss of generality that it is $0$. Again, we can assume that $\{Im(z_{n})\}$ converges to some value $y_{e} \in [0,y^{+}]$.
\par
If $y_{e}>0$, then Lemma~\ref{lem:qp-1HoriSlope}, shows that for any horizontal line $L_{f}$ with $y_{f} \in ]y_{e},2y_{e}[$, the associated rays of the points in $\minvset{CH}^{c}$ converging to the end $\kappa$ sweep out points of $L_{f}$ whose real part is arbitrarily close to $+\infty$. Assuming that $y_{e}$ is the maximal possible limit value, we deduce that no infinite component of $\minvset{CH}$ can separate $\kappa$ from the upper main end containing asymptotic directions of $]0,\pi[$.
\par
Hence, for a sequence $\{z_{n}\}$ of points in $(\minvset{CH})^{c}$ approaching $\kappa$, the only accumulation value of $\{Im(z_{n})\}$ is $0$. In this case, the associated rays $r(z_{n})$ accumulate onto the $\mathbb{R}$-axis which is therefore contained in the closure of $\mathcal{D}^{+}$.
\par
Now we prove that the open upper half-plane defined by $\Im(z)>0$ is disjoint from $\minvset{CH}$. We assume by contradiction the existence of a point $z_{0}$ such that $y_{0}=\Im(z_{0})$ is positive and $z_{0} \in \minvset{CH}$. We denote by $L_{y_{0}}$ the horizontal line formed by points satisfying $\Im(z)=y_{0}$. Since there is a family of associated rays accumulating onto the $\mathbb{R}$-axis, there exists a path $(t+if(t))_{t \in \mathbb{R}}$ such that for any $t$, $f(t) \in ]0,\frac{y_{0}}{4}[$ and $t+if(t) \in \mathcal{D}^{+}$. Applying Lemma~\ref{lem:qp-1HoriSlope} to the intersection between $L_{y_{0}}$ and the family of associated rays starting from $t+if(t)$, a continuity argument proves that $z_{0}$ belongs to some associated ray of the family (as $t$ moves from $-\infty$ to $+ \infty$, the intersection of the associated rays with $L_{y_{0}}$ moves from the right end to the left end of this horizontal line). Therefore, $z_{0}$ cannot belong to $\minvset{CH}$ and the open upper half-plane defined by $\Im(z)>0$ is disjoint from $\minvset{CH}$.
\par
Then, there are interior points of connected a component $X$ of $\minvset{CH}$ located above $\kappa$ whose imaginary value is negative. It follows that the associated rays of points of $\mathcal{D}^{+}$ approaching $\kappa$ intersect the interior of $X$ (these associated rays accumulate on the $\mathbb{R}$-axis). Therefore, there is no such end $\kappa$ and $\minvset{CH}$ is connected.
\end{proof}

\begin{proposition}\label{prop:qp-1THICK}
There is a compact set $K$ and a positive constant $B>0$ such that the intersection $\minvset{CH} \cap K^{c}$ is contained in the closure of the domain bounded by the hyperbolas given by \begin{align}
    y=\frac{y^{+}}{2}\left(1+\frac{B}{ x}\right),\quad
    y=\frac{y^{-}}{2}\left(1+\frac{B}{ x}\right):\quad  x>0 \\
    y=\frac{y^{-}}{2}\left(1-\frac{B}{x}\right),\quad 
    y=\frac{y^{+}}{2}\left(1-\frac{B}{x}\right):\quad  x<0. 
\end{align}
\end{proposition}

\begin{proof}
By Lemma~\ref{lem:qp-1Horizon}  for any $y$ in $J=[y^{-},\frac{y^{-}}{2}[ \cup ]\frac{y^{+}}{2},y^{+}]$, there is a positive constant $A>0$ such that the union of the two semi-infinite horizontal strips characterized by $\Im(z) \in J$ and $|\Re(z)|>A$ is disjoint from $\mathcal{H}_{R}$.
\par
Consider some positive number $B>A$ and introduce the domain $D_{B}$ characterized by the  inequalities:
\begin{itemize}
    \item $\Im(z) > y^{+}$ if $\Re(z) \in [-B,B]$;
    \item $\Im(z) > g(t)$ where $g(t)=\frac{y^{+}}{2}\frac{|t|+B}{|t|}$ if $t = \Re(z) \notin [-B,B]$.
\end{itemize}
For any point $z$ such that $\Im(z) > y^{+}$, the associated ray $r(z)$ remains in $D_{B}$. Now we assume that $z=t+iy$ satisfies the conditions 
$$|t|>B\quad \text{and} \quad \frac{y^{+}}{2}\frac{|t|+B}{|t|}<|y| \leq y^{+}.$$ Without loss of generality, we  assume that $t<-B$. 
\par
In order to prove that the associated ray $r(z)$ remains in $D_{B}$, we  have to show that  for any $t<-B$ and any $s \in [t,-B]$, we have
$$\frac{\Im(R(z))}{\Re(R(z))} > \frac{g(s)-g(t)}{s-t}.$$
\par
Since $\frac{g(s)-g(t)}{s-t} \leq \frac{By^{+}}{2st} \leq -\frac{y^{+}}{2t}$, we just have to prove that
$$\frac{\Im(R(z))}{\Re(R(z))} > -\frac{y^{+}}{2t}.$$
\par
In our case $\Re(R(z))=-\frac{1}{t}+o(t^{-2})$ and $\Im(R(z))=\frac{y}{t^{2}}+o(t^{-3})$ imply that 
$$\frac{\Im(R(z))}{\Re(R(z))}=-\frac{y}{t}+o(t^{-2}).$$
Since $y-\frac{y^{+}}{2}>\frac{y^+B}{2|t|}>0$, the inequality holds provided $B$ is large enough.
\par
By replacing $y^{+}$ by $y^{-}$, we get an analogous result for the lower part of the complement to $\minvset{CH}$.
\end{proof}

\subsubsection{Examples}

Consider a family of operators of the form $T_{\alpha}=Q(z)\frac{d}{dz}+P(z)$ where $Q(z)=(z-\alpha)^{k}$ and $P(z)=z(\alpha-z)^{k}$ with the  common root $\alpha \in \mathbb{C}$ of degree $k \in \mathbb{N}^{\ast}$.
\par
The family $T_{\alpha}$ provides a rich assortment of examples. We have $R(z)=-\frac{1}{z}$. The special line is the real axis $\mathbb{R}$ which coincides with the horizontal locus $\mathcal{H}_{R}$. Besides, the integral curves of $R(z)\partial_{z}$ are hyperbolas (level sets of  $xy$).

\begin{proposition}
If $\alpha \in \mathbb{R}$, then the minimal set $\minvset{CH}$ of operator $T_{\alpha}$ coincides with the real axis $\mathbb{R}$.
\end{proposition}

\begin{proof}
This follows immediately from Proposition~\ref{prop:qp-1Convex} and Proposition~\ref{prop:NOFINGER}.
\end{proof}

If $\alpha$ does not belong to the real axis, we get different pictures depending on whether or not $\alpha$ belongs to the imaginary axis. Without loss of generality, we will assume that $Im(\alpha)>0$.

\begin{proposition}
If $\alpha$ is of the form $y_{0}i$ with $y_{0}>0$, then the minimal set $\minvset{CH}$  is the union of the segment $[\frac{y_{0}}{2}i,y_{0}i]$ with the horizontal strip formed by points $z$ satisfying $0 \leq Im(z) \leq \frac{y_{0}}{2}$.
\end{proposition}

\begin{proof}
From Proposition~\ref{prop:qp-1Convex} it follows immediately  that the convex hull of $\minvset{CH}$ is contained in the strip bounded by $\mathbb{R}$ and the horizontal line  $Im(z)=Im(y_{0})$. For any point of segment $[0,y_{0}i]$, the associated ray contains $\alpha$ so $[0,y_{0}i] \subset \minvset{CH}$.
\par
For any point  of the horizontal strip given by the inequalities $0 \leq Im(z) \leq \frac{y_{0}}{2}$, a simple  computation proves that its associated ray intersects the segment $[0,y_{0}i]$.
\par
Finally, for any point $z$ such that $Im(z)>\frac{y_{0}}{2}$ and $Re(z) \neq 0$, the associated ray is disjoint from the segment $[0,y_{0}i]$. This completely characterizes the minimal set.
\end{proof}

The latter case provides an example of a partially irregular minimal set whose irregularity locus is contained in a $R$-invariant line (the imaginary axis in this case).
\par
In the general case, the boundary of $\minvset{CH}$ is more complicated. Up to conjugation, we can restrict us to the case when $\Re(\alpha),\Im(\alpha)>0$.

\begin{proposition}
If $\alpha$ is of the form $x_{0}+y_{0}i$ with $x_{0},y_{0}>0$, then the minimal set $\minvset{CH}$ of  $T_{\alpha}$ is bounded by the following arcs:
\begin{itemize}
    \item the real $\mathbb{R}$-axis ;
    \item global arc $(t,f_{1}(t))$ where $f_{1}(t)=\frac{y_{0}t}{2t-x_{0}}$ for $t \in [x_{0},+\infty[$;
    \item local arc $(t,f_{2}(t))$ where $f_{2}(t)=\frac{x_{0}y_{0}}{t}$ for $t \in [x_{0},x_{e}]$;
    \item global arc $(t,f_{3}(t)$ where $f_{3}(t)=\frac{x_{0}y_{0}t}{(2\sqrt{x_{0}t}+x_{0})^{2}}$ for $t \in [0,x_{e}]$;
    \item global arc $(t,f_{4}(t))$ where $f_{4}(t)=\frac{y_{0}t}{2t-x_{0}}$ for $t \in ]-\infty,0]$.
\end{itemize}
Here, $(x_{e},y_{e})$ is a point of extruding type. Its coordinates are $x_{e}=(3+2\sqrt{2})x_{0}$ and $y_{e}=\frac{y_{0}}{3+2\sqrt{2}}$.
\end{proposition}
\begin{proof}
The convex hull of $\minvset{CH}$ is contained in the strip bounded by $\mathbb{R}$ and the horizontal line $Im(z)=Im(y_0)$, see Proposition~\ref{prop:qp-1Convex}. The arcs $(t,f_{1}(t))$ and $(t,f_{4}(t))$ are characterized by the fact that the associated rays starting from their points contain $x_{0}+iy_{0}$ (this can be checked by a direct computation). In particular, they belong to two distinct branches of the same hyperbola. Besides, the domain $\mathcal{D}$ between $\mathbb{R}^{-}$ and arc $(t,f_{4}(t))$ is automatically contained in $\minvset{CH}$.
\par
Following Proposition~\ref{prop:IntegralCurve}, the backward trajectory  of the vector field $R(z)\partial_{z}$ starting at $x_{0}+y_{0}i$ is contained in $\minvset{CH}$. The domain between this portion of the integral curve and the arc $(t,f_{1}(t))$ is also contained in $\minvset{CH}$.
\par
We denote by $\mathcal{D}'$ the domain in the open right upper quadrant where the associated ray intersects the domain $\mathcal{D}$. At each point $(t,\gamma(t))$ of the upper boundary of $\mathcal{D}'$, the associated ray is tangent to the branch of hyperbola $(s,f_{4}(s))$ for some $s \leq 0$. Since $R(z)=-\frac{1}{z}$, the argument of $t+i\gamma(t)$ equals the negative of the slope of $(s,f_{4}(s))$ at $s$. Since  $\frac{df_{4}}{ds}(s)=-\frac{x_{0}y_{0}}{(2s-x_{0})^{2}}$, we get 
$$
\frac{\gamma(t)}{t}
=
\frac{x_{0}y_{0}}{(2s-x_{0})^{2}}
.
$$
Since the tangent line has to intersect the imaginary axis at  $2\gamma(t)i$, we obtain the following equation:
$$
\frac{f_{4}(s)-2\gamma(t)}{s}
=
-\frac{x_{0}y_{0}}{(2s-x_{0})^{2}}
.
$$
Replacing $\gamma(t)$ by $\frac{x_{0}y_{0}t}{(2s-z_{0})^{2}}$, we get $t= \frac{s^{2}}{x_{0}}$. 
\par
Since $s$ is the negative square root of $x_{0}t$, we deduce that $\gamma(t)=\frac{x_{0}y_{0}t}{(2\sqrt{x_{0}t}+x_{0})^{2}}$. In particular, for $s=-x_{0}$, we get $t=x_{0}$ and $\gamma(x_{0})=\frac{y_{0}}{9}$.
\par
The arc $\gamma$ and the backward trajectory starting at $x_{0}+iy_{0}$ (which is a branch of hyperbola) intersect each other at some point $x_{e}+iy_{e}$. From a computation, we obtain $x_{e}=(3+2\sqrt{2})x_{0}$ and therefore $y_{e}=\frac{y_{0}}{3+2\sqrt{2}}$.
\par
It is then geometrically clear that for any point $z$ above the curve formed by arcs defined by functions $f_{1},f_{2},f_{3},f_{4}$, the associated ray cannot intersect any of these arcs.
\end{proof}

The latter example provides an illustration of a point of extruding type. Since the boundary arcs are explicit algebraic curves, we can obtain the exact picture shown in Figure~\ref{fig:minimalAlpha}.

\begin{figure}[!ht]
    \centering
    \includegraphics[width=0.9\textwidth]{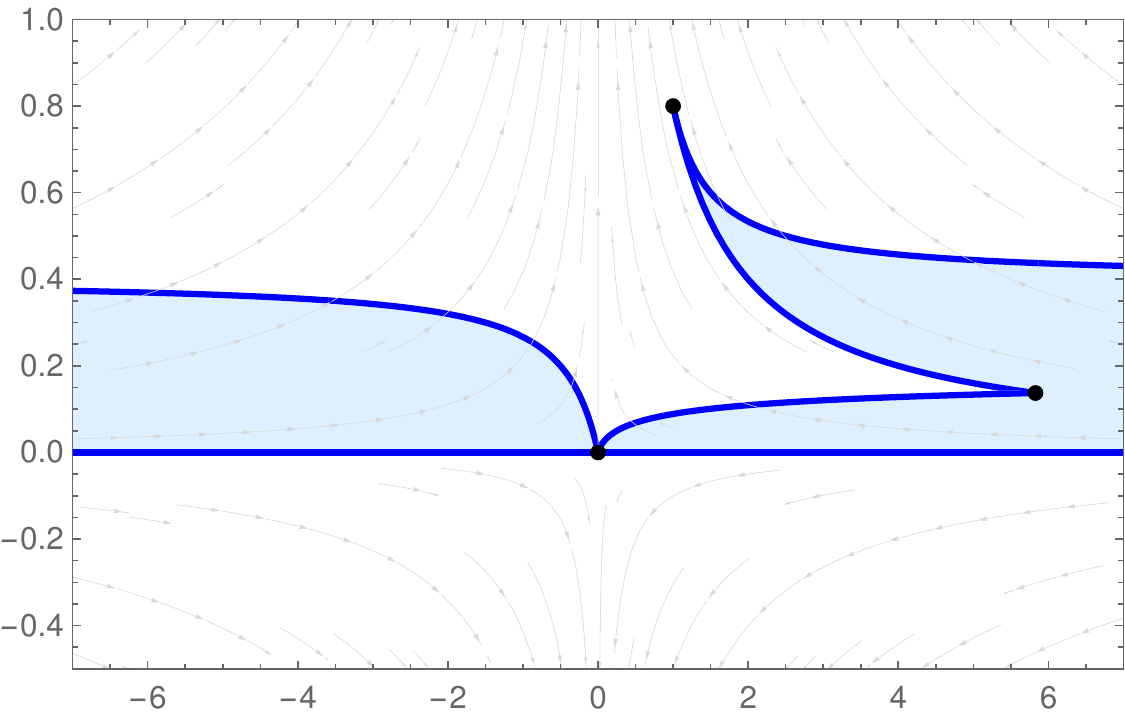}
    \caption{The case when $\alpha=1+0.8i$.}
    \label{fig:minimalAlpha}
\end{figure}

\subsection{Case $\deg Q-\deg P=0$}
For the sake of simplicity, we normalize the vector field $R(z)\partial_{z}$  by an affine change of variable so that  $R(z)=1+\frac{\mu}{z^{\kappa}}+o(z^{-\kappa-1})$ for some $\mu \in \mathbb{C^{\ast}}$ and $\kappa \geq 1$. The case of a constant vector field is already treated in Section~2.3 of \cite{AHN+24}.

Under the assumptions $\Im(\mu) \neq 0$ and $\kappa = 1$, we are going to prove that the minimal set $\minvset{CH}$ is  connected. Firstly we show that $\minvset{CH}$ is regular and disjoint from the curve of inflections $\mathfrak{I}_{R}$ outside a compact set.

\begin{lemma}\label{lem:qp0CONE}
Assuming that $\Im(\mu)(-1)^{\kappa} > 0$, there is a cone $\mathcal{C}$ and a compact set $K$ such that:
\begin{itemize}
    \item for any $z \in \mathcal{C}$, $\Im(R(z))>0$ and $\Im(R'(z))>0$;
    \item $\minvset{CH} \subset \mathcal{C} \cup K$.
\end{itemize}
Besides, $\minvset{CH}$ is a regular subset of $\mathbb{C}$.
\end{lemma}

\begin{proof}
Computing $R(t)$ and $R'(t)$ for a negative real number $t$, we obtain that $R(t)=1+\frac{\mu}{t^{\kappa}}+o(t^{-\kappa-1})$ and thus $\Im(R(t)) \sim \Im(\mu) t^{-\kappa}$. The sign of the former is thus the sign of $\Im(\mu)(-1)^{\kappa}$. Similarly we obtain that it is also the sign of $\Im(R'(t))$ for $t$ close enough to infinity and negative.
\par
It follows from Proposition~\ref{prop:qp0} that $\minvset{CH}$ is contained in an infinite cone $\mathcal{C}_{0}$ whose asymptotic directions are $]\pi-\epsilon,\pi+\epsilon[$ for some $\epsilon \in ]0,\frac{\pi}{2}[$. Besides, since the asymptotic directions of infinite branches of the algebraic curves defined by equations $\Im(R)=0$ and $\Im(R')=0$ are not horizontal, $\mathcal{C}_{0}$ and thus $\minvset{CH}$ are covered by the union of a cone $\mathcal{C}$ and a compact set $K$ such that for any $z \in \mathcal{C}$, $\Im(R(z))>0$ and $\Im(R'(z))>0$.
\par
Any $R$-invariant line (see Definition~\ref{def:Rinvariant}) has to be horizontal and therefore it intersects the cone $\mathcal{C}$. Thus some points of any $R$-invariant line $\Lambda$ have the associated rays that are not contained in $\Lambda$. Therefore, there are no $R$-invariant lines for such a vector field $R(z)\partial_{z}$. The minimal set $\minvset{CH}$ has hence no tails and Theorem~\ref{thm:main} guarantees that $\minvset{CH}$ is regular.
\end{proof}

\begin{corollary}\label{cor:qp0ASYMPTOLINE}
Assuming that $\Im(\mu)(-1)^{\kappa} > 0$, consider a sequence $(\alpha_{n})_{n \in \mathbb{N}}$ of points of $\partial \minvset{CH}$ such that $|\alpha_{n}| \rightarrow + \infty$ and $\Delta(\alpha_{n}) \neq \emptyset$ for any $n \in \mathbb{N}$.
Then there exist a subsequence $(\alpha_{f(n)})_{n \in \mathbb{N}}$ and a line $\mathcal{L}_{y_{0}}$ given by $\Im(z)=y_{0}$
such that:
\begin{itemize}
    \item the line $\mathcal{L}_{y_{0}}$ is disjoint from the interior of $\minvset{CH}$;
    \item the line $\mathcal{L}_{y_{0}}$ contains a point of $\partial \minvset{CH}$; 
    \item $\Re (\alpha_{f(n)}) \rightarrow - \infty$;
    \item $\Im (\alpha_{f(n)}) \leq y_{0}$ for any $n \in \mathbb{N}$.
\end{itemize}
\end{corollary}

\begin{proof}
Up to taking a subsequence, we can also assume that every $a_{n}$ belongs to the cone $\mathcal{C}$ defined in Lemma~\ref{lem:qp0CONE}. Lemma~\ref{lem:EchangeInflectionDelta} implies that for any $n$, points of $\Delta(\alpha_{f(n)})$ belong to $\mathfrak{I}^{-}$, $\mathfrak{I}_{R}$ or $\mathcal{Z}(PQ)$. Therefore, following Lemma~\ref{lem:qp0CONE},  points of $\Delta(\alpha_{n})$ accumulate in a compact set as $n\to \infty$. We denote by $z_{0}$ one of their accumulation points and by $L_{y_{0}}$ the horizontal line containing $z_{0}$ (here $y_{0} = \Im (z_{0})$).
\par
Thus, up to taking a subsequence of $\alpha$, we get a sequence $(y_{n})_{n \in \mathbb{N}}$ such that $y_{n} \rightarrow y_{0}$ and $y_{n} \in \Delta(\alpha_{n})$ for any $n \in \mathbb{N}$. As $\alpha_{n}$ goes to infinity while $\Delta(\alpha_{n})$ remains in a compact set, the associated rays $r(\alpha_{n})$ accumulate on $L_{y_{0}}$. Thus the line $L_{y_{0}}$ is disjoint from the interior of $\minvset{CH}$. Besides, since $\alpha_{n} \in \mathcal{C}$ for any $n \in \mathbb{N}$, we have $Im(R(\alpha_{n})) >0$ and therefore $\Im(\alpha_{n}) \leq y_{0}$.
\end{proof}

\begin{lemma}\label{lem:qp0kappa1noline}
Provided that $\Im(\mu) \neq 0$ and $\kappa=1$, no integral curve has a horizontal asymptotic line at infinity.
\end{lemma}

\begin{proof}
Since $R(z)=1+\frac{\mu}{z}+o(z^{-2})$, the integral curve $\gamma(t)$ satisfies $\Re(\gamma(t)) \sim t$ as $t \rightarrow \pm \infty$. Then $Im(\gamma'(t))=\frac{Im(\mu)}{t}+o(t^{-1})$. We obtain that $Im(\gamma(t))$ has logarithmic growth as $t \rightarrow \pm \infty$ and therefore the integral curve has  no asymptotic lines at infinity.
\end{proof}

\begin{corollary}\label{cor:qp0ONEFINGER}
Provided that $\Im(\mu) \neq 0$ and $\kappa=1$, the minimal set $\minvset{CH}$ is connected in $\mathbb{C}$. Besides, $\partial \minvset{CH}$ has exactly two infinite arcs: one is a local arc starting at infinity while the other is a global arc  ending at infinity.
\end{corollary}

\begin{proof}
Without loss of generality, we can assume that $\Im(\mu)>0$. Proposition~\ref{prop:CCrawBOUND} shows that there are finitely many connected components of $\minvset{CH}$. Moreover they are attached to the point  $\pi \in \mathbb{S}^{1}$ at infinity in some linear order. We refer to these components of $\minvset{CH}$ as $X_{1},\dots,X_{k}$  where $X_{1}$ is the lowest component while $X_{k}$ is the highest component. Besides the boundary $\partial X_{i}$ of any component $X_{i}$ has exactly two topological ends. We call them the lower end $\partial X_{i}^{-}$ and the upper end $\partial X_{i}^{+}$.
\par
Since $\partial \minvset{CH} \cap \mathfrak{I}_{R}$ is contained in a compact set (see Lemma~\ref{lem:qp0CONE}), points of $\partial X$ that are close enough to infinity are either of local, global or of extruding types. Proposition~\ref{prop:localEND} proves that every local arc has an endpoint in $\mathfrak{I}_{R} \cup \mathcal{Z}(PQ)$. Thus the ends of $\partial X$ are represented either by a local arc starting at infinity or by a global arc. Since $\Im(\mu)>0$, these points belong to $\mathfrak{I}^{-}$ so the orientation constraint shows that only the upper end can be represented by a local arc. Otherwise, the local arc would have the point at infinity as its endpoint. Equivalently, any lower end $\partial X_{i}^{-}$ has to be represented by an infinite global arc ending at infinity (see Lemma~\ref{lem:globalorientation}).
\par
For any component $X_{i}$, the lower end $\partial X_{i}^{-}$ of its boundary is approached by a sequence of points of global type. Applying Corollary~\ref{cor:qp0ASYMPTOLINE} to such a sequence we prove the existence of a horizontal line $L_{i}$ lying below the component $X_{i}$ and disjoint from the interior of $\minvset{CH}$. Thus no component of $\minvset{CH}$ lying below the line $L_{i}$ can contain an infinite local arc because the latter has no asymptotic line at infinity (see Lemma~\ref{lem:qp0kappa1noline}). Consequently, among the ends of $\partial \minvset{CH}$, only $\partial X_{k}^{+}$ can be represented by a local arc.
\par
It remains to prove that $\minvset{CH}$ has only one connected component. Assuming that $k>1$, we consider the upper end $\partial X_{1}^{+}$. We already know that it can be approached by points $(\alpha_{n})_{n \in \mathbb{N}}$ for which $\Delta(a_{n}) \neq \emptyset$. Applying Corollary~\ref{cor:qp0ASYMPTOLINE}, we prove the existence of a line $L$ such that:
\begin{itemize}
    \item $L$ is disjoint from the interior of $\minvset{CH}$;
    \item there is some point $z_{0} \in L \cap \partial \minvset{CH}$;
    \item points of $(\alpha_{n})_{n \in \mathbb{N}}$ lie below the line $L$.
\end{itemize}
We deduce from the first and the third bullet points that the interior of component $X_{1}$ lies below line $L$.
\par
Besides, since for each $n \in \mathbb{N}$, $\alpha_{n}$ belongs to $\mathfrak{I}^{-}$, $\Delta(\alpha_{n})$ is a direct support point of $\minvset{CH}$ for the associated ray $r(\alpha_{n})$ (see Lemma~\ref{lem:directsupport}). Then, $z_{0}$ is also a direct support point of $\minvset{CH}$ for (oriented) line $L$. It follows that the interior of the component of $\minvset{CH}$ containing $z_{0}$ lies above line $L$. In other words, $z_{0}$ belongs to some component $X_{i}$ such that $i>1$.
\par
For any point $z$ of $\partial X_{i}$ close enough to $\partial X_{i}^{-}$, the associated ray $r(z)$ has to cross the portion of the line $L$ formed by points whose real part is smaller than $\Re(z_{0})$ (otherwise, the associated ray would have to cross the interior of $X_{i}$). This is impossible since $z$ lies on or above $L$ and $Im(R(z))>0$. This is a contradiction. There is no such component $X_{i}$ and $\minvset{CH}$ is connected. A neighborhood of its upper end is contained in a local arc while a neighborhood of its lower end is contained in a global arc.
\end{proof}

%\begin{comment}
%    \subsection{Number of connected components of $\minvset{CH}$}
%We have the following statement concerning the number of connected components of $\minvset{CH}.$
%\begin{proposition}
%    The number of connected components of $\minvset{CH}$ is either 1 or 2, and both situations occur.
%\end{proposition}
%\begin{proof}
%    If $\deg Q-\deg P=\pm 1$ then $\minvset{CH}$ is connected. In \cite{AHN+} we saw examples of $\minvset{CH}$ for $\deg Q-\deg P$ with two connected components. It remains to prove that these are the only possibilities. Suppose there is $N\geq 3$ connected components of $\minvset{CH}$
%    We may assume that $\deg Q-\deg P=0$ and that 
%    $$R(z)=1+\frac{\mu}{z^m}+o(z^m)$$
%    at infinity, where $\mu$ is non-zero. Take a vertical line $L$ defined by $\Re(z)=C$ such that $L$ intersects all components of $\minvset{CH}$ and such that $\bC\setminus (\minvset{CH}\cup L)$ consists of $N+1$ connected components. There is some $C_0$ such that this holds for all $C<C_0$, which can be realized by the expansion of $R(z)$ close to $\infty$. At least two of these components, say $X_1$ and $X_2$ are such that $\overline{X_i}\cap L\cap \minvset{CH}$ contains two points. By picking $C$ of sufficiently large absolute value (and negative), we may assume that one of the $X_i$ do not intersect the real line, say $X_1$. 
%\end{proof}
%\end{comment}

\subsection{Connected components of minimal sets}\label{sub:CC}

Putting together partial results for the different values of $\deg Q - \deg P$, we are able to state a bound on the number of connected components of $\minvset{CH}$ in $\mathbb{C}$. It is already known that the closure of $\minvset{CH}$ in the extended plane $\mathbb{C} \cup \mathbb{S}^{1}$ is always connected.

\begin{proof}[Proof of Theorem~\ref{thm:CC}]
For any operator $T$ satisfying $|\deg Q - \deg P|>1$, it has been proved in Theorem~1.11 of \cite{AHN+24} that $\minvset{CH}=\mathbb{C}$. Besides, when $\deg Q - \deg P =1$, Section~6.3 and Corollary~5.20 of the same paper proves that $\minvset{CH}$ is connected and contractible. For $\deg Q - \deg P =-1$, it follows from Proposition~\ref{prop:NOFINGER}.
\par
The only case where there could be several connected components is  $\deg Q - \deg P =0$. If $R(z)$ is constant, then there are two situations. If $P,Q$ are both constant, then there is no meaningful notion of minimal set (see Section~2.3.1 in \cite{AHN+24}). Otherwise, $\minvset{CH}$ is formed by parallel half-lines starting at points of $\mathcal{Z}(PQ)$. Since every point of $\mathcal{Z}(PQ)$ is a common root of $P$ and $Q$ (otherwise $R(z)$ would not be constant) we get that there are at most $\frac{1}{2}\deg P + \frac{1}{2}\deg Q$ such half-lines.
\par
If $R(z)$ is not constant, then we have $R(z)=\lambda+\frac{\mu}{z} + o(z^{-1})$ for some $\lambda \in \mathbb{C}^{\ast}$ and $\mu \in \mathbb{C}$. If $\Im(\mu/\lambda) \neq 0$, then Corollary~\ref{cor:qp0ONEFINGER} proves that $\minvset{CH}$ is connected. Otherwise, Proposition~\ref{prop:CCrawBOUND} provides an upper bound $\frac{1}{2}\deg P + \frac{1}{2}\deg Q$.
\end{proof}

\subsection{Case $\deg Q-\deg P=1$}

In \cite {AHN+24} we found that, outside a rather trivial case\footnote{When $\deg P =0$ and $\deg Q =1$, $\minvset{CH}$ coincides with the unique root of $Q(z)$ when $\lambda \notin \mathbb{R}_{<0}$ and coincides with $\mathbb{C}$ otherwise.}, a necessary and sufficient condition for the compactness of $\minvset{CH}$ in case $\deg Q -\deg P=1$ is $\Re (\lambda) \geq 0$. Moreover in case $\Re (\lambda) <0 $, we get $\minvset{CH}=\mathbb{C}$. 
\par
We will describe $\minvset{CH}$ for $\Re (\lambda) = 0$. Unfortunately, in the most interesting situation $\Re (\lambda) > 0$, we do not have a general description of $\minvset{CH}$, but  we provide a number of partial results, observations and examples. 

\subsubsection{$\Re(\lambda)=0$}

In this case a complete characterization of $\partial\minvset{CH}$ can be carried out.

\begin{theorem}\label{thm:periodic}
Consider a linear differential operator $T$ given by \eqref{eq:1stN} such that $\deg Q - \deg P =1$ and $\Re(\lambda)=0$. In this case, the neighborhood of infinity is foliated by a family $\mathcal{C}$ of closed integral curves of the vector field $R(z)\partial_{z}$.
\par
The boundary $\partial \minvset{CH}$ of the minimal set of  $T$ is described as the first closed leaf (according to the natural ordering starting at infinity) of the family $\mathcal{C}$ containing a point of $\mathcal{Z}(PQ) \cup \mathfrak{I}_{R}$.
\par
If the latter leaf $\gamma$ contains a point of the curve of inflections $\mathfrak{I}_{R}$, then the latter point is a tangency point  between $\gamma$ and $\mathfrak{I}_{R}$. Moreover it is the first leaf that is non-strictly convex (the curvature at the tangency point vanishes).
\par
In particular, $\partial \minvset{CH}$ is formed by finitely many local arcs. It is real-analytic and convex (but can fail to be strictly convex). It contains neither zeros nor poles of $R(z)\partial_{z}$.
\end{theorem}

\begin{proof}
It follows from $\lambda \in \mathbb{C}^{\ast}$ and $\Re(\lambda) = 0$ that $\Im(\lambda) \neq 0$. The curve of inflections $\mathfrak{I}_{R}$ is therefore compact. The neighborhood of infinity is foliated by a family $\mathcal{C}$ of integral curves of vector field $R(z)\partial_{z}$. The orientation of these integral curves depends on the sign of $\Im(\lambda)$. By compactness of $\mathfrak{I}_{R}$, another neighborhood $\mathcal{C}'$ of infinity is foliated by strictly convex integral curves (the curvature of integral curves vanishes precisely on $\mathfrak{I}_{R}$).
\par
We first consider the case when some point $\alpha$ of $\mathcal{Z}(PQ)$ belongs to $\mathcal{C}'$. Denoting by $\gamma$ the periodic leaf $\alpha$ belongs to, we deduce from Proposition~\ref{prop:IntegralCurve} that $\gamma$ belongs to $\minvset{CH}$ and bounds a strictly convex domain $\mathcal{D}$. Provided the complement of $\mathcal{D}$ does not contain any other point of $\mathcal{Z}(PQ)$, we obtain that $\mathcal{D}$ coincides with $\minvset{CH}$. Since $\alpha$ is disjoint from $\mathfrak{I}_{R}$, it follows from Corollary~\ref{cor:SingInflec} that it cannot be a zero or a pole of $R(z)$ ($\alpha$ is a root of both $P$ and $Q$ of the same multiplicity).
\par
In the remaining cases, we can assume that $\mathcal{Z}(PQ)$ is disjoint from $\mathcal{C}'$. The cylinder $\mathcal{C}$ is bounded by a singular curve formed by separatrices (integral curves connecting singularities of $R(z)\partial_{z}$). We denote by $\Sigma$ the union of these separatrices and by $\mathcal{S}$ the smallest simply connected subset containing $\Sigma$. By Proposition~\ref{prop:IntegralCurve},  $\Sigma$ and $\mathcal{S}$ are contained in $\minvset{CH}$. For the same reason, a point $z$ of cylinder $\mathcal{C}$ is contained in $\minvset{CH}$ if and only if the periodic integral curve containing $z$ belongs entirely to $\minvset{CH}$. Therefore, the boundary of $\minvset{CH}$ coincides with some periodic integral curve of the cylinder $\mathcal{C}$.
\par
Since the associated rays cannot cross the interior of  $\minvset{CH}$, its boundary  $\partial\minvset{CH}$ (which is a periodic integral curve) has to be  convex. Therefore, it is contained in the domain of inflection of infinity (or in its boundary). Since the domain $\mathcal{C}'$ does not belong to the interior of $\minvset{CH}$ (its complement is clearly a $T_{CH}$-invariant set), these conditions characterize the boundary $\gamma$ of $\mathcal{C}'$ as the boundary of $\minvset{CH}$.
\par
The curve $\gamma$ cannot cross the curve of inflections because it is convex. If it did not intersect $\mathcal I_R$ there would be a strictly smaller invariant set whose boundary is an integral curve between $\mathcal I_R$ and $\gamma$. Thus $\gamma$ has a tangency point with $\mathcal{I}_{R}$. At this point, the curvature of $\gamma$ vanishes.
\par
The boundary $\partial \minvset{CH}$ is formed by local arcs joining points of $\mathcal{Z}(PQ)$ (with the same multiplicity of $P$ and $Q$) and some points of the tangency locus. By  Proposition~\ref{prop:localarc},  there arcs are strictly convex and real-analytic. 
\end{proof}

\subsubsection{$\Re(\lambda)>0$}

As we mentioned above, we do not have a general description of $\minvset{CH}$, but only a number of interesting  examples. Observe that in this case $\infty$ is a sink of $R(z)\partial_z$).
\par
A qualitative description of the convex hull $Conv(\minvset{CH})$ is  the best that we can obtain with our current knowledge.

\begin{proposition}
Consider a linear differential operator $T$  given by \eqref{eq:1stN} with $\deg Q - \deg P =1$. The boundary $\partial Conv(\minvset{CH})$ of the convex hull $Conv(\minvset{CH})$ of the minimal set is formed by:
\begin{itemize}
    \item finitely many straight segments;
    \item finitely many portions of integral curves of vector field $R(z)\partial_{z}$.
\end{itemize}
In particular, the latter are strictly convex and belong to local arcs of $\partial \minvset{CH}$. In particular,  $\partial Conv(\minvset{CH})$ is piecewise-analytic.
\end{proposition}

\begin{proof}
We denote by $\mathcal{S}$ the set of points where the boundary $\partial Conv(\minvset{CH})$ is strictly convex. They also belong to $\partial \minvset{CH}$ (these points belong to the support of the hull). It follows from Theorem~\ref{thm:MAINClassification} that outside finitely many points, $\mathcal{S}$ is formed by either local or global arcs of $\partial \minvset{CH}$. If such a point $z$ belongs to a global arc, then the line containing the associated ray $r(z)$ is a support line of $Conv(\minvset{CH})$ at $z$ and every point of $\Delta(z)$. It follows that $[z,\Delta^{max}(z)]$ is a straight segment contained in $\partial Conv(\minvset{CH})$. Consequently any arc of $\mathcal{S}$ has to be a portion of local arc.
\par
We know that $\partial Conv(\minvset{CH})$ is formed by straight segments and portions of local arcs. It remains to prove that there are finitely many of them. We consider an arc $\alpha$ of $\partial Conv(\minvset{CH})$ contained in a local arc $\gamma$ of $\partial \minvset{CH}$. The endpoint of $\alpha$ (with the orientation defined by $R(z)\partial_{z}$) has at the same time to be  the endpoint of $\gamma$ (since otherwise the associated rays starting at points of $\alpha$ would intersect $\minvset{CH}$). Therefore, the endpoint of every such arc $\alpha$ in $\partial Conv(\minvset{CH})$ belongs to $\mathcal{Z}(PQ) \cup \mathfrak{I}_{R}$ (see Proposition~\ref{prop:localEND}). Since there are finitely many such points in $\mathcal{S}$,  there are finitely many such arcs in $\partial Conv(\minvset{CH})$.
\par
If the boundary of the convex hull is not formed by finitely many straight segments and portions of integral curves, then there are infinitely many corner points of angle smaller than $\pi$ between the pairs of consecutive straight segments of the boundary. It follows from Corollary~\ref{cor:CONVEXCORNER} that these points belong to $\mathcal{Z}(PQ)$. Therefore, we have finitely many corner points and finitely many straight segments.
\end{proof}

In the examples below (including a very interesting family of operators in which $Q(z)$ has simple roots and $P(z)=Q'(z)$), $Conv(\minvset{CH})$ is a polygon.

\begin{proposition}\label{prop:QP1convexRESIDUE}
Consider a linear differential operator $T$ given by \eqref{eq:1stN}, such that every root $\alpha$ of $Q(z)$ is simple and satisfies  $P(\alpha) \neq 0$ and $\phi_{\alpha}=0$.
\par
Then, $Conv(\minvset{CH})$ coincides with the convex hull of $\mathcal{Z}(Q)$.
\end{proposition}

\begin{proof} The argument is similar to the one used in the proof of the classical Gauss--Lucas theorem (see \cite{MR0225972}). If the differential form $\frac{P(z)dz}{Q(z)}$ has all positive residues, then the roots of $P(z)$ are contained in the convex hull of $\mathcal{Z}(Q)$.
\par
The proof is based on consideration of the electrostatic force $F$ created by the system of point charges placed at the poles of $\frac{P(z)dz}{Q(z)}$ where the value of each charge equals the residue at the corresponding pole. This electrostatic force $F$ equals the conjugate of $\frac{P(z)dz}{Q(z)}$ and one can show that if we take any line $L$ not intersecting the convex hull of $\mathcal{Z}(Q)$ then at any point $p \in L$, $F$ points inside the half-plane of $\bC \setminus L$ not containing $\mathcal{Z}(Q)$. Now recall that the associated ray has the same direction as the conjugate of $P/Q$. Thus, the associated ray $r(p)$ does not intersect the convex hull of $\mathcal{Z}(Q)$. 
\end{proof}

\subsubsection{The first family of examples}

Consider a family of operators of the form $T_{\lambda}=Q(z)\frac{d}{dz}+P(z)$ where $Q(z)=\lambda(z-1)^{k}z$ and $P(z)=(z-1)^{k}$ for some principal coefficient $\lambda \in \mathbb{C}^{\ast}$ and some degree $k \in \mathbb{N}^{\ast}$.
\par
Integral curves of the vector field $R(z)\partial_{z}$ are logarithmic spirals parametrized by $\gamma(t)=\gamma(0)e^{\lambda t}$. In particular, they are concentric circles for $Re(\lambda)=0$.
\par
Depending on the value of $\lambda$, the shape of the minimal set $\minvset{CH}$ can change drastically. Namely, 
\begin{itemize}
    \item if $Re(\lambda)<0$, then $\minvset{CH}=\mathbb{C}$ (see Theorem~1.11 of \cite{AHN+24});
    \item if $Re(\lambda)=0$, then $\minvset{CH}$ is the closed unit disk (see Theorem~\ref{thm:periodic});
    \item if $Re(\lambda)>0$ and $Im(\lambda)=0$, then $\minvset{CH}$ is the segment $[0,1]$.
\end{itemize}
When $Re(\lambda)>0$ and $Im(\lambda) \neq 0$, $\minvset{CH}$ has a more complicated shape we describe below in terms of local and global arcs. Up to conjugation, we will assume that $\Im(\lambda)>0$. 

\begin{proposition}\label{prop:lambdaMinimal}
If $\lambda$ satisfies $\Re(\lambda),\Im(\lambda)>0$, then the minimal set $\minvset{CH}$ of operator $T_{\lambda}$ is bounded by the following arcs:
\begin{itemize}
\item local arc $\gamma$ where $\gamma(t)=e^{-\lambda t}$ and $t \in ]0,t_{0}[$;
\item global arc $\alpha$ where $\alpha(t)=\frac{1}{1+\lambda t}$ and $t \in ]0,t_{1}[$.
\end{itemize}
These two arcs intersect at $1$ and the point $\gamma(t_{0})=\alpha(t_{1})$ of extruding type characterized as the first intersection point between $\alpha$ and $\gamma$ defined on $\mathbb{R}_{>0}$.
\end{proposition}

\begin{proof}
The backward trajectory of the vector field $R(z)\partial_{z}$ starting at $1$ is parametrized by $\gamma(t)=e^{-\lambda t}$ and $t \in [0,\infty)$. Proposition~\ref{prop:IntegralCurve} shows that this arc is entirely contained in $\minvset{CH}$.
\par
Points $z$ for which the associated ray contains $1$ are characterized by the condition $\frac{1-z}{\lambda z} \in \mathbb{R}_{>0}$. They form an arc parametrized by $\alpha(t)=\frac{1}{1+\lambda t}$ for $t \in [0,+\infty[$. This arc is also contained in $\minvset{CH}$.
\par
Since $R(z)=\lambda z$, it is geometrically clear that these two arcs bound $\minvset{CH}$. The boundary $\partial \minvset{CH}$ is formed by a portion of each of them with two singular points at $1$ (when $t=0$) and the first intersection point in the parametrization. There are different ways to see that such an intersection occurs. One of them is to note that $\lim_{t\to \infty}\alpha(t)=0$ and $\lim_{t\to \infty}\arg(\alpha'(t))=\lim_{t\to \infty}\arg(\frac{-\lambda }{(1+\lambda t)^2})=\arg(\frac{-1}{\lambda})$ exists. Since we have $\Re(\lambda),\Im(\lambda)>0$, it follows that the asymptotically straight arc $\alpha(t)$ and the logarithmic spiral $\gamma(t)$ intersect infinitely many times. The endpoint distinct from $1$ common to $\alpha$ and $\gamma$ is the first intersection point between the two parametrized arcs defined on $\mathbb{R}_{>0}$.
\end{proof}

\begin{figure}[!ht]
    \centering
    \includegraphics[width=0.75\textwidth]{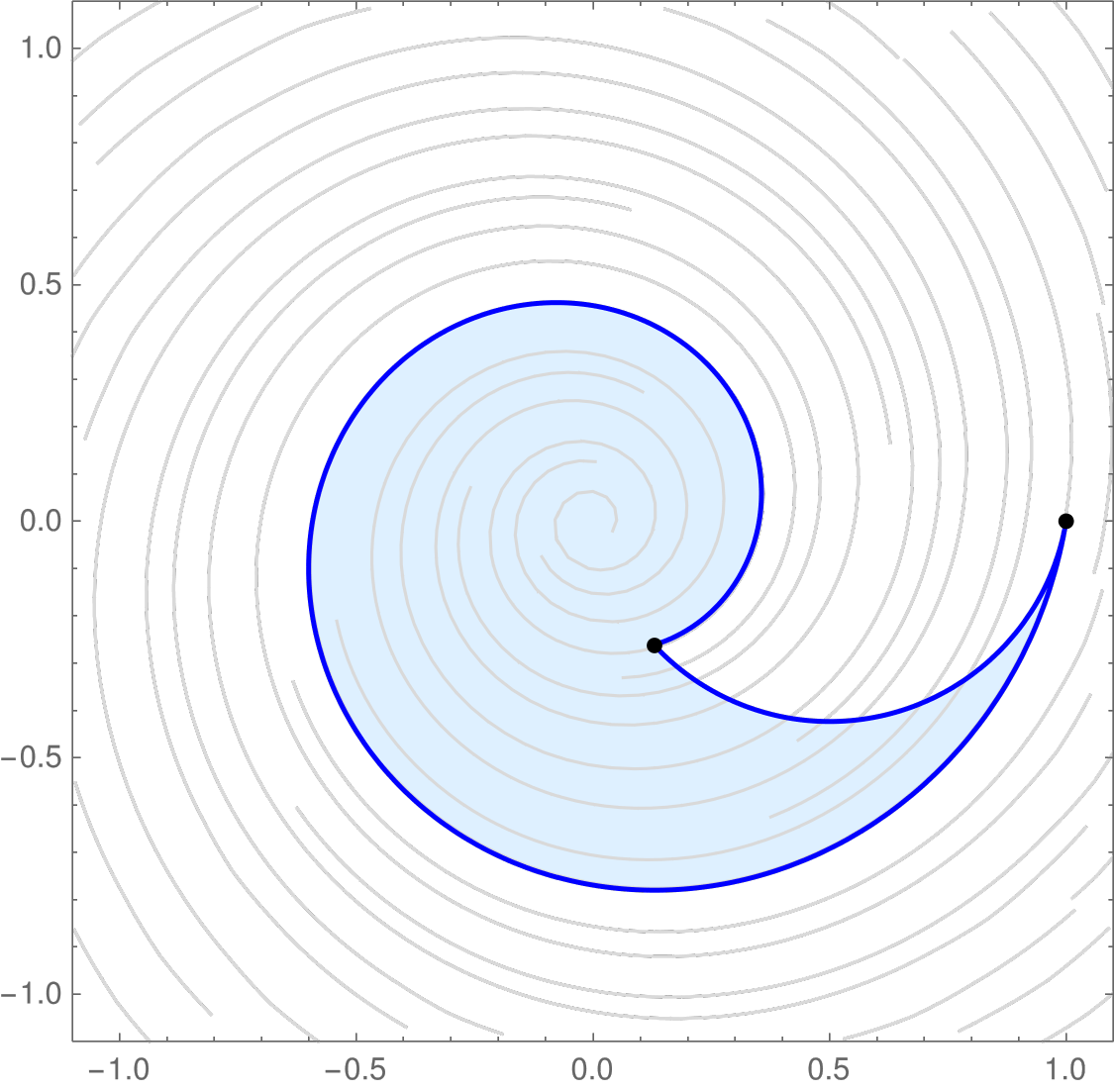}
    \caption{Illustration of the boundary of the minimal set when  $\lambda=1+6i$ in Prop.~\ref{prop:lambdaMinimal}.}
    \label{fig:minimal2}
\end{figure}

\subsubsection{A second family of examples}

Consider the family $T = z(z^{k}-1)\frac{d}{dz} + (z^{k}+1)$,
where $k$ is a positive integer. We are going to prove that for any $k$,  the minimal set $\minvset{CH}$ is the unit disk.

\begin{lemma}\label{lem:unitDiskRays}
Set $f(z) = z + t \frac{z(z^{k}-1)}{z^{k}+1}$, with $t>0$.
Then $|f(z)|>1$ whenever $|z|>1$.
\end{lemma}
\begin{proof}
We substitute $z = r^{\frac{1}{k}} e^{i \theta}$ with $r>1$.
After  some algebraic manipulations, we find that
\begin{equation}
\frac{|f(z)|^2}{|z|^2} = 
\frac{ |f( r^{\frac{1}{k}} e^{i \theta})|^2 }{ r^{2/k} } =  
 1
 +
 t\frac{2 r^2-2+r^2 t-2 \cos(\theta k) r t+t}{r^2+ 2\cos(\theta k) r+1}. 
\end{equation}
Setting $c \coloneqq \cos k \theta$ and rewriting further, we get
\begin{equation}\label{eq:absValueSimplified}
\frac{ |f( r^{\frac{1}{k}} e^{i \theta})|^2 }{ r^{2/k} } =  
 1
 +
 t\frac{2 (r^2-1)+ t \left( (r-c)^2 + (1-c^2) \right )}{(r+c)^2 + (1-c^2)}.
\end{equation}
Since $-1 \leq c \leq 1$, it follows that
\[
 \frac{ |f( r^{\frac{1}{k}} e^{i \theta})|^2 }{ r^{2/k} } 
 > 1 + t \frac{r^2-1}{(r+1)^2} >1.
\]
Consequently, $|f(z)| > |z|$ whenever $|z|>1$ and the statement follows.
\end{proof}

\begin{lemma}\label{lem:unitDiskBdd}
The separatrices of the vector field $R(z)\partial_z=\frac{z(z^k-1)}{z^k+1}\partial_z$
are the arcs of the unit circle, connecting roots of $P(z)$ with roots of $Q(z)$.
\end{lemma}

\begin{proof}
Assuming that $z$ is not a root of $Q$, we have that 
\[
\int \frac{z^k+1}{z(z^k-1)} dz = k^{-1} \log\left( \frac{(1-z^k)^2}{z^k} \right).
\]
Now for $z = e^{i \theta}$, we find that
\[
\Im \log( (1-z^k)^2/z^k ) = \arg( (1-z^k)^2/z^k ) = \arg(-2+ e^{i k \theta}+ e^{-i k \theta})=\pi.
\]
As the integral trajectories are level curves of $\Im\int \frac{dz}{R(z)}$ away from zeros or poles of $R(z)$, it follows that the unit circle consists of the integral trajectories of $R(z) {\partial_z}$. Since the roots of $P$ lie on the unit circle and the zeros of $Q$ on the unit circle have positive residues, it follows that these integral trajectories must be separatrices that are contained in $\minvset{CH}$.
\end{proof}

\begin{corollary}
For $T = z(z^k-1)\partial_{z} + (z^k+1)$, the minimal set $\minvset{CH}$ coincides with the unit disk.
\end{corollary}

\begin{proof}
By Lemma~\ref{lem:unitDiskRays}, we have that all the associated rays for points lying  outside the unit disk never intersect the unit disk. Therefore, $\minvset{CH}$ is contained in the unit disk. Since the unit circle consists of separatrices of $-R(z) {\partial_z}$ (see Lemma~\ref{lem:unitDiskBdd}), it follows that $\minvset{CH}$ contains the unit circle. The associated ray of any point (distinct from $0$) of the open unit disk intersects the unit circle so $\minvset{CH}$ coincides with the unit disk.
\end{proof}

\end{document}